\documentclass[a4paper]{amsart}

\usepackage{geometry}
\usepackage{verbatim}
\usepackage{enumitem}

\usepackage{amssymb}
\usepackage[mathscr]{euscript}
\usepackage{MnSymbol}
\usepackage{todonotes}
\usepackage{stmaryrd}
\usepackage{tikz-cd}
\usepackage{nicefrac}
\usepackage{spectralsequences}
\tikzcdset{arrow style=math font}

\usepackage[bbgreekl]{mathbbol}
\DeclareSymbolFontAlphabet{\mathbb}{AMSb}
\DeclareSymbolFontAlphabet{\mathbbl}{bbold}

\DeclareMathOperator{\id}{id}
\DeclareMathOperator{\Map}{Map}
\DeclareMathOperator{\Hom}{Hom}
\DeclareMathOperator{\map}{map}
\DeclareMathOperator{\Nm}{Nm}
\DeclareMathOperator{\K}{K}
\DeclareMathOperator{\THH}{THH}
\DeclareMathOperator{\TC}{TC}

\DeclareMathOperator{\TR}{TR}
\DeclareMathOperator{\HH}{HH}

\DeclareMathOperator{\gr}{gr}

\DeclareMathOperator{\cyc}{cyc}

\DeclareMathOperator{\Fr}{Fr}

\DeclareMathOperator{\colim}{colim}
\DeclareMathOperator{\Fun}{Fun}
\DeclareMathOperator{\perf}{perf}

\DeclareMathOperator{\can}{can}

\DeclareMathOperator{\op}{op}

\DeclareMathOperator{\nerve}{N}
\DeclareMathOperator{\gen}{gen}
\DeclareMathOperator{\Fin}{Fin}
\DeclareMathOperator{\Orb}{Orb}
\DeclareMathOperator{\Mack}{Mack}

\DeclareMathOperator{\Fix}{Fix}

\DeclareMathOperator{\epi}{epi}
\DeclareMathOperator{\Span}{Span}

\DeclareMathOperator{\Free}{Free}
\DeclareMathOperator{\Gr}{Gr}
\DeclareMathOperator{\ds}{ds}

\newcommand{\Sp}{\mathrm{Sp}}
\newcommand{\GrSp}{\mathrm{GrSp}}
\newcommand{\Spaces}{\mathcal{S}}
\newcommand{\Alg}{\mathrm{Alg}}
\newcommand{\coAlg}{\mathrm{coAlg}}
\newcommand{\CAlg}{\mathrm{CAlg}}
\newcommand{\Mon}{\mathrm{Mon}}
\newcommand{\Grp}{\mathrm{Grp}}

\newcommand{\Mod}{\mathrm{Mod}}

\newcommand{\Cat}{\cat{C}\mathrm{at}}
\newcommand{\CycSp}{\mathrm{CycSp}}
\newcommand{\GrCycSp}{\mathrm{GrCycSp}}

\newcommand{\TCart}{\mathrm{TCart}}

\newcommand{\OrbT}{\mathrm{Orb}_{\T}}
\newcommand{\epicyclic}{\tilde{\Lambda}}
\newcommand{\paracyclic}{\Lambda_{\infty}}

\renewcommand{\S}{\mathbf{S}}
\newcommand{\T}{\mathbf{T}}
\newcommand{\N}{\mathbf{N}}
\newcommand{\Z}{\mathbf{Z}}

\newcommand{\E}{\mathbf{E}}
\newcommand{\W}{\mathbf{W}}

\newcommand{\cat}[1]{\mathscr{#1}}
\newcommand{\B}{\mathrm{B}}

\usepackage{hyperref}

\newtheorem{theorem}{Theorem}[subsection]
\newtheorem{proposition}[theorem]{Proposition}
\newtheorem{lemma}[theorem]{Lemma}
\newtheorem{corollary}[theorem]{Corollary}

\newtheorem{theoremvar}{Theorem}
\newenvironment{theorem*}[1]{\renewcommand{\thetheoremvar}{#1}\theoremvar}{\endtheoremvar}

\theoremstyle{definition}
\newtheorem{definition}[theorem]{Definition}

\newtheorem{remark}[theorem]{Remark}
\newtheorem{example}[theorem]{Example}

\newtheorem{notation}[theorem]{Notation}
\newtheorem{warning}[theorem]{Warning}
\newtheorem{construction}[theorem]{Construction}

\title{On curves in K-theory and TR}
\author{Jonas McCandless}

\address{
Fachbereich Mathematik und Informatik\\
Westfälische Wilhelms--Universität Münster\\
Germany}
\email{\href{mailto:mccandless.jonas@gmail.com}{mccandless.jonas@gmail.com}}
\urladdr{\url{https://www.sites.google.com/view/jonasmccandless/}}

\begin{document}

\begin{abstract}
We prove that $\TR$ is corepresentable by the reduced topological Hochschild homology of the flat affine line $\S[t]$ as a functor defined on the $\infty$-category of cyclotomic spectra with values in the $\infty$-category of spectra with Frobenius lifts, refining a result of Blumberg--Mandell. We define the notion of an integral topological Cartier module using Barwick's formalism of spectral Mackey functors on orbital $\infty$-categories, extending the work of Antieau--Nikolaus in the $p$-typical setting. As an application, we show that $\TR$ evaluated on a connective $\E_1$-ring admits a description in terms of the spectrum of curves on algebraic $\K$-theory generalizing the work of Hesselholt and Betley--Schlichtkrull.
\end{abstract}

\maketitle

\setcounter{tocdepth}{2}
\tableofcontents

\section{Introduction} \label{section:introduction}
Topological cyclic homology provides an effective tool for studying the invariant determined by algebraic $\K$-theory by means of the cyclotomic trace. The classical construction of topological cyclic homology using equivariant stable homotopy theory given by Bökstedt--Hsiang--Madsen~\cite{BHM93} is facilitated by another invariant called $\TR$ together with the additional structure of an operator referred to as the Frobenius. In~\cite{NS18}, Nikolaus--Scholze demonstrate that the construction of topological cyclic homology admits a drastic simplification bypassing the use of equivariant stable homotopy theory, and Bhatt--Morrow--Scholze~\cite{BMS19} use the foundational work of Nikolaus--Scholze to construct motivic filtrations on topological Hochschild homology and its variants identifying the graded pieces with completed prismatic cohomology in the sense of Bhatt--Scholze~\cite{BS19}. The invariant given by $\TR$ carries important arithmetic information by itself. The calculations of the $p$-adic $\K$-theory of local fields obtained by Hesselholt--Madsen~\cite{HM03,HM04} and Geisser--Hesselholt~\cite{GH06} rely on the relationship between $\TR$ and the de Rham--Witt complex with log poles. In particular, Hesselholt--Madsen~\cite{HM03} confirmed the Lichtenbaum--Quillen conjecture for $p$-adic fields prior to the work of Rost--Voevodsky on the Bloch--Kato conjecture~\cite{Voe03,Voe11}. These calculations are in turn based on the previous work of Hesselholt~\cite{Hes96} and Hesselholt--Madsen~\cite{HM97}. 

\subsection{Statement of results}
If $I \subseteq A$ is a two-sided ideal of a ring $A$, then the relative algebraic $\K$-theory spectrum $\K(A, I)$ is defined as the fiber of the map of spectra $\K(A) \to \K(A/I)$ induced by the canonical ring homomorphism $A \to A/I$. The starting point of this paper is the following result of Hesselholt~\cite[Theorem 3.1.10]{Hes96}, which asserts that if $R$ is a commutative $\Z/p^j$-algebra for some integer $j \geq 1$, then there is a natural equivalence of $p$-complete spectra
\[
	\TR(R) \simeq \varprojlim \Omega \K(R[t]/t^n, (t)).
\]
The inverse limit appearing on the right hand side of the equivalence is referred to as the spectrum of curves on the algebraic $\K$-theory of $R$, and was defined by Hesselholt~\cite{Hes96} based on previous work of Bloch~\cite{Blo77} on the relationship between crystalline cohomology and algebraic $\K$-theory. Betley--Schlichtkrull~\cite[Theorem 1.3]{BS05} extend the result above to associative rings after profinite completion, where $\TR$ is replaced by $\TC$. In this case, the inverse limit defining the spectrum of curves on $\K$-theory is replaced by a limit over a diagram  which additionally encodes the transfer maps $R[t]/t^m \to R[t]/t^{mn}$ determined by $t \mapsto t^n$. We recall the following notation:

\begin{notation} \label{notation:truncated_over_sphere}
In the following, we will let $\S[t]$ denote the $\E_\infty$-ring defined by $\S[t] = \Sigma^\infty_+ \Z_{\geq 0}$. Note that $\S[t]$ is not the free $\E_\infty$-ring on one generator, but the underlying $\E_1$-ring of $\S[t]$ is the free $\E_1$-ring on one generator since $\Z_{\geq 0}$ is the free $\E_1$-monoid on one generator in the $\infty$-category of spaces. Let $\S[t]/t^n$ denote the $\E_\infty$-ring defined by the following pushout of $\E_\infty$-rings
\[
\begin{tikzcd}
	\S[t] \arrow{r}{t \mapsto t^n} \arrow{d}{t \mapsto 0} & \S[t] \arrow{d} \\
	\S \arrow{r} & \S[t]/t^n
\end{tikzcd}
\]
If $R$ is a connective $\E_1$-ring, then we define the $\E_1$-ring $R[t]/t^n$ by $R[t]/t^n = R \otimes \S[t]/t^n$, and we have that $\pi_\ast(R[t]/t^n) \simeq (\pi_\ast R)[t]/t^n$. We obtain a map of connective $\E_1$-rings $R[t]/t^n \to R$ such that the kernel of the induced ring homomorphism
\[
	\pi_0(R[t]/t^n) \simeq (\pi_0 R)[t]/t^n \to \pi_0 R
\]
is given by the nilpotent ideal $(t)$. If $E : \Alg^{\mathrm{cn}} \to \Sp$ is a functor, then we will let $E(R[t]/t^n, (t))$ denote the fiber of the induced map of spectra $E(R[t]/t^n) \to E(R)$. In the following, we will be interested in situation where $E = \K$ or $E = \TC$.
\end{notation}

\begin{remark}
If $R$ is a discrete $\mathrm{H}\Z$-algebra, then $R \otimes \S[t]/t^n$ is discrete with $\pi_0(R \otimes \S[t]/t^n) \simeq R[t]/t^n$. In general, if $R$ is a connective $\E_1$-ring, then $R \otimes \S[t]/t^n$ is not necessarily discrete. 
\end{remark}

As an application of the formalism developed in this paper, we obtain a common generalization of the results of Hesselholt and Betley--Schlichtkrull discussed above (see Corollary~\ref{corollary:TR_curves}). 

\begin{theorem*}{A} \label{theorem:introduction_main_result}
If $R$ is a connective $\E_1$-ring, then there is a natural equivalence of spectra
\[
	\TR(R) \simeq \varprojlim \Omega \K(R[t]/t^n, (t)). 
\]
\end{theorem*}

Note that in the setting of Theorem~\ref{theorem:introduction_main_result}, there is an equivalence of spectra
\[
	\varprojlim \Omega \K(R[t]/t^n, (t)) \simeq \varprojlim \Omega \TC(R[t]/t^n, (t)) 
\]
by virtue of the Dundas--Goodwillie--McCarthy theorem~\cite{BGM12}. The results of Betley--Schlichtkrull and Hesselholt discussed above rely on this equivalence combined with an analysis of the fixed points of $\THH(R[t]/t^n)$ by finite cyclic groups using~\cite{HM97}. In fact, the Frobenius endomorphism of $\TR$ can be expressed in terms of certain transfer maps on the spectrum of curves on $\K$-theory under the equivalence of Theorem~\ref{theorem:introduction_main_result} (cf. Remark~\ref{remark:frobenius_TR_in_terms_of_curves}). The proof of Theorem~\ref{theorem:introduction_main_result} proceeds by completely different methods which we summarize below. 

\subsection{Methods}
The main technical contribution of this paper is a construction of $\TR$ which bypasses the use of equivariant stable homotopy theory. This is analogous to the construction of topological cyclic homology recently given by Nikolaus--Scholze~\cite{NS18} under suitable boundedness assumptions. We will begin by briefly reviewing the classical construction of $\TR$ following~\cite{HM97,BM16}, and refer the reader to \textsection\ref{subsection:genuine_cycsp} for a comprehensive account of the classical construction.

\begin{construction}
A genuine cyclotomic spectrum is a genuine $\T$-spectrum $X$ with respect to the family of finite cyclic subgroups $C_k$ of $\T$ together with compatible equivalences $X^{\Phi C_k} \simeq X$ for every $k \geq 1$, where $X^{\Phi C_k}$ denotes the geometric fixedpoints for the action of $C_k$ on $X$. For instance, if $R$ is a connective $\E_1$-ring, then $\THH(R)$ admits the structure of a genuine cyclotomic spectrum~\cite{BHM93,HM97}. If $(m, n)$ is a pair of positive integers with $m = ln$, then the restriction map $R : X^{C_m} \to X^{C_n}$ is the map of genuine $\T$-spectra defined by
\[
	X^{C_m} \simeq (X^{C_l})^{C_n} \to (X^{\Phi C_l})^{C_n} \simeq X^{C_n},
\]
where the final equivalence is induced by the genuine cyclotomic structure of $X$. We have that 
\[
	\TR(X) = \mathrm{lim}_{n, R} X^{C_n},
\]
where the limit is indexed over the set of positive integers regarded as a poset under divisibility. The collection of genuine cyclotomic spectra assemble into an $\infty$-category $\CycSp^{\gen}$ (see Definition~\ref{definition:cycsp_gen}), and there is a canonical functor of $\infty$-categories $\CycSp^{\gen} \to \CycSp$ which restricts to an equivalence on the full subcategories of those objects whose underlying spectrum is bounded below~\cite{NS18}. Consequently, if $X$ is a cyclotomic spectrum whose underlying spectrum is bounded below, then we may evaluate $\TR : \CycSp^{\gen} \to \Sp$ on $X$ using this equivalence.
\end{construction}

Throughout this paper, we will be interested in the reduced topological Hochschild homology of the $\E_\infty$-ring $\S[t]$, whose construction we briefly recall. The map of $\E_\infty$-rings $\S[t] \to \S$ given by $t\mapsto 0$ induces a map of cyclotomic spectra $\THH(\S[t]) \to \S$, where the sphere spectrum $\S$ is equipped with the trivial cyclotomic structure. The reduced topological Hochschild homology of the $\E_\infty$-ring $\S[t]$ is the cyclotomic spectrum defined by
\[
	\widetilde{\THH}(\S[t]) = \mathrm{fib}(\THH(\S[t]) \to \S).
\]
The cyclotomic structure of the reduced topologogical Hochschild homology of $\S[t]$ admits a more refined structure, namely the structure of a spectrum with Frobenius lifts\footnote{This is referred to as a cyclotomic spectrum with Frobenius lifts in~\cite{NS18,AN20,KN19} (cf. Warning~\ref{warning:cycspfr_terminology})}. Informally, the datum of a spectrum with Frobenius lifts is given by a spectrum $X$ with an action of $\T$ together with a collection of compatible $\T$-equivariant maps of spectra
\[
	\psi_k : X \to X^{hC_k}
\]
for every integer $k \geq 1$. Note that for every prime $p$, the $\T$-equivariant map of spectra
\[
	X \xrightarrow{\psi_p} X^{hC_p} \xrightarrow{\can} X^{tC_p}
\]
endows the underlying spectrum with $\T$-action of $X$ with the structure of a cyclotomic spectrum. The notion of a $p$-typical spectrum with Frobenius lift was introduced by Nikolaus--Scholze~\cite{NS18} and Antieau--Nikolaus~\cite{AN20}. In the $p$-typical situation we only require the existence of the map $\psi_p$ in sharp contrast to the integral situation, where we also require coherences between these maps. Returning to $\widetilde{\THH}(\S[t])$, there is an equivalence of cyclotomic spectra
\[
	\widetilde{\THH}(\S[t]) \simeq \bigoplus_{i \geq 1} \Sigma^\infty_+(S^1/C_i),
\]
where $S^1/C_i$ denotes the quotient space whose $\T$-action is given by $x \mapsto \lambda x$ for every element $\lambda$ of $\T$, and the cyclotomic structure of the right hand side of the equivalence above is induced by the following $\T$-equivariant map of spaces
\[
	S^1/C_i \xrightarrow{x \mapsto \sqrt[p]{x}} (S^1/C_{pi})^{hC_p} \xrightarrow{\can} (S^1/C_{pi})^{tC_p}
\]
for every prime $p$ (cf. Example~\ref{example:epibarconstruction_N}). Consequently, the cyclotomic Frobenius map of $\widetilde{\THH}(\S[t])$ is canonically equivalent to the following $\T$-equivariant composite map of spectra
\[
	\widetilde{\THH}(\S[t]) \xrightarrow{\psi_p} \widetilde{\THH}(\S[t])^{hC_p} \xrightarrow{\can} \widetilde{\THH}(\S[t])^{tC_p}
\]
for every prime $p$. A substantial amount of this paper is devoted to making the discussion above precise. To that end, we introduce the following notions:

\begin{itemize}[leftmargin=2em, topsep=7pt, itemsep=7pt]
	\item The Witt monoid is the $\E_1$-monoid defined by the semidirect product $\W = \T \rtimes \N^{\times}$, where the multiplicative monoid $\N^{\times}$ acts on the circle $\T$ by covering maps of positive degree (see Construction~\ref{construction:Wittmonoid}). Let $\B\W$ denote the $\infty$-category with one object and $\W$ as endomorphisms. We will define the notion of a spectrum with Frobenius lifts as a spectrum with a right action of the Witt monoid $\W$ which in turn is the datum of a functor of $\infty$-categories $\B\W^{\op} \to \Sp$. This precisely encodes the datum of a spectrum $X$ with an action of the circle $\T$ together with compatible $\T$-equivariant maps $\psi_k : X \to X^{hC_k}$ for every integer $k \geq 1$. The collection of spectra with Frobenius lifts assemble into an $\infty$-category $\Sp^{\Fr}$ and we prove that the reduced topological Hochschild homology of the $\E_\infty$-ring $\S[t]$ refines to an object of this $\infty$-category. Furthermore, there is a canonical functor of $\infty$-categories $\Sp^{\Fr} \to \CycSp$ which is described informally above. The insight that the Witt monoid parametrizes Frobenius lifts is present in the unpublished work of Goodwillie~\cite{Good90} on the cyclotomic trace. 

	\item We prove that $\B\W$ is an orbital $\infty$-category in the sense of~\cite{BDGNS16}, and we define the notion of an integral topological Cartier module as a spectral Mackey functor on $\B\W$ in the sense of Barwick~\cite{Bar17}. This precisely encodes the datum of a spectrum $M$ with an action of the circle $\T$ together with compatible $\T$-equivariant factorizations
	\[
		M_{hC_k} \to M \to M^{hC_k}
	\]
	of the norm map for the cyclic group $C_k$ for every integer $k \geq 1$. This extends the definition of a $p$-typical topological Cartier module given by Antieau--Nikolaus~\cite{AN20} to the integral situation. The collection of topological Cartier modules assemble into an $\infty$-category $\TCart$, and we will prove that $\TR$ refines to a functor with values in this $\infty$-category.    
\end{itemize}

In general, there is a forgetful functor of $\infty$-categories $\TCart \to \Sp^{\Fr}$, and it follows that we may regard $\TR(X)$ as a spectrum with Frobenius lifts for every cyclotomic spectrum $X$, where each Frobenius lift of $\TR(X)$ is given by the $\T$-equivariant map of spectra
\[
	\TR(X) \simeq \TR(X)^{C_k} \to \TR(X)^{hC_k}.
\]
Finally, the mapping spectrum $\map_{\CycSp}(\widetilde{\THH}(\S[t]), X)$ refines to a spectrum with Frobenius lifts using that the reduced topological Hochschild homology of $\S[t]$ admits the structure of a bimodule over $\S[\W]$ by left and right multiplication (see Example~\ref{example:witt_monoid_right_multiplication}). With these notions in place, we state the following result (see Theorem~\ref{theorem:comparison_TR} and Proposition~\ref{proposition:TR_right_adjoint}).

\begin{theorem*}{B} \label{theorem:introduction_corepresentability_TR}
For every cyclotomic spectrum $X$ whose underlying spectrum is bounded below, there is a natural equivalence of spectra with Frobenius lifts
\[
	\TR(X) \simeq \map_{\CycSp}(\widetilde{\THH}(\S[t]), X).
\]
Furthermore, there is an adjunction of $\infty$-categories
\[
\begin{tikzcd}
	\Sp^{\Fr} \arrow[yshift=0.7ex]{r} & \CycSp \arrow[yshift=-0.7ex]{l}{\TR}
\end{tikzcd}
\]	
where the left adjoint is given by the canonical functor $\Sp^{\Fr} \to \CycSp$. 
\end{theorem*}

We remark that a variant of Theorem~\ref{theorem:introduction_corepresentability_TR} above was previously obtained by Blumberg--Mandell~\cite{BM16} using point-set models for genuine cyclotomic spectra. More precisely, Blumberg--Mandell show that the functor $\TR : \CycSp^{\gen} \to \Sp$ is corepresentable by $\widetilde{\THH}(\S[t])$. Theorem~\ref{theorem:introduction_corepresentability_TR} above asserts that $\TR$ is additionally corepresentable as a functor defined on the $\infty$-category of cyclotomic spectra with values in the $\infty$-category of spectra with Frobenius lifts. The adjunction above was previously established by Krause--Nikolaus~\cite{KN19} for $p$-typical $\TR$. 

\begin{remark}
Theorem~\ref{theorem:introduction_corepresentability_TR} can be regarded as an analogue of the result of Nikolaus--Scholze~\cite{NS18} which asserts that there is an equivalence of spectra
\[
	\TC(X) \simeq \map_{\CycSp}(\S, X)
\]
for every cyclotomic spectrum $X$ whose underlying spectrum is bounded below. This corepresentability result for topological cyclic homology was conjectured by Kaledin~\cite{Kal10} and established by Blumberg--Mandell~\cite{BM16} after $p$-completion prior to the work of Nikolaus--Scholze. 
\end{remark}

\begin{remark}
In this paper, the $\infty$-category of topological Cartier modules will mostly  function as a convenient formalism for proving Theorem~\ref{theorem:introduction_corepresentability_TR}. However, we believe that the notion of a topological Cartier module is important in its own right. For instance, the notion of a topological Cartier module formalizes additional structure present on the rational Witt vectors and cyclic $\K$-theory. An extensive treatment of the $\infty$-category of topological Cartier modules extending the result of Antieau--Nikolaus~\cite{AN20} to the integral situation will appear in forthcoming work. 
\end{remark} 

We prove Theorem~\ref{theorem:introduction_corepresentability_TR} by showing that the functors given by the constructions $X \mapsto \TR(X)$ and $X \mapsto \map_{\CycSp}(\widetilde{\THH}(\S[t]), X)$ both determine right adjoints of the canonical functor
\[
	\Sp^{\Fr} \to \CycSp
\]
when restricted to the full subcategories of those objects whose underlying spectrum is bounded below. This relies on a genuine version of the Tate orbit lemma~\cite[Lemma I.2.1]{NS18} established by Antieau--Nikolaus~\cite{AN20} together with an explicit description of the free topological Cartier module on a spectrum with Frobenius lifts. The latter we deduce from a general version of the Segal--tom Dieck splitting for spectral Mackey functors on orbital $\infty$-categories (see Proposition~\ref{proposition:tomDieck_orbital}). We deduce Theorem~\ref{theorem:introduction_main_result} from the celebrated Dundas--Goodwillie--McCarthy theorem~\cite{BGM12} and the following result (see Theorem~\ref{theorem:TR_curves_TC}). 

\begin{theorem*}{C} \label{theorem:introduction_TR_TC_loops}
There is a natural equivalence of spectra
\[
	\TR(X) \simeq \varprojlim \Omega \TC(X \otimes \widetilde{\THH}(\S[t]/t^n))
\]
for every cyclotomic spectrum $X$ whose underlying spectrum is bounded below. 
\end{theorem*}

The crucial observation is that there is an equivalence of spectra with $\T$-action
\[
	\map_{\Sp}(\widetilde{\THH}(\S[t]), \S) \simeq \prod_{n \geq 1} \Sigma^{\infty-1}_+(S^1/C_n),
\]
where the mapping spectrum on the left carries the residual $\T$-action. We deduce Theorem~\ref{theorem:introduction_TR_TC_loops} from Theorem~\ref{theorem:introduction_corepresentability_TR} by exploiting the equalizer formula for the mapping spectrum in the $\infty$-category of cyclotomic spectra obtained by Nikolaus--Scholze~\cite{NS18} together with a careful analysis of the cyclotomic structure of $\varprojlim \widetilde{\THH}(\S[t]/t^n)$. This observation is due to Achim Krause. 

\subsubsection*{Acknowledgements}
It is a pleasure to thank Stefano Ariotta, Elden Elmanto, Felix Janssen, Markus Land, Malte Leip, and Jay Shah for very useful discussions related to this work. The author would also like to thank Benjamin Antieau, Elden Elmanto, Lars Hesselholt, Liam Keenan, Markus Land, and Thomas Nikolaus for very valuable comments on a previous version. The author is extremely grateful to two anonymous referees who wrote extremely detailed reports pointing out many inaccuracies and mistakes, which led to many improvements. Most importantly, the author is extremely grateful to Achim Krause and Thomas Nikolaus for their constant support and many insights which made this project possible. Funded by the Deutsche Forschungsgemeinschaft (DFG, German Research Foundation) under Germany's Excellence Strategy EXC 2044–390685587, Mathematics Münster: Dynamics–Geometry–Structure and the CRC 1442 Geometry: Deformations and Rigidity. 

\section{Spectra with Frobenius lifts and TR} \label{section:cycspfr_TR}
The main goal of this section is to construct the reduced topological Hochschild homology of the $\E_\infty$-ring $\S[t]$ as a spectrum with Frobenius lifts, and provide an alternative construction of $\TR$ bypassing the otherwise instrumental use of equivariant stable homotopy theory. In \textsection\ref{subsection:epicyclic_category}, we define the notion of an object with Frobenius lifts in any $\infty$-category, and prove that the geometric realization of a presheaf on the epicyclic category with values in an $\infty$-category which admits geometric realizations refines to an object with Frobenius lifts. In \textsection\ref{subsection:spaces_fr}, we will discuss the $\infty$-category of spaces with Frobenius lifts, and construct a refinement of the cyclic bar construction which automatically carries Frobenius lifts. We use this to produce examples of spaces with Frobenius lifts which will play an essential role throughout this exposition. In \textsection\ref{subsection:cycsp_fr}, we discuss the $\infty$-category of spectra with Frobenius lifts, and construct the canonical functor from the $\infty$-category of spectra with Frobenius lifts to the $\infty$-category of cyclotomic spectra. Finally, in \textsection\ref{subsection:TR} we provide an alternative construction of $\TR$ and establish descent properties of this construction, deferring the comparison with the classical construction of $\TR$ to \textsection\ref{section:genuine}. 

\subsection{The epicyclic category} \label{subsection:epicyclic_category}
We introduce the $\infty$-category of objects with Frobenius lifts in an $\infty$-category, and describe these in terms of presheaves on the epicyclic category. The following construction will play an important role throughout this paper. 

\begin{construction} \label{construction:Wittmonoid}
We define a monoid which we will refer to as the Witt monoid following~\cite{AMGR17b}. 

\begin{enumerate}[leftmargin=2em, topsep=7pt, itemsep=7pt]
	\item The multiplicative monoid $\N^{\times}$ of positive integers acts on the circle group $\T$ by the assignment $x \mapsto x^k$ for every integer $k \geq 1$, and the Witt monoid is the topological monoid defined by $\W = \T \rtimes \N^{\times}$, where $\N^{\times}$ is regarded as a discrete topological monoid. The underlying space of the Witt monoid is given by $\T \times \N^{\times}$, and the multiplication is given by
	\[
		(\lambda, k) \cdot (\mu, l) = (\lambda \mu^k, kl)
	\]
	for every pair of elements $(\lambda, k)$ and $(\mu, l)$ of the underlying space of $\W$. It will be convenient to regard the Witt monoid as an object of the $\infty$-category $\Alg(\Spaces)$ of $\E_1$-monoids. We will let $\B\W$ denote the $\infty$-category with one object and the Witt monoid $\W$ as endomorphisms. 

	\item For every prime $p$, the multiplicative monoid $\mu_{p^\N} = \{1, p, p^2, \ldots\}$ acts on the circle group $\T$ by restriction along the canonical inclusion of multiplicative monoids $\mu_{p^\N} \hookrightarrow \N^{\times}$, and the $p$-typical Witt monoid is the topological monoid defined by $\W_{p^{\N}} = \T \rtimes \mu_{p^\N}$, where $\mu_{p^\N}$ is regarded as a discrete topological monoid. We will let $\B\W_{p^\N}$ denote the $\infty$-category with one object and the $p$-typical Witt monoid $\W_{p^\N}$ as endomorphisms.  
\end{enumerate}
\end{construction}

The Witt monoid also appears in the unpublished work of Goodwillie~\cite{Good90} on the cyclotomic trace, where the insight that it parametrizes Frobenius lifts was already present. The following definition also appears in the work of Ayala--Mazel-Gee--Rozenblyum~\cite{AMGR17a,AMGR17b,AMGR17c}. 

\begin{definition} \label{definition:frobenius_lifts}
The $\infty$-category of objects with Frobenius lifts in an $\infty$-category $\cat{C}$ is defined by
\[
	\cat{C}^{\Fr} = \Fun(\B\W^{\op}, \cat{C}) = \cat{P}_{\cat{C}}(\B\W).
\]
\end{definition}

\begin{remark} \label{remark:unwind_cfr}
There is a fiber sequence of $\infty$-categories
\[
	\B\T \simeq (\B\T)^{\op} \to (\B\W)^{\op} \to (\B\N^{\times})^{\op},
\]
and the functor $\B\W \to \B\N^{\times}$ is a cocartesian fibration. There is a functor of $\infty$-categories
\[
	\cat{C}^{\Fr} \to \cat{C}^{\B\T}
\]
which regards an object with Frobenius lifts as an object with $\T$-action. If $X$ is an object with Frobenius lifts, then there is a morphism $\psi_k : X \to X$ in $\cat{C}$ determined by the action of $(1, k)$ on the underlying object of $X$ for every integer $k \geq 1$. We have that
\[
	(\lambda^k, 1) \cdot (1, k) = (1, k) \cdot (\lambda, 1)
\]
for every element $\lambda$ of $\T$, which precisely encodes that $\psi_k$ is $\T$-equivariant with respect to the degree $k$ action on the source. We conclude that an object with Frobenius lifts in an $\infty$-category $\cat{C}$ is given by an object $X$ of $\cat{C}$ with $\T$-action together with compatible $\T$-equivariant maps
\[
	\psi_k : X \to X^{hC_k},
\]
for every positive integer $k$, where the target carries the residual $\T/C_k \simeq \T$-action. For instance, if $(k, l)$ is a pair of positive integers, then the following diagram commutes
\[
\begin{tikzcd}
	X \arrow{rr}{\psi_k} \arrow{d}{\psi_l} & & X^{hC_k} \arrow{d}{\psi_l^{hC_k}} \\ 
	X^{hC_l} \arrow{r}{\psi_k^{hC_l}} & (X^{hC_k})^{hC_l} \arrow{r}{\simeq} & (X^{hC_l})^{hC_k}
\end{tikzcd}
\]
\end{remark}

In \textsection\ref{subsection:spaces_fr} and \textsection\ref{subsection:cycsp_fr} we discuss the $\infty$-category of spaces with Frobenius lifts and the $\infty$-category of spectra with Frobenius lifts in more detail. Presently, we discuss Goodwillie's epicyclic category which provides a formalism for constructing objects with Frobenius lifts in the sense of Definition~\ref{definition:frobenius_lifts}. We begin by reviewing the cyclic category following~\cite[Appendix B]{NS18}.

\begin{definition}
A parasimplex is a nonempty linearly ordered set $I$ equipped with an action of the group of integers $\Z$ denoted $+ : I \times \Z \to I$, such that the following conditions are satisfied:
\begin{enumerate}[topsep=5pt, itemsep=3pt]
	\item For every pair of elements $\lambda$ and $\lambda'$ of $I$, the set $\{\mu \in I \mid \lambda \leq \mu \leq \lambda'\}$ is finite.
	\item If $\lambda$ is an element of $I$, then $\lambda < \lambda + 1$.
\end{enumerate}
A paracyclic morphism is a morphism of sets $f : I \to J$ which is nondecreasing and $\Z$-equivariant. The paracyclic category $\paracyclic$ is defined as the category whose objects are given by the parasimplices and whose morphisms are given by the paracyclic morphisms.
\end{definition}

\begin{example}
For every positive integer $n$, the set $\frac{1}{n}\Z = \{\frac{m}{n} \mid m \in \Z\}$ can be regarded as a parasimplex with respect to its usual ordering and the action of $\Z$ given by addition. We will denote this particular parasimplex by $[n]_{\Lambda}$. Conversely, if $I$ is a parasimplex, then there is an isomorphism of parasimplices $I \simeq [n]_{\Lambda}$ for a uniquely determined positive integer $n$ which, can be characterized as the cardinality of the set $\{x \in I \mid y \leq x < y+1\}$ for any element $y$ of $I$. 
\end{example}

\begin{example}
If $S$ is an object of the simplex category $\Delta$, then the product $S \times \Z$ can be regarded as a parasimplex with respect to the lexicographic ordering and the action of $\Z$ given by the formula $m + (s, n) = (s, m+n)$. The construction $S \mapsto S \times \Z$ defines a faithful functor $\Delta \to \paracyclic$ which is essentially surjective since there is an isomorphism of parasimplices $[n] \times \Z \simeq [n+1]_{\Lambda}$ for every integer $n \geq 0$. 
\end{example}

The cyclic category is defined as follows:

\begin{definition} \label{definition:cyclic_category}
The cyclic category $\Lambda$ is the category whose objects are parasimplices, where the set of morphisms between a pair of parasimplices $I$ and $J$ is defined by
\[
	\Hom_{\Lambda}(I, J) = \Hom_{\paracyclic}(I, J)/\Z,
\]
where $\Z$ acts on the set of paracyclic morphisms $\Hom_{\paracyclic}(I, J)$ by the formula $(f+n)(\lambda) = f(\lambda) + n$. 
\end{definition}

\begin{remark} \label{remark:strict_BZ_action_paracyclic}
The action of the group $\Z$ on the set of paracyclic morphisms $\Hom_{\paracyclic}(I, J)$ in the definition of the cyclic category determines a strict action of the simplicial abelian group $\B\Z$ on the simplicial set $\nerve(\paracyclic)$, which in turn induces an action of the circle $\T$ on $\nerve(\paracyclic)$. Note that $\nerve(\Lambda) \simeq \nerve(\paracyclic)_{\B\Z} \simeq \nerve(\paracyclic)_{h\T}$ since the action of the simplicial abelian group $\B\Z$ on $\nerve(\paracyclic)$ is free. Consequently, there is a fiber sequence
\[
	\nerve(\paracyclic) \to \nerve(\Lambda) \to \B\T,
\]
where the map $\nerve(\Lambda) \to \B\T$ is both a cartesian and cocartesian fibration. 
\end{remark}

We will now define the epicyclic category. The construction $I \mapsto I/\Z$ determines a functor
\[
	\Lambda \to\cat{C}\mathrm{at},
\]
where the parasimplex $I$ is regarded as a poset, and $I/\Z$ denotes the quotient in $\Cat$. We regard $\Cat$ as a $1$-category. Let $[n]_{\epicyclic}$ denote the category defined by $[n]_{\epicyclic} = [n]_{\Lambda}/\Z$ for every $n \geq 1$.

\begin{definition} \label{definition:epicyclic_category}
The epicyclic category $\epicyclic$ is the subcategory of $\Cat$ on those categories which are isomorphic to $[n]_{\epicyclic}$ for some integer $n \geq 1$, and those functors which are essentially surjective. 
\end{definition}

The epicyclic category was introduced by Goodwillie in an unpublished letter to Waldhausen, and Burghelea--Fiedorowicz--Gajda~\cite{BFG94} give a combinatorial description of the epicyclic category similar to the combinatorial description of the cyclic category~\cite{Con83}. Kaledin~\cite{Kal13} studies a variant of the epicyclic category which is referred to as the cyclotomic category. Ayala--Mazel-Gee--Rozenblyum~\cite{AMGR17b} describe the epicyclic category in terms of stratified $1$-manifolds and disk refinements. The definition of the epicyclic category given in Definition~\ref{definition:epicyclic_category} is due to Nikolaus. Note that the functor $\Lambda \to \epicyclic$ is essentially surjective and faithful by construction. To see that this functor is not full, we give a geometric description of the morphisms in $\Lambda$ and $\epicyclic$ in the following examples. We may regard the object $[n]_{\Lambda}$ of $\Lambda$ as a cyclic graph with $n$ vertices, or equivalently as a configuration of $n$ marked points on $S^1$.

\begin{example} \label{example:morphisms_in_cyclic}
The datum of a morphism $[m]_{\Lambda} \to [n]_{\Lambda}$ in the cyclic category $\Lambda$ is equivalent to the datum of a self-map of $S^1$ of degree $1$ which preserves markings and satisfies that the induced self-map on universal covers is non-decreasing. For example, the morphism $\tau_n : [n]_{\Lambda} \to [n]_{\Lambda}$ defined by $\tau_n(i) = i+1$ corresponds to the self-map of $S^1$ given by rotation with $2\pi/n$, where the configuration of $S^1$ is given by $\{1, \zeta, \ldots, \zeta^{n-1}\}$ for a primitive $n$th root of unity $\zeta$. Note that the morphism $\tau_n$ is an automorphism of $[n]_{\Lambda}$, and there is a group isomorphism
\[
	\Z/(n+1)\Z \to \mathrm{Aut}_{\Lambda}([n]_{\Lambda})
\]
given by sending the generator of the cyclic group $\Z/(n+1)\Z$ to the automorphism $\tau_n$. In conclusion, the group of automorphisms of $[n]_{\Lambda}$ in $\Lambda$ is cyclic of order $n+1$ generated by $\tau_n$. In contrast, the group of automorphisms of an object of the simplex category is trivial. 
\end{example}

\begin{example} \label{example:morphisms_in_epicyclic}
The datum of a morphism $[m]_{\epicyclic} \to [n]_{\epicyclic}$ in the epicyclic category $\epicyclic$ is equivalent to the datum of a self-map of $S^1$ of positive degree which preserves markings and satisfies that the induced self-map on universal covers is non-decreasing. If $\alpha$ is a morphism in $\epicyclic$, then the degree of $\alpha$ is defined as the degree of the induced self-map of $S^1$. There is a map
\[
	\mathrm{deg} : \nerve(\epicyclic) \to \B\N^{\times}
\]
given by sending a morphism of $\epicyclic$ to its degree, and this is a cartesian fibration which classifies the canonical action of the multiplicative monoid $\N^{\times}$ on the cyclic category $\Lambda$. As an example, the map $p_k : S^1 \to S^1$ given by $p_k(x) = x^k$ corresponds to a morphism $\alpha_k : [kn]_{\epicyclic} \to [n]_{\epicyclic}$ in $\epicyclic$ of degree $k$, where the source of $p_k$ carries the configuration consisting of $kn$ preimages of the $n$ marked points on the target. The morphism $\alpha_k$ does not exist in the cyclic category, which means that the functor $\Lambda \to \epicyclic$ is not full. Furthermore, for every morphism $f : [n]_{\epicyclic} \to [n]_{\epicyclic}$ in $\epicyclic$ of degree $k$, there exists a morphism $f' : [n]_{\epicyclic} \to [kn]_{\epicyclic}$ in $\epicyclic$ of degree $1$ such that $f = \alpha_k \circ f'$.
\end{example}

The epicyclic category $\epicyclic$ admits an alternative description in terms of a right action of the Witt monoid on the paracyclic category which we now address. 

\begin{construction} \label{construction:action_witt_paracyclic}
In the following, we will regard the Witt monoid as the monoid $\B\Z \rtimes \N^\times$ in $\Cat$. The action of the group $\Z$ on the set $\mathrm{Hom}_{\paracyclic}(I, J)$ given by addition as in Definition~\ref{definition:cyclic_category} defines an action of $\B\Z$ on the paracyclic category $\paracyclic$. The action of $n \in \N^{\times}$ on $\paracyclic$ is given by the assignment $I \mapsto nI$, where $nI$ denotes the parasimplex with the same underlying linearly ordered set as $I$ but whose $\Z$-action is given by restricting the $\Z$-action on $I$ along the homomorphism $\Z \to \Z$ given by multiplication by $n$. This defines a right action of $\B\Z \rtimes \N^\times$ on the paracyclic category $\paracyclic$, so we obtain a right action of the Witt monoid $\W$ on $\nerve(\paracyclic)$. 
\end{construction}

In the following, we will employ the formalism of lax limits and colimits in the sense of Gepner--Haugseng--Nikolaus~\cite{GHN15}. Recall that if $F : \B M^{\op} \to \Cat_\infty$ denotes a functor which defines a right action of an $\E_1$-monoid $M$ on an $\infty$-category $\cat{C}$, then the lax coinvariants is the cartesian fibration $p : \cat{C}_{\ell M} \to \B M$ classified by $F$. The lax fixed points $\cat{C}^{\ell M}$ is the $\infty$-category of sections of $p$ which carry every arrow in $\B M$ to a $p$-cartesian edge. 

\begin{lemma} \label{lemma:laxorbits_witt_paracyclic}
There is an equivalence of $\infty$-categories\footnote{Note that $\nerve(\paracyclic)_{\ell \W}$ is a priori a $2$-category but the result asserts that it is in fact equivalent to a $1$-category.}
\[
	\nerve(\paracyclic)_{\ell \W} \simeq \nerve(\epicyclic),
\]
where the Witt monoid acts on the paracyclic category $\paracyclic$ as in Construction~\ref{construction:action_witt_paracyclic}.
\end{lemma}

\begin{proof}
There is an equivalence of $\infty$-categories
\[
	\nerve(\paracyclic)_{\ell \W} \simeq (\nerve(\paracyclic)_{\ell \T})_{\ell \N^{\times}} \simeq (\nerve(\paracyclic)_{h\T})_{\ell \N^\times},
\]
using that $\B\W \simeq (\B\T)_{\ell \N^\times}$, and that $\nerve(\paracyclic)_{\ell \T} \simeq \nerve(\paracyclic)_{h\T}$ since $\T$ is a group. It follows from Remark~\ref{remark:strict_BZ_action_paracyclic} that $\nerve(\paracyclic)_{h\T} \simeq \nerve(\Lambda)$, so we conclude that there is an equivalence
\[
	\nerve(\paracyclic)_{\ell \W} \simeq \nerve(\Lambda)_{\ell \N^\times}
\]
which shows that $\nerve(\paracyclic)_{\ell \W}$ is  equivalent to a $1$-category. Recall from Example~\ref{example:morphisms_in_epicyclic}, that the canonical action of $\N^\times$ on the cyclic category $\nerve(\Lambda)$ is classified by the cartesian fibration
\[
	\mathrm{deg} : \nerve(\epicyclic) \to \B\N^\times
\]
which shows that $\nerve(\Lambda)_{\ell \N^\times} \simeq \nerve(\epicyclic)$. This shows the desired statement. 
\end{proof}

Recall that the geometric realization of a cyclic object admits a canonical action of the circle group as observed by Connes~\cite{Con83}. Similarly, we introduce the formalism of epicyclic objects and show that the geometric realization of an epicyclic object canonically admits Frobenius lifts, that is a right action of the Witt monoid. 

\begin{definition}
The $\infty$-category of epicyclic objects in an $\infty$-category $\cat{C}$ is defined by
\[
	\cat{P}_{\cat{C}}(\epicyclic) = \Fun(\nerve(\epicyclic^{\op}), \cat{C}).
\]
The $\infty$-category of epicyclic objects in the $\infty$-category of spaces is denoted by $\cat{P}(\epicyclic)$. 
\end{definition}

Assume that $X$ is an epicyclic object in an $\infty$-category $\cat{C}$ which admits geometric realizations. The geometric realization of $X$ is defined as the colimit of the underlying simplicial object
\[
	\nerve(\Delta^{\op}) \to \nerve(\paracyclic^{\op}) \to \nerve(\Lambda^{\op}) \to \nerve(\epicyclic^{\op}) \to \cat{C}. 
\]
Let $|X|$ denote the geometric realization of $X$, and note that there is a functor of $\infty$-categories
\[
	\cat{P}_{\cat{C}}(\epicyclic) \to \cat{C}
\]
determined by the construction $X \mapsto |X|$. Burghelea--Fiedorowicz--Gajda~\cite{BFG94} prove that the geometric realization of an epicyclic object admits the structure of an object with Frobenius lifts using an explicit combinatorial description of the epicyclic category. We present a proof of this result using the language employed in this paper. 

\begin{proposition} \label{proposition:geometric_realization_epicyclic_frobenius_lifts}
If $\cat{C}$ is an $\infty$-category which admits geometric realizations, then the construction $X \mapsto |X|$ admits a canonical refinement to a functor of $\infty$-categories
\[
	\cat{P}_{\cat{C}}(\epicyclic) \to \cat{C}^{\Fr}
\]
such that the following diagram commutes
\[
\begin{tikzcd}
	\cat{P}_{\cat{C}}(\tilde{\Lambda}) \arrow{r} \arrow{d} & \cat{C}^{\Fr} \arrow{d} \\
	\cat{P}_{\cat{C}}(\Lambda) \arrow{r} & \cat{C}^{\B\T}
\end{tikzcd}
\]
\end{proposition}

\begin{proof}
The functor $\nerve(\Delta^{\op}) \to \nerve(\paracyclic^{\op})$ is cofinal by virtue of~\cite[Theorem B.3]{NS18}, hence the geometric realization of an epicyclic object of $\cat{C}$ is equivalent to the colimit of the paracyclic object
\[
	\nerve(\paracyclic^{\op}) \to \nerve(\Lambda^{\op}) \to \nerve(\epicyclic^{\op}) \to \cat{C}.
\]
There is an adjunction of $\infty$-categories
\[
\begin{tikzcd}
	\cat{P}_{\cat{C}}(\paracyclic) \arrow[yshift=0.7ex]{r}{|-|} & \cat{C} \arrow[yshift=-0.7ex]{l}{\delta^{\ast}}
\end{tikzcd}
\]
where the right adjoint $\delta^{\ast}$ is given by precomposition with the terminal functor $\delta : \nerve(\paracyclic^{\op}) \to \Delta^0$ which is equivariant with respect to the left $\W$-action on $\nerve(\paracyclic^{\op})$ defined in Construction~\ref{construction:action_witt_paracyclic}, and the trivial left $\W$-action on $\Delta^0$. This in turn means that the right adjoint $\delta^{\ast}$ is equivariant with respect to the induced right action of $\W$ on $\cat{P}_{\cat{C}}(\paracyclic)$ and the trivial right action of $\W$ on $\cat{C}$. Consequently, the left adjoint canonically refines to a functor of $\infty$-categories
\[
	|-| : \cat{P}_{\cat{C}}(\paracyclic)^{\ell \W} \to \cat{C}^{\Fr},
\]
where $(-)^{\ell \W}$ denotes the lax fixed points of the right action of the Witt monoid on $\cat{P}_{\cat{C}}(\paracyclic)$. Additionally, we have used that $\cat{C}^{\Fr} \simeq \cat{C}^{\ell \W}$ since the action of the Witt monoid on $\cat{C}$ is trivial. The desired functor is given by composing the functor above with the canonical functor 
\[
	\cat{P}_{\cat{C}}(\epicyclic) \to \cat{P}_{\cat{C}}(\paracyclic)^{\ell \W},
\]
where we have used that $\nerve(\paracyclic)_{\ell \W} \simeq \nerve(\epicyclic)$ by Lemma~\ref{lemma:laxorbits_witt_paracyclic}. The final assertion follows by a similar argument using the equivalences $(-)_{\ell \T} \simeq (-)_{h\T}$ and $(-)^{\ell \T} \simeq (-)^{h\T}$ since $\T$ is a group.  
\end{proof}

\subsection{Spaces with Frobenius lifts} \label{subsection:spaces_fr}
Using the formalism of epicyclic spaces discussed in \textsection\ref{subsection:epicyclic_category}, we give a more comprehensive treatment of the $\infty$-category of spaces with Frobenius lifts. We construct an epicyclic variant of the cyclic bar construction and use this to produce examples of spaces with Frobenius lifts which will play an important role throughout this paper. Recall that the $\infty$-category of spaces with Frobenius is defined by $\Spaces^{\Fr} = \cat{P}(\B\W)$. We summarize the salient features of the $\infty$-category of spaces with Frobenius lifts. 

\begin{proposition}
The $\infty$-category of spaces with Frobenius lifts is a presentable $\infty$-category, and the forgetful functor $\Spaces^{\Fr} \to \Spaces^{\B\T}$ is conservative and preserves small limits and colimits. 
\end{proposition}

Alternatively, the $\infty$-category of spaces with Frobenius lifts is equivalent to the full subcategory of $\Spaces_{/\B\W}$ spanned by those maps over $\B\W$ which are right fibrations~\cite[Section 2.2]{Lur09}. In practice, we will make use of the formalism of epicyclic spaces as discussed in \textsection\ref{subsection:epicyclic_category} to construct examples of spaces with Frobenius lifts.

\begin{example} \label{example:free_loop_space}
Let $X$ denote an object of the $\infty$-category of spaces. The free loop space $\mathrm{L}(X)$ of $X$ is the geometric realization of the epicyclic space given by the assignment
\[
	[n]_{\epicyclic} \mapsto \Map_{\Spaces}(|[n]_{\epicyclic}|, X) \simeq \mathrm{L}(X),
\]
so we may regard the free loop space $\mathrm{L}(X)$ as a space with Frobenius lifts by Proposition~\ref{proposition:geometric_realization_epicyclic_frobenius_lifts}. For concreteness, we note that the self-map of $\mathrm{L}(X)$ determined by the cyclic operator $\tau_n$ is induced by the self-map of $S^1$ given by rotation with $2\pi/n$. The self-map of $\mathrm{L}(X)$ determined by the epicyclic operator $\alpha_k$ is induced by the self-map of $S^1$ given by the $k$th power map. 
\end{example} 

An important class of examples of spaces with Frobenius lifts arises by evaluating the cyclic bar construction on an $\E_1$-monoid. Using the combinatorial description of the epicyclic category, Burghelea--Fiederowicz--Gajda~\cite{BFG94} prove that the cyclic space whose geometric realization defines the cyclic bar construction further refines to an epicyclic space. A similar result was obtained by Schlichtkrull~\cite{Sch09}. More recently, Nikolaus--Scholze~\cite{NS18} construct the individual Frobenius lifts on the cyclic bar construction using the space-valued diagonal. For instance, if $M$ is the $\E_1$-monoid given by $\Omega X$ for a connected pointed space $X$, then $\B^{\cyc} M \simeq \mathrm{L}(X)$. In fact, this equivalence refines to an equivalence of spaces with Frobenius lifts, where we regard $\mathrm{L}(X)$ as a space with Frobenius lifts as in Example~\ref{example:free_loop_space} (see Corollary~\ref{corollary:epi_THH_loop_space}). We proceed by introducing an epicyclic bar construction following an idea of Nikolaus. 

\begin{construction} \label{construction:epicyclic_bar_construction}
Let $\mathrm{T}_{\mathrm{Ass}}$ denote the Lawvere theory of monoids~\cite{Law63}. By definition, the opposite of $\mathrm{T}_{\mathrm{Ass}}$ is the full subcategory of the category of monoids spanned by those monoids which are free on a finite set. We construct a functor $j : \epicyclic \to \mathrm{T}_{\mathrm{Ass}}^{\op}$ which informally is given by sending an object $[n]_{\epicyclic}$ of the epicyclic category to the free monoid on the set $\{1, \ldots, n\}$. First note that there is a functor of $\infty$-categories
\[
	\nerve(\epicyclic) \to \Cat_{\infty}^{\Delta^1} \xrightarrow{\mathrm{cofib}} (\Cat_{\infty})_{\ast}
\]
determined by the construction which sends an object $[n]_{\epicyclic}$ of the epicyclic category $\epicyclic$ to the cofiber of the canonical functor $[n]_{\epicyclic}^{\mathrm{ds}} \hookrightarrow [n]_{\epicyclic}$ which exhibits the source as the discrete category on the set of objects of $[n]_{\epicyclic}$. The space of endomorphisms of the canonical basepoint of the cofiber $[n]_{\epicyclic}/[n]_{\epicyclic}^{\mathrm{ds}}$ is equivalent to a discrete monoid which is free on the finite set $\{1, \ldots, n\}$, so we obtain a functor
\[
	j : \nerve(\epicyclic) \to \nerve(\mathrm{T}_{\mathrm{Ass}}^{\op})
\]
defined by $j([n]_{\epicyclic}) = \mathrm{End}_{[n]_{\epicyclic}/[n]_{\epicyclic}^{\mathrm{ds}}}(\ast)$, where $\ast$ denotes the canonical basepoint of $[n]_{\epicyclic}/[n]_{\epicyclic}^{\mathrm{ds}}$. 
\end{construction}

The datum of an $\E_1$-monoid in the sense of~\cite{Lur17} is equivalently specified by the datum of a functor $\nerve(\mathrm{T}_{\mathrm{Ass}}) \to \Spaces$ which preserves finite products as explained in~\cite{Cra10} or~\cite[Appendix B]{GGN16}. We define the epicyclic bar construction of an $\E_1$-monoid using Construction~\ref{construction:epicyclic_bar_construction}. 

\begin{definition} \label{definition:epicyclic_bar_construction}
The epicyclic bar construction $\B^{\epi} M$ of an $\E_1$-monoid $M$ is defined by forming the geometric realization of the epicyclic space given by
\[
	\nerve(\epicyclic^{\op}) \xrightarrow{j^{\op}} \nerve(\mathrm{T}_{\mathrm{Ass}}) \xrightarrow{M} \Spaces.
\]
The construction $M \mapsto \B^{\epi} M$ determines a functor of $\infty$-categories
\[
	\B^{\epi} : \Alg(\Spaces) \to \Spaces^{\Fr}.
\]
\end{definition} 

\begin{remark}
If $M$ is an $\E_1$-monoid, then the geometric realization of the cyclic space
\[
	\nerve(\Lambda^{\op}) \to \nerve(\epicyclic^{\op}) \xrightarrow{j^{\op}} \nerve(\mathrm{T}_{\mathrm{Ass}}) \xrightarrow{M} \Spaces
\]
is canonically equivalent to the cyclic bar construction $\B^{\cyc} M$ regarded as an object of $\Spaces^{\B\T}$.
\end{remark}

We introduce an epicyclic variant of topological Hochschild homology of a small $\infty$-category as a refinement of the unstable topological Hochschild homology introduced by Nikolaus~\cite{HS19}.

\begin{definition} \label{definition:epicyclic_THH}
The epicyclic topological Hochschild homology $\THH^{\epi}(\cat{C})$ of a small $\infty$-category $\cat{C}$ is defined as the geometric realization of the epicyclic space given by the assignment
\[
	[n]_{\epicyclic} \mapsto \Fun([n]_{\epicyclic}, \cat{C})^{\simeq}. 
\]
The construction $\cat{C} \mapsto \THH^{\epi}(\cat{C})$ determines a functor of $\infty$-categories
\[
	\THH^{\epi} : \Cat_{\infty} \to \Spaces^{\Fr}.
\]
\end{definition}

It will be convenient to work with both the epicyclic bar construction and the epicyclic topological Hochschild homology, so we show that they coincide as spaces with Frobenius lifts for $\E_1$-monoids. The author learned the proof of the following result from Nikolaus, and a proof of a similar result appears in Krause--Nikolaus~\cite[Proposition 8.1]{KN19}.

\begin{proposition} \label{proposition:epi_bar_conincide_epi_THH}
Let $M$ be an $\E_1$-monoid, and let $\B M$ denote the $\infty$-category with one object and $M$ as endomorphisms. There is an equivalence of spaces with Frobenius lifts
\[
	\B^{\epi} M \to \THH^{\epi}(\B M).
\]
\end{proposition}

\begin{proof}
Let $C_n$ denote the cofiber of the functor $([n]_{\epicyclic}^{\mathrm{ds}})_+ \hookrightarrow ([n]_{\epicyclic})_+$ in $(\Cat_\infty)_{\ast}$, and note that $C_n$ is equivalent to the pointed $\infty$-category $\B j([n]_{\epicyclic})$, where $j([n]_{\epicyclic})$ denotes the monoid discussed in Construction~\ref{construction:epicyclic_bar_construction}. There is a fiber sequence of epicyclic spaces
\[
	\Fun_{\ast}(C_n, \B M)^\simeq \to \Fun([n]_{\epicyclic}, \B M)^{\simeq} \to (\B M^{\times})^n,
\]
where $M^{\times}$ denotes the group of units of $M$, and the epicyclic space on the left is equivalent to the epicyclic space defining $\B^{\epi} M$ since $\Fun_{\ast}(C_n, \B M)^{\simeq} \simeq \Map_{\Spaces}([n]_{\epicyclic}, M)$. As a consequence, we obtain a map of spaces with Frobenius lifts
\[
	\B^{\epi} M \to \THH^{\epi}(\B M)
\]
obtained as the geometric realization of the first map in the fiber sequence above, and we show that this map is an equivalence. Note that the space $\Fun_{\ast}(C_n, \B M)^\simeq$ admits a canonical action of the epicyclic group given by $G_\bullet : [n]_{\epicyclic} \mapsto (M^\times)^n$, and the first map in the fiber sequence exhibits $\Fun([n]_{\epicyclic}, \B M)^\simeq$ as the homotopy coinvariants of this action in the $\infty$-category $\cat{P}(\epicyclic)$. In other words, the map $\B^{\epi} M \to \THH^{\epi}(\B M)$ above is equivalent to the following composite
\[
	\B^{\epi} M \to (\B^{\epi} M)_{h|G_\bullet|} \xrightarrow{\simeq} \THH^{\epi}(\B M),
\]
where $|G_\bullet|$ denotes the geometric realization of the epicyclic group $G_\bullet$. Note that the geometric realization $|G_\bullet|$ of $G_\bullet$ is contractible since the underlying simplicial object of $G_\bullet$ admits an extra degeneracy, so the map $\B^{\epi} M \to (\B^{\epi} M)_{h|G_\bullet|}$ is an equivalence, which proves the claim. 
\end{proof}

Consequently, we obtain the following result which allow us to explicitly identify the Frobenius lifts on the epicyclic bar construction in certain examples. This result was first obtained by Goodwillie~\cite{God85}, Burghelea--Fiedorowicz~\cite{BF96}, and Jones~\cite{Jon87} ignoring the epicyclic structure, and by Burghelea--Fiedorowicz--Gajda~\cite{BFG94} in the epicyclic case.

\begin{corollary} \label{corollary:epi_THH_loop_space}
Let $X$ denote an object of the $\infty$-category of spaces regarded as an $\infty$-category. There is an equivalence of spaces with Frobenius lifts
\[
	\THH^{\epi}(X) \simeq \mathrm{L}(X),
\]
where $\mathrm{L}(X)$ is regarded as a space with Frobenius lifts as in Example~\ref{example:free_loop_space}. 
\end{corollary}

\begin{proof}
For every integer $n \geq 1$, there is an equivalence of spaces
\[
	\Fun([n]_{\epicyclic}, X)^{\simeq} \simeq \Map_{\Spaces}(|[n]_{\epicyclic}|, X) \simeq \mathrm{L}(X),
\]
which shows that the underlying simplicial object of $\THH^{\epi}(X)$ is constant with value $\mathrm{L}(X)$. The assertion now follows from Example~\ref{example:free_loop_space}. 
\end{proof}

Finally, using the material discussed above we present further examples of spaces with Frobenius lifts which will play an important role throughout this paper.


\begin{example} \label{example:epibarconstruction_N}
There is an equivalence of spaces with Frobenius lifts $\B^{\epi} \Z \simeq \mathrm{L}(S^1)$ by virtue of Proposition~\ref{proposition:epi_bar_conincide_epi_THH} and Corollary~\ref{corollary:epi_THH_loop_space}. For every integer $i$, we will let
\[
	\mathrm{L}(S^1)_i = \{f \in \mathrm{L}(S^1) \mid \mathrm{deg}(f) = i\}
\]
denote the space of self-maps of $S^1$ of degree $i$. Note that the map
\[
	S^1/C_i \to \mathrm{L}(S^1)_i
\]
given by $[x] \mapsto (t \mapsto (tx)^i)$ is an equivalence of spaces with $\T$-action provided that we regard the left hand side as a space with $\T$-action given by the formula $\lambda \cdot [x] = [\lambda x]$. The $p$th Frobenius lift of $\mathrm{L}(S^1)$ is induced by the $\T$-equivariant map of spaces 
\[
	\mathrm{L}(S^1)_{i} \to (\mathrm{L}(S^1)_{pi})^{hC_p}
\]
given by the construction $f \mapsto f \circ (-)^p$, where $(\mathrm{L}(S^1)_{pi})^{hC_p}$ carries the residual $\T/C_p \simeq \T$-action. Using the equivalence described above, this map corresponds to the $\T$-equivariant map of spaces
\[
	S^1/C_i \to (S^1/C_{pi})^{hC_p}
\]
given by $x \mapsto \sqrt[p]{x}$, where the residual action of $\T/C_p$ on $(S^1/C_{pi})^{hC_p}$ is given by $[\lambda]\cdot[x] = [\lambda x]$. Under the isomorphism $\T \simeq \T/C_p$ given by the $p$th root, the action of $\T$ on $(S^1/C_{pi})^{hC_p}$ is given by $\lambda \cdot [x] = [\sqrt[p]{\lambda} x]$. In conclusion, there is an equivalence of spaces with Frobenius lifts
\[
	\B^{\epi}\Z \simeq \displaystyle\coprod_{i \in \Z} S^1/C_{|i|},
\]
where the $p$th Frobenius lift is determined by the $\T$-equivariant map $S^1/C_i \to (S^1/C_{pi})^{hC_p}$ given by $x \mapsto \sqrt[p]{x}$. Furthermore, there is an equivalence of spaces with $\T$-action
\[
	\B^{\epi} \Z_{\geq 0} \simeq \ast \amalg \displaystyle\coprod_{i \geq 1} S^1/C_i,
\]
where the Frobenius lifts are described as above. This identification is originally due to Hesselholt~\cite[Lemma 2.2.3]{Hes96}. We will let $\widetilde{\B}^{\epi} \Z_{\geq 0}$ denote the space with Frobenius lifts given by
\[
	\widetilde{\B}^{\epi} \Z_{\geq 0} \simeq \displaystyle\coprod_{i \geq 1} S^1/C_i,
\]
and we refer to $\widetilde{\B}^{\epi} \Z_{\geq 0}$ as the reduced epicyclic bar construction of $\Z_{\geq 0}$.  
\end{example}

\begin{example} \label{example:witt_monoid_right_multiplication}
The Witt monoid admits the structure of $\W$-$\W$-bimodule by left and right multiplication, so we may regard the Witt monoid as a space with Frobenius lifts. Furthermore, there is an equivalence of $\W$-$\W$-bimodules in spaces
\[
	\widetilde{\B}^{\epi}\Z_{\geq 0} \simeq \W
\]
by Example~\ref{example:epibarconstruction_N}. Explicitly, under this equivalence, the action of an element $\lambda$ of $\T$ is given by left multiplication by $(\lambda, 1)$ in the Witt monoid. The $p$th Frobenius lift is given by right multiplication by $(1, p)$ in the Witt monoid. As a consequence, the Yoneda embedding $\B\W \to \Spaces^{\Fr}$ is determined by the construction $\ast \mapsto \widetilde{\B}^{\epi}\Z_{\geq 0}$. 
\end{example}

\subsection{Spectra with Frobenius lifts} \label{subsection:cycsp_fr}
Next, we discuss the $\infty$-category of spectra with Frobenius lifts in further detail, and construct the canonical functor from the $\infty$-category of spectra with Frobenius lifts to the $\infty$-category of cyclotomic spectra as an instance of the stable nerve-realization adjunction. Recall that the $\infty$-category of spectra with Frobenius lifts is defined by $\Sp^{\Fr} = \cat{P}_{\Sp}(\B\W)$. We will start by summarizing the salient features of the $\infty$-category of spectra with Frobenius lifts. 

\begin{proposition}
The $\infty$-category $\Sp^{\Fr}$ of spectra with Frobenius lifts is a presentable and stable $\infty$-category, and the forgetful functor $\Sp^{\Fr} \to \Sp^{\B\T}$ is conservative and preserves small limits and colimits. 
\end{proposition}

\begin{warning} \label{warning:cycspfr_terminology}
The notion of a spectrum with Frobenius lifts in the sense of Definition~\ref{definition:frobenius_lifts} is referred to as a cyclotomic spectrum with Frobenius lifts in the literature (cf.~\cite{NS18,AN20,KN19}). More precisely, Nikolaus--Scholze~\cite{NS18} and Antieau--Nikolaus~\cite{AN20} study the notion of a $p$-typical cyclotomic spectrum with Frobenius lift which we briefly recall in Remark~\ref{remark:ptypical_cycspfr}. In~\cite{KN19}, Krause--Nikolaus study an integral version of cyclotomic spectra with Frobenius lifts. A variant of Definition~\ref{definition:frobenius_lifts} features in the work of Ayala--Mazel-Gee--Rozenblyum~\cite{AMGR17a,AMGR17b,AMGR17c} under the name of an unstable cyclotomic object. We have opted for the present terminology to avoid confusion in later parts of the present exposition, where we will have the chance to consider objects in the $\infty$-category of cyclotomic spectra with an additional action of the Witt monoid (cf. Construction~\ref{construction:canonical_functor}).
\end{warning}

\begin{remark} \label{remark:ptypical_cycspfr}
We recall the $p$-typical variant of the notion of a spectrum with Frobenius lifts as previously considered by Nikolaus--Scholze~\cite{NS18} and Antieau--Nikolaus~\cite{AN20}. A $p$-typical spectrum with Frobenius lift\footnote{Referred to as a $p$-typical cyclotomic spectrum with Frobenius lift by Nikolaus--Scholze and Antieau--Nikolaus.} is a spectrum $X$ with $\T$-action together with a $\T$-equivariant map of spectra
\[
	\psi_p : X \to X^{hC_p},
\]
where the target carries the residual $\T/C_p \simeq \T$-action. The $\infty$-category $\Sp_p^{\Fr}$ of $p$-typical spectra with Frobenius lift is defined by the following pullback of $\infty$-categories
\[
\begin{tikzcd}
	\Sp^{\Fr}_p \arrow{rr} \arrow{d} & & (\Sp^{\B\T})^{\Delta^1} \arrow{d} \\
	\Sp^{\B\T} \arrow{rr}{(\id, (-)^{hC_p})} & & \Sp^{\B\T} \times \Sp^{\B\T}
\end{tikzcd}
\]
The canonical map of Witt monoids $\W_{p^{\N}} \to \W$ induces a canonical functor of $\infty$-categories
\[
	\Sp^{\Fr} \to \cat{P}_{\Sp}(\B\W_{p^{\N}}) \xrightarrow{\simeq} \Sp_p^{\Fr},
\]
where the equivalence is induced by the universal property of $\Sp_p^{\Fr}$ as the pullback above. Informally, the functor above is given by only remembering the $p$th Frobenius lift. 
\end{remark}

In practice, it is difficult to specify the datum of a spectrum with Frobenius lifts due to the infinite hierarchy of coherences that needs to be specified. However, there is a functor
\[
	\Spaces^{\Fr} \to \Sp^{\Fr}
\]
given by postcomposition with $\Sigma^\infty_+ : \Spaces \to \Sp$, and we will mostly be interested in spectra with Frobenius lifts contained in the essential image of this functor. In this situation, the formalism of epicyclic spaces as discussed in \textsection\ref{subsection:epicyclic_category} allows us to effectively control the coherences.

\begin{example}
The suspension spectrum $\Sigma^\infty_+ \widetilde{\B}^{\epi} \Z_{\geq 0}$ admits the structure of a spectrum with Frobenius lifts. There is an equivalence of spectra with Frobenius lifts
\[
	\Sigma^\infty_+ \widetilde{\B}^{\epi} \Z_{\geq 0} \simeq \bigoplus_{i \geq 1} \Sigma^\infty_+(S^1/C_i),
\]
and the $k$th Frobenius lift is induced by the $\T$-equivariant map of spaces $S^1/C_i \to (S^1/C_{ki})^{hC_k}$ as described in Example~\ref{example:epibarconstruction_N}. The coherences are encoded by the epicyclic bar construction as in \textsection\ref{subsection:spaces_fr}. Additionally, there is an equivalence of $\S[\W]$-$\S[\W]$-bimodules in spectra $\Sigma^\infty_+ \widetilde{\B}^{\epi} \Z_{\geq 0} \simeq \S[\W]$ by Example~\ref{example:witt_monoid_right_multiplication}.
\end{example}{}

Informally, if $X$ denotes a spectrum with Frobenius lifts, then the underlying spectrum with $\T$-action admits the structure of a cyclotomic spectrum, where the $p$th cyclotomic Frobenius is given by the $\T$-equivariant composite map of spectra $X \to X^{hC_p} \to X^{tC_p}$ for every prime $p$. We show that this construction refines to a functor of $\infty$-categories
\[
	\Sp^{\Fr} \to \CycSp
\]
as an instance of the stable nerve-realization adjunction. The stable nerve-realization adjunction appears in the work of Dwyer--Kan~\cite{DK84} in the unstable context of $G$-spaces, and we refer the reader to~\cite[Appendix A]{Ari20} for a systematic treatment phrased in the language of $\infty$-categories. 

\begin{definition} \label{definition:stable_realization}
Let $\cat{C}$ be a small $\infty$-category. The stable realization of a functor $F : \cat{C} \to \cat{D}$ with values in a stable and presentable $\infty$-category $\cat{D}$ is the functor
\[
	|-|_F : \cat{P}_{\Sp}(\cat{C}) \to \cat{D}
\]
defined by the left Kan extension of $F$ along the stable Yoneda embedding $y_\mathrm{st} : \cat{C} \to \cat{P}_{\Sp}(\cat{C})$.
\end{definition}

\begin{proposition} \label{proposition:stable_realization_right_adjoint}
Let $\cat{C}$ be a small $\infty$-category, and let $F : \cat{C} \to \cat{D}$ be a functor with values in a stable and presentable $\infty$-category $\cat{D}$. The stable realization of $F$ admits a right adjoint
\[
	\nerve_F : \cat{D} \to \cat{P}_{\Sp}(\cat{C})
\]
determined by the construction $d \mapsto \map_{\cat{D}}(F(-), d)$. 
\end{proposition}

\begin{proof}
Let $\mathrm{L}_y F$ denote the left Kan extension of $F$ along the Yoneda embedding $y : \cat{C} \to \cat{P}(\cat{C})$, and note that it suffices to show that $\mathrm{L}_y F$ admits a right adjoint determined by the construction $d \mapsto \Map_{\cat{D}}(F(-), d)$. It follows from~\cite[Theorem 5.1.5.6]{Lur09} that there is an adjoint equivalence
\[
\begin{tikzcd}
	\Fun^{\mathrm{L}}(\cat{P}(\cat{C}), \cat{D}) \arrow[yshift=0.7ex]{r}{y^{\ast}} & \Fun(\cat{C}, \cat{D}) \arrow[yshift=-0.7ex]{l}
\end{tikzcd}
\]
where the right adjoint is determined by the assignment $F \mapsto \mathrm{L}_y F$, so we conclude that $\mathrm{L}_y F$ admits a right adjoint by~\cite[Corollary 5.5.2.9]{Lur09}. There is a natural equivalence 
\[
	(\mathrm{L}_y F)(X) \simeq \int^{c \in \cat{C}} \Map_{\cat{P}(\cat{C})}(y(c), X) \otimes F(c),
\]
for every presheaf $X$ on $\cat{C}$, where we have used that $\cat{D}$ is canonically tensored over the $\infty$-category of spaces. For every object $d$ of $\cat{D}$, there is a sequence of natural equivalences
\begin{align*}
	\Map_{\cat{D}}((\mathrm{L}_y F)(X), d) &\simeq \Map_{\cat{D}}\left(\int^{c \in \cat{C}} \Map_{\cat{P}(\cat{C})}(y(c), X) \otimes F(c), d\right) \\
	&\simeq \int_{c \in \cat{C}} \Map_{\cat{D}}(X(c) \otimes F(c), d) \\
	&\simeq \int_{c \in \cat{C}} \Map_{\Spaces}(X(c), \Map_{\cat{D}}(F(c), d)) \\
	&\simeq \Map_{\cat{P}(\cat{C})}(X, \Map_{\cat{D}}(F(-), d))
\end{align*}
which finishes the proof. 
\end{proof}

We construct a functor $\B\W \to \CycSp$ which encodes the cyclotomic structure of $\Sigma^\infty_+ \widetilde{\B}^{\epi}\Z_{\geq 0}$, which will allow us to define the canonical functor from the $\infty$-category of spectra with Frobenius lifts to the $\infty$-category of cyclotomic spectra. 

\begin{construction} \label{construction:canonical_functor}
Following Nikolaus--Scholze~\cite{NS18} and Antieau--Nikolaus~\cite{AN20}, the $\infty$-category of $p$-typical spaces with Frobenius lift is defined by the following pullback of $\infty$-categories
\[
\begin{tikzcd}
	\Spaces^{\Fr}_p \arrow{rr} \arrow{d} & & (\Spaces^{\B\T})^{\Delta^1} \arrow{d} \\
	\Spaces^{\B\T} \arrow{rr}{(\id, (-)^{hC_p})} & & \Spaces^{\B\T} \times \Spaces^{\B\T}
\end{tikzcd}
\]
Consider the following sequence of functors of $\infty$-categories
\[
	\Spaces_p^{\Fr} \to \Sp^{\Fr}_p \to \CycSp_p,
\]
where the first induced by the suspension spectrum functor $\Sigma^\infty_+ : \Spaces \to \Sp$, and the second functor is induced by the natural transformation $(-)^{hC_p} \to (-)^{tC_p}$. Using that there is an equivalence of $\infty$-categories $\cat{P}(\B\W_{p^{\N}}) \simeq \Spaces^{\Fr}_p$, we obtain a sequence of functors of $\infty$-categories
\[
	\B\W \to \Spaces^{\Fr} \to \Spaces^{\Fr}_p \to \CycSp_p
\]
for every prime number $p$, where the first functor is given by the Yoneda embedding which we described explicitly in Example~\ref{example:witt_monoid_right_multiplication}, and the second functor is induced by the canonical map of $\E_1$-monoids $\W_{p^\N} \to \W$. Consequently, we obtain a functor of $\infty$-categories
\[
	\B\W \to \CycSp
\]
by virtue of the universal property of $\CycSp$ as a lax equalizer (cf.~\cite[Definition II.1.4]{NS18}). This functor carries the unique object of $\B\W$ to $\Sigma^\infty_+ \widetilde{\B}^{\epi} \Z_{\geq 0} \simeq \S[\W]$, and we note that the $\T$-action on $\Sigma^\infty_+ \widetilde{\B}^{\epi}\Z_{\geq 0}$ is the usual $\T$-action which unstably is given by $\lambda \cdot (x, n) = (\lambda x, n)$ for every element $\lambda$ of $\T$. The stable realization of the functor $\B\W \to \CycSp$ defines a functor of $\infty$-categories 
\[
	\Sp^{\Fr} \to \CycSp
\]
which we will refer to as the canonical functor. 
\end{construction}

The canonical functor $\Sp^{\Fr} \to \CycSp$ does not exhibit the $\infty$-category of spectra with Frobenius lifts as a subcategory of the $\infty$-category of cyclotomic spectra, that is the Frobenius lifts are structure and not a property. We establish the main result of this section. 

\begin{theorem} \label{theorem:adjunction_cycspfr_cycsp}
There is an adjunction of $\infty$-categories
\[
\begin{tikzcd}
	\Sp^{\Fr} \arrow[yshift=0.7ex]{r} & \CycSp \arrow[yshift=-0.7ex]{l}
\end{tikzcd}
\]
where the right adjoint is determined by the construction $X \mapsto \map_{\CycSp}(\widetilde{\THH}(\S[t]), X)$, and the left adjoint is given by the canonical functor constructed above. 
\end{theorem}

\begin{proof}
There is an equivalence of cyclotomic spectra $\widetilde{\THH}(\S[t]) \simeq \Sigma^\infty_+ \widetilde{\B}^{\epi} \Z_{\geq 0}$, so the claim follows by combining Proposition~\ref{proposition:stable_realization_right_adjoint} and Construction~\ref{construction:canonical_functor}. 
\end{proof}

\subsection{Topological restriction homology} \label{subsection:TR}
We present an alternative definition of $\TR$ as a functor on the $\infty$-category of cyclotomic spectra valued in the $\infty$-category of spectra with Frobenius lifts inspired by a result of Blumberg--Mandell~\cite{BM16}. Additionally, we study $\TR$ as a localizing invariant and discuss various descent properties of $\TR$.  

\begin{definition} \label{definition:TR}
If $X$ is a cyclotomic spectrum, then $\TR(X)$ is defined by
\[
	\TR(X) = \map_{\CycSp}(\widetilde{\THH}(\S[t]), X) \simeq \map_{\CycSp}\Big(\bigoplus_{i \geq 1} \Sigma^\infty_+(S^1/C_i), X\Big). 
\]
The construction $X \mapsto \TR(X)$ determines a functor of $\infty$-categories $\TR : \CycSp \to \Sp$. 
\end{definition}

\begin{notation}
If $\cat{C}$ is a stable $\infty$-category, then $\TR(\cat{C}) = \TR(\THH(\cat{C}))$. In particular, if $R$ is an $\E_1$-ring, then $\TR(R) = \THH(\mathrm{Perf}_R) \simeq \TR(\THH(R))$. 
\end{notation}

Definition~\ref{definition:TR} is based on a result of Blumberg--Mandell~\cite[Theorem 6.11]{BM16} which asserts that the classical construction of $\TR$ as considered by~\cite{BHM93,HM97,BM16} is corepresentable by the reduced topological Hochschild homology of the flat affine line $\S[t]$ as a functor on the homotopy category of genuine cyclotomic spectra with values in the homotopy category of spectra. In \textsection\ref{section:genuine}, we prove that the definition of $\TR$ given above recovers the classical definition of $\TR$ in the bounded below case (see Theorem~\ref{theorem:comparison_TR}). The following result is now an immediate consequence of the formalism developed in \textsection\ref{subsection:cycsp_fr}.

\begin{proposition} \label{proposition:TR_right_adjoint}
The functor $\TR : \CycSp \to \Sp$ refines to a functor of $\infty$-categories
\[
	\TR : \CycSp \to \Sp^{\Fr}
\]
which is a right adjoint of the canonical functor $\Sp^{\Fr} \to \CycSp$. 
\end{proposition}

\begin{proof}
Combine Theorem~\ref{theorem:adjunction_cycspfr_cycsp} and Definition~\ref{definition:TR}. 
\end{proof}

Krause--Nikolaus~\cite[Proposition 10.3]{KN19} prove that the construction given by $p$-typical $\TR$ determines a right adjoint of the canonical functor $\Sp^{\Fr}_p \to \CycSp_p$, and Proposition~\ref{proposition:TR_right_adjoint} extends this result to the integral situation.

\begin{remark}
Nikolaus--Scholze~\cite{NS18} prove that for every cyclotomic spectrum whose underlying spectrum is bounded below, there is a natural equivalence of spectra
\[
	\TC(X) \simeq \map_{\CycSp}(\S, X),
\]
where $\TC(X)$ denotes Goodwillie's integral topological cyclic homology. This was conjectured by Kaledin~\cite{Kal10} and proven by Blumberg--Mandell~\cite{BM16} after $p$-completion. Definition~\ref{definition:TR} provides a similar description of $\TR$ in the bounded below case removing the otherwise instrumental use of equivariant stable homotopy theory in the construction of $\TR$. 
\end{remark}

\begin{remark}
An advantage of Definition~\ref{definition:TR} is that we obtain an explicit equalizer formula for $\TR$. Let $X$ be a cyclotomic spectrum whose underlying spectrum is bounded below, and note that $X^{hC_i} \simeq \map_{\Sp^{\B\T}}(\Sigma^\infty_+(S^1/C_i), X)$ for every $i \geq 1$. Consequently, we obtain that
\[
	\TR(X) \simeq \mathrm{Eq} \Big( 
	\begin{tikzcd}
		\displaystyle\prod_{i \geq 1} X^{hC_i} \arrow[yshift=0.7ex]{r} \arrow[yshift=-0.7ex]{r} & \displaystyle\prod_{p}\displaystyle\prod_{i \geq 1} (X^{tC_p})^{hC_i}
	\end{tikzcd}
	\Big),
\]
where the top map is induced by $X \to X^{tC_p}$, and the bottom map is induced by the composite
\[
	X^{hC_i} \simeq (X^{hC_p})^{hC_{i/p}} \xrightarrow{\can^{hC_{i/p}}} (X^{tC_p})^{hC_{i/p}}
\]
provided that $p$ divides $i$. This is a consequence of the formula for the mapping spectrum in the $\infty$-category $\CycSp$ obtained in~\cite[Proposition II.1.5]{NS18}. 
\end{remark}

There is a construction of $p$-typical $\TR$ due to Nikolaus--Scholze~\cite{NS18} which only relies on the Borel equivariant homotopy theory of cyclotomic spectra, which we recall for completeness. 

\begin{remark}
For every $p$-typical cyclotomic spectrum $X$, we let $\TR^{n+1}(X, p)$ denote the following iterated pullback in the $\infty$-category of spectra with $\T$-action
\[
	X^{hC_{p^n}} \times_{(X^{tC_p})^{hC_{p^{n-1}}}} X^{hC_{p^{n-1}}} \times_{(X^{tC_p})^{hC_{p^{n-2}}}} \cdots \times_{(X^{tC_p})^{hC_{p}}} X^{hC_p} \times_{X^{tC_p}} X 
\]
for each $n \geq 0$. The maps from the left factors to the right factors are induced by the canonical map $X^{hC_p} \to X^{tC_p}$, and the maps from the right factors to the left factors are induced by cyclotomic Frobenius $X \to X^{tC_p}$. In fact, if the underlying spectrum of $X$ is bounded below, then the underlying spectrum of $X$ admits the structure of a genuine $\T$-spectrum with respect to the family of finite $p$-subgroups of $\T$, such that there is an equivalence $\TR^{n+1}(X, p) \simeq X^{C_{p^n}}$ for each $n \geq 0$.
For each $n \geq 1$, there is a map of spectra with $\T$-action $R : \TR^{n+1}(X, p) \to \TR^n(X, p)$ induced by forgetting the first factor in the iterated pullback defining $\TR^{n+1}(X, p)$, and there is an equivalence of spectra with $\T$-action 
\[
	\TR(X, p) \simeq \varprojlim_n \TR^{n+1}(X, p).
\]
Krause--Nikolaus~\cite[Definition 9.5]{KN19} give a description of integral $\TR$ by similar methods as above. However, it becomes complicated to specify coherences between the Frobenius lifts using this description. By similar methods as used in this paper, it is possible to show that there is a natural equivalence of $p$-typical cyclotomic spectra with Frobenius lift
\[
	\TR(X, p) \simeq \map_{\CycSp_p}\Big(\bigoplus_{i \geq 0} \Sigma^\infty_+(S^1/C_{p^i}), X\Big)
\]
for every $p$-typical cyclotomic spectrum $X$ whose underlying spectrum is bounded below.
\end{remark}

As an application, we describe $\TR$ as a localizing invariant. Let $\Cat_{\infty}^{\perf}$ denote the $\infty$-category of small idempotent-complete stable $\infty$-categories and exact functors among them. Recall that an exact sequence in $\Cat_\infty^{\perf}$ is a sequence which is both a fiber sequence and a cofiber sequence. Following Tamme~\cite{Tam18}, a localizing invariant with values in a stable $\infty$-category $\cat{D}$ is a functor $E : \Cat_{\infty}^{\perf} \to \cat{D}$ which carries exact sequences in $\Cat_\infty^{\perf}$ to fiber sequences in $\cat{D}$. Blumberg--Gepner--Tabuada~\cite{BGT13} additionally require that localizing invariants preserve filtered colimits.

\begin{corollary} \label{corollary:TR_localizing}
The functor $\Cat_\infty^{\perf} \xrightarrow{\THH} \CycSp \xrightarrow{\TR} \Sp^{\Fr}$ is a localizing invariant. 
\end{corollary}

\begin{proof}
The functor $\THH : \Cat_{\infty}^{\perf} \to \CycSp$ is a localizing invariant by Blumberg--Mandell~\cite{BM12}, and the functor $\TR : \CycSp \to \Sp^{\Fr}$ preserves limits by virtue of Proposition~\ref{proposition:TR_right_adjoint}.  
\end{proof}

\begin{remark}
We extend the definition of $\TR$ to schemes. Let $\mathrm{Perf}_X$ denote the $\infty$-category of perfect $\mathscr{O}_X$-modules for a scheme $X$, and define $\TR(X) = \TR(\mathrm{Perf}_X)$. Consequently, the construction $X \mapsto \TR(X)$ satisfies Nisnevich descent on quasi-compact quasi-separated schemes by a result of Thomason~\cite{TT07} since $\TR$ is a localizing invariant. 
\end{remark}

We end by discussing various descent properties for $\TR$ building on the work of~\cite{BMS19,Kee20,AN20,CMNN20,CM19}. In the following, we will let $\CAlg^\heartsuit$ denote the category of discrete commutative rings. 

\begin{corollary} \label{corollary:TR_fpqc}
The construction $R \mapsto \TR(R)$ determines a functor of $\infty$-categories
\[
	\TR : \CAlg^\heartsuit \to \Sp^{\Fr}
\]
which is a sheaf for the fpqc topology on $\CAlg^\heartsuit$. 
\end{corollary}

\begin{proof}
We first show that the functor $\THH : \CAlg^\heartsuit \to \CycSp$ is an fpqc sheaf on $\CAlg^\heartsuit$. As a consequence of the construction of the $\infty$-category $\CycSp$ as a lax equalizer it suffices to show that the functors $\THH : \CAlg^\heartsuit \to \Sp^{\B\T}$ and $\THH^{tC_p} : \CAlg^\heartsuit \to \Sp^{\B\T}$ are fpqc sheaves for every prime $p$. This follows from~\cite[Corollary 3.4 and Remark 3.5]{BMS19} since limits in $\Sp^{\B\T}$ are computed pointwise. This shows the desired statement since $\TR : \CycSp \to \Sp^{\Fr}$ preserves limits. 
\end{proof}

The proof of Corollary~\ref{corollary:TR_fpqc} above shows that every descent result for $\THH$ regarded as a functor with values in the $\infty$-category of cyclotomic spectra yields a corresponding descent statement for $\TR$. In~\cite{Kee20}, Keenan extends the result of Bhatt--Morrow--Scholze~\cite{BMS19} on fpqc descent for $\THH$ to connective $\E_\infty$-rings, and Antieau--Nikolaus~\cite{AN20} prove that $\THH$ is a hypercomplete sheaf with values in the $\infty$-category of cyclotomic spectra for the pro-étale topology on $\CAlg^{\heartsuit}$. In~\cite{CM19}, Clausen--Mathew prove that $\THH$ is a Postnikov complete sheaf for the étale topology on $\E_2$-rings extending their previous work with Naumann and Noel in~\cite{CMNN20}.

\section{Comparison with genuine TR} \label{section:genuine}
In this section, we establish the main technical result of this paper which asserts that the classical construction of $\TR$ agrees with the construction of $\TR$ considered in \textsection\ref{subsection:TR}. In \textsection\ref{subsection:equivariant}, we review the formalism of Mackey functors defined on orbital $\infty$-categories following Barwick~\cite{Bar17}, and employ this to construct the $\infty$-category of genuine $\T$-spectra following Barwick--Glasman~\cite{BG16}. In \textsection\ref{subsection:tcart}, we define the notion of a topological Cartier modules, and obtain an explicit description of the free topological Cartier module on a spectrum with Frobenius lifts as an instance of the Segal--tom Dieck splitting for Mackey functors defined on orbital $\infty$-categories. In \textsection\ref{subsection:genuine_cycsp}, we briefly recall the classical construction of $\TR$ following Blumberg--Mandell~\cite{BM16}, and prove that this coincides with the construction of $\TR$ considered in \textsection\ref{subsection:TR} in the bounded below case. 

\subsection{Equivariant stable homotopy theory} \label{subsection:equivariant}
In this section, we briefly review the notions from equivariant stable homotopy theory that will play an important role in this paper. In~\cite{GM11}, Guillou--May establish a model for the homotopy theory of $G$-spectra in the sense of~\cite{LMS86,MM02,HHR16} for a finite group $G$, in terms of spectral Mackey functors, and Barwick~\cite{Bar17} developed an $\infty$-categorical approach to Mackey functors on orbital $\infty$-categories. Barwick--Dotto--Glasman--Nardin--Shah revisit the result of Guillou--May in the general context of parametrized homotopy theory~\cite{BDGNS16, Nar16}. Our goal in this section is to briefly review the formalism of Mackey functors following~\cite{Bar17}, and use this to construct the $\infty$-category of genuine $\T$-spectra following Barwick--Glasman~\cite{BG16}.

\begin{definition} \label{definition:orbital_infty_category}
The finite coproduct completion $\Fin_T$ of a small $\infty$-category $T$ is the smallest full subcategory of $\cat{P}(T)$ which contains the essential image of the Yoneda embedding and which is closed under finite coproducts. A small $\infty$-category $T$ is orbital if the finite coproduct completion of $T$ admits pullbacks.
\end{definition}

\begin{remark}
Definition~\ref{definition:orbital_infty_category} was introduced by Barwick--Dotto--Glasman--Nardin--Shah in~\cite{BDGNS16}. Note that the finite coproduct completion of an orbital $\infty$-category is disjunctive in the sense of Barwick \cite{Bar17}. Additionally, the epiorbital $\infty$-categories in the sense of Glasman~\cite{Gla15} are examples of orbital $\infty$-categories.
\end{remark}

For instance, if $G$ is a finite group, then the finite coproduct completion of the orbit category of $G$ is equivalent to the category of finite $G$-sets, since every finite $G$-set admits a unique decomposition as a disjoint union of orbits. In other words, the orbit category of $G$ is orbital in the sense of Definition~\ref{definition:orbital_infty_category}. 

\begin{remark} \label{remark:universal_property_Fin}
The finite coproduct completion of a small $\infty$-category $T$ is characterized by a universal property: For every $\infty$-category $\cat{D}$ which admits finite coproducts, the Yoneda embedding $j : T \to \Fin_T$ induces an equivalence of $\infty$-categories
\[
	\Fun^{\amalg}(\Fin_T, \cat{D}) \to \Fun(T, \cat{D}),
\]
where $\Fun^{\amalg}(\Fin_T, \cat{D})$ denotes the full subcategory of $\Fun(\Fin_T, \cat{D})$ spanned by those functors which preserve finite coproducts. This follows from~\cite[Proposition 5.3.6.1]{Lur09} by taking $\cat{K} = \mathrm{Fin}$ and $\cat{R} = \varnothing$. Note that the inverse is determined by forming the left Kan extension along the Yoneda embedding. 
\end{remark}

Let $T$ be an orbital $\infty$-category, and let $\Span(\Fin_T)$ denote the $\infty$-category of spans in the finite coproduct completion of $T$. Concretely, the objects of $\Span(\Fin_T)$ are given by the objects of $\Fin_T$, and a morphism from $X$ to $X'$ in $\Span(\Fin_T)$ is given by a span
\[
\begin{tikzcd}[sep = small]
	& Y \arrow{dr} \arrow{dl} & \\
	X & & X' 
\end{tikzcd}
\]
in $\Fin_T$. Additionally, if $X' \leftarrow Y' \rightarrow X''$ is a morphism from $X'$ to $X''$ in $\Span(\Fin_T)$, then composition is defined by forming a pullback in $\Fin_T$
\[
\begin{tikzcd}[sep = small]
	& & P \arrow{dl} \arrow{dr} & & \\
	& Y \arrow{dl} \arrow{dr} & & Y' \arrow{dl} \arrow{dr} & \\
	X & & X' & & X''
\end{tikzcd}
\]
which exists by virtue of the assumption that $T$ is an orbital $\infty$-category. The reader is invited to consult~\cite[Proposition 3.4]{Bar17} for a precise construction of $\Span(\Fin_T)$ as a complete Segal space.

\begin{construction} \label{construction:maps_into_span}
Let $T$ be an orbital $\infty$-category. There is a functor $i : \Fin_T^{\op} \to \Span(\Fin_T)$ which is the identity on objects, and determined by the following construction on morphisms
\[
	(X \to X') \mapsto (X' \leftarrow X \xrightarrow{\id} X)
\]
Consequently, we obtain a functor $T^{\op} \to \Span(\Fin_T)$ given by precomposing the functor $i$ with the opposite of the canonical functor $T \to \Fin_T$. Similarly, there is a functor $i' : \Fin_T \to \Span(\Fin_T)$ which is the identity on objects, and determined by  the following construction on morphisms
\[
	(X \to X') \mapsto (X \xleftarrow{\id} X \to X')
\]
We obtain a functor $T \to \Span(\Fin_T)$ given by precomposing the functor $i'$ above with the canonical functor $T \to \Fin_T$. See~\cite[Notation 3.9]{Bar17} for a formal description of these functors. 
\end{construction}

If $T$ is an orbital $\infty$-category, then the $\infty$-category $\Span(\Fin_T)$ is semiadditive and the sum is given by the coproduct in $\Fin_T$ (cf.~\cite[Proposition 4.3]{Bar17}). Consequently, the $\infty$-category of Mackey functors on an orbital $\infty$-category is defined as follows (cf.~\cite[Definition 6.1]{Bar17}):

\begin{definition}
Let $T$ denote an orbital $\infty$-category, and let $\cat{D}$ denote an additive $\infty$-category. The $\infty$-category of $\cat{D}$-valued Mackey functors on $T$ is defined by
\[
	\Mack_{\cat{D}}(T) = \Fun^{\times}(\Span(\Fin_T), \cat{D}),
\]
where $\Fun^{\times}(\Span(\Fin_T), \cat{D})$ denotes the full subcategory of $\Fun(\Span(\Fin_T), \cat{D})$ spanned by those functors which preserve products. If $\cat{D} = \Sp$ is the $\infty$-category of spectra, then we write $\Mack(T)$ instead of $\Mack_{\Sp}(T)$, and refer to the former as the $\infty$-category of spectral Mackey functors on $T$. 
\end{definition}

\begin{example}
The orbit category of a finite group $G$ is an example of an orbital $\infty$-category. The $\infty$-category of spectral Mackey functors on $\Orb_G$ is equivalent to the $\infty$-category of genuine $G$-spectra by the work of Guillou--May~\cite{GM11}. By the $\infty$-category of genuine $G$-spectra we mean the underlying $\infty$-category of the category of orthogonal $G$-spectra equipped with the model structure established by Mandell--May in~\cite{MM02}. Alternatively, we refer to Nardin~\cite[Theorem A.4]{Nar16} or Clausen--Mathew--Naumann--Noel~\cite[Appendix A]{CMNN20b} for a direct comparison using Barwick's model of spectral Mackey functors.
\end{example}

We construct the $\infty$-category of genuine $\T$-spectra with respect to the family of finite cyclic subgroups of $\T$ using the formalism of spectral Mackey functors on orbital $\infty$-categories following Barwick--Glasman~\cite{BG16}. An alternative construction has been obtained by Ayala--Mazel-Gee--Rozenblyum~\cite{AMGR17a}, and we refer the reader to the foundational work of Lewis--May--Steinberger~\cite{LMS86} and Mandell--May~\cite{MM02} for the classical approach. 

\begin{definition} \label{definition:orbitcategory}
The orbit $\infty$-category $\OrbT$ of the circle is defined as the full subcategory of the $\infty$-category $\Spaces^{\B\T}$ spanned by those spaces with $\T$-action of the form $\T/C_n$ for every $n \geq 1$. 
\end{definition}

The orbit $\infty$-category of the circle is equivalent to the cyclonic category of Barwick--Glasman~\cite[Definition 1.10]{BG16} as explained in~\cite[Remark 1.13]{BG16}, so we conclude that the orbit $\infty$-category $\OrbT$ is orbital by virtue of Barwick--Glasman~\cite[Proposition 1.25.1]{BG16}.

\begin{definition}
The $\infty$-category of genuine $\T$-spectra is defined by $\Sp_{\T} = \Mack(\OrbT)$. 
\end{definition}

\begin{remark}
The $\infty$-category of genuine $\T$-spectra is equivalent to the underlying $\infty$-category of orthogonal $\T$-spectra with respect to the family of finite subgroups of $\T$, equipped with the stable model structure established by Mandell--May~\cite{MM02}. This is proved by Barwick--Glasman~\cite[Theorem 2.8]{BG16}, where $\Sp_{\T}$ is referred to as the $\infty$-category of cyclonic spectra.
\end{remark}

We construct an action of the multiplicative monoid $\N^{\times}$ on the $\infty$-category $\Sp_{\T}$ of genuine $\T$-spectra which in turn is determined by an action of $\N^{\times}$ on the orbit $\infty$-category $\OrbT$.

\begin{construction} \label{construction:actionN_OrbT}
Let $\mathbf{Orb}_{\T}$ denote the category whose objects are given by the orbits $\T/C_n$ for every integer $n \geq 1$, and whose morphisms are given by $\T$-equivariant maps. For every pair of positive integers $m$ and $n$, there is a canonical bijection
\[
	\Hom_{\mathbf{Orb}_{\T}}(\T/C_m, \T/C_n) \simeq (\T/C_n)^{C_m},
\]
where the right hand side is equipped with the canonical topology. As a consequence, we may regard $\mathbf{Orb}_{\T}$ as a topological category. The underlying $\infty$-category of $\mathbf{Orb}_{\T}$ is equivalent to the orbit $\infty$-category $\OrbT$ as defined in Definition~\ref{definition:orbitcategory}.
Indeed, note that there is an essentially surjective functor of $\infty$-categories $\mathbf{Orb}_{\T} \to \OrbT$ which is fully faithful since the canonical map
\[
	(\T/C_n)^{C_m} \to (\T/C_n)^{hC_m}
\]
is an equivalence of spaces for every pair of positive integers $m$ and $n$. The analogues statement for higher-dimensional Lie groups fails. For every integer $k \geq 1$, the construction $\T/C_n \mapsto \T/C_{kn}$ determines a functor of topological categories $i_k : \mathbf{Orb}_{\T} \to \mathbf{Orb}_{\T}$ which in turn defines an action of the multiplicative monoid $\N^{\times}$ on $\mathbf{Orb}_{\T}$. This defines an action of $\N^{\times}$ on the orbit $\infty$-category $\OrbT$ which is determined by a functor of $\infty$-categories $\B\N^{\times} \to \Cat_\infty$. 
\end{construction}

The following result is due to Barwick--Glasman~\cite[Lemma 3.2]{BG16}.

\begin{lemma} \label{lemma:N_action_SpT}
For every positive integer $k$, there is an adjunction of $\infty$-categories
\[
\begin{tikzcd}[column sep = large]
	\Fin_{\OrbT} \arrow[yshift=0.7ex]{r}{\Fin(i_k)} &  \Fin_{\OrbT} \arrow[yshift=-0.7ex]{l}{p_k}
\end{tikzcd}
\]
where the left adjoint $\Fin(i_k)$ additionally preserves finite coproducts. The action of the multiplicative monoid $\N^\times$ on the $\infty$-category $\OrbT$ extends to an action on the $\infty$-category $\Sp_{\T}$.  
\end{lemma}

\begin{proof}
See~\cite[Lemma 3.2]{BG16}. 
\end{proof}

We have that the action of the multiplicative monoid $\N^{\times}$ on the $\infty$-category $\Sp_{\T}$ of genuine $\T$-spectra afforded by Lemma~\ref{lemma:N_action_SpT} determines a functor of $\infty$-categories
\[
	F_{\Psi} : \B\N^{\times} \to \Cat_\infty
\]
which is determined by the construction which sends $k \in \N^{\times}$ to the endofunctor of the $\infty$-category $\Sp_{\T}$ determined by the construction $X \mapsto X^{C_k}$. Furthermore, there is an action of the multiplicative monoid $\N^{\times}$ on the $\infty$-category $\Sp_{\T}$ given by the geometric fixed points construction $(-)^{\Phi C_k}$. We briefly discuss this, and refer the reader to~\cite[Notation 3.4]{BG16} for a complete treatment.

\begin{example} \label{example:geometric_fixedpoints}
For every positive integer $k$, there is an adjunction of $\infty$-categories
\[
\begin{tikzcd}
	\Sp_{\T} \arrow[yshift=0.7ex]{rr}{(-)^{\Phi C_k}} & & \Sp_{\T} \arrow[yshift=-0.7ex]{ll}{\Span(p_k)^{\ast}},
\end{tikzcd}
\]
where the left adjoint $(-)^{\Phi C_k}$ is defined by left Kan extension along $\Span(p_k)^{\ast}$, where $p_k$ denotes a right adjoint of $\Fin(i_k)$ as in Lemma~\ref{lemma:N_action_SpT}. Consequently, there is an action of the multiplicative monoid $\N^{\times}$ on $\Sp_{\T}$ by functoriality, and there is a functor of $\infty$-categories
\[
	F_{\Phi} : \B\N^{\times} \to \Cat_\infty
\]
which is determined by the construction which sends $k \in \N^\times$ to the endofunctor of the $\infty$-category of genuine $\T$-spectra determined by the construction $X \mapsto X^{\Phi C_k}$. 
\end{example}

\subsection{Topological Cartier modules} \label{subsection:tcart}
We introduce the notion of an integral topological Cartier module using the formalism of spectral Mackey functors on orbital $\infty$-categories as reviewed in \textsection\ref{subsection:equivariant}, extending the work of Antieau--Nikolaus~\cite{AN20} in the $p$-typical situation. Furthermore, we prove a general version of the Segal--tom Dieck splitting for spectral Mackey functors on orbital $\infty$-categories, and use this to obtain an explicit formula for the free topological Cartier module on a spectrum with Frobenius lifts. The starting point is the following result:

\begin{lemma} \label{lemma:BW_orbital}
The $\infty$-category $\B\W$ is orbital, and there is an equivalence of $\infty$-categories
\[
	\B\W \simeq (\OrbT)_{h\N^{\times}},
\]
where the multiplicative monoid $\N^{\times}$ acts on the orbit $\infty$-category $\OrbT$ as in Construction~\ref{construction:actionN_OrbT}.
\end{lemma}

\begin{proof}
We first prove that the $\infty$-category $\B\W$ is orbital. Let $\mathbf{Mfld}_1^c$ denote the topological category whose objects are given by compact oriented $1$-manifolds and whose space of morphisms between a pair of compact oriented $1$-manifolds is given by the set of covering maps of positive degree equipped with the compact-open topology. In the following, we will let $\mathrm{Mfld}_1^c$ denote the underlying $\infty$-category of the topological category $\mathbf{Mfld}_1^c$. The construction $\ast \mapsto S^1$ determines a functor of $\infty$-categories $\B\W \hookrightarrow \mathrm{Mlfd}_1^c$ which is fully faithful since there is an equivalence of $\E_1$-monoids $\W \simeq \mathrm{End}_{\mathrm{Mfld}_1^c}(S^1)$. We obtain a coproduct-preserving functor of $\infty$-categories
\[
	\Fin_{\B\W} \to \mathrm{Mfld}_1^c
\]
given by forming the left Kan extension of the functor $\B\W \hookrightarrow \mathrm{Mfld}_1^c$ along the Yoneda embedding $\B\W \to \Fin_{\B\W}$ by virtue of Remark~\ref{remark:universal_property_Fin}, and this functor is given by the construction $\ast \amalg \cdots \amalg \ast \mapsto S^1 \amalg \cdots \amalg S^1$. Consequently, it suffices to show that the $\infty$-category $\mathrm{Mfld}_1^c$ admits pullbacks since the functor $\Fin_{\B\W} \to \mathrm{Mfld}_1^c$ reflects pullbacks. To this end, it suffices to show that the diagram
\[
\begin{tikzcd}
	& S^1 \arrow{d}{\alpha_m} \\
	S^1 \arrow{r}{\alpha_n} & S^1
\end{tikzcd}
\]
admits a pullback in $\mathrm{Mfld}_1^c$, where $\alpha_k$ denotes the $k$th power map of $S^1$ for every integer $k \geq 1$. Let $g = \mathrm{gcd}(m, n)$ denote the greatest common divisor of $m$ and $n$. Then the following diagram 
\[
\begin{tikzcd}[sep = large]
	S^1 \amalg \cdots \amalg S^1 \arrow{r}{\alpha_{\frac{n}{g}} \circ (g \alpha_g)} \arrow[swap]{d}{\alpha_{\frac{m}{g}} \circ (g \alpha_g)} & S^1 \arrow{d}{\alpha_m} \\
	S^1 \arrow{r}{\alpha_n} & S^1
\end{tikzcd}
\]
is a pullback of the diagram above in the $\infty$-category $\mathrm{Mfld}_1^c$, where $g \alpha_g$ denotes the self-map of the $g$-fold coproduct of $S^1$ with itself defined by the formula $g\alpha_g = \alpha_g \amalg \cdots \amalg \alpha_g$. This proves that the $\infty$-category $\B\W$ is orbital, so it remains to prove that there is an equivalence of $\infty$-categories $\B\W \simeq (\OrbT)_{h\N^\times}$. Let $\cat{E}$ denote the $2$-category defined as follows:

\begin{itemize}[leftmargin=2em, topsep=5pt, itemsep=3pt]
	\item An object of $\cat{E}$ is given by an object $\T/C_n$ of $\OrbT$.

	\item A morphism from $\T/C_m$ to $\T/C_n$ in $\cat{E}$ is given by a pair $(k, f)$ consisting of an element $k$ of $\N^\times$, and a morphism $f : \T/C_{km} \to \T/C_n$ in $\OrbT$. If $(l, g)$ is an additional morphism from $\T/C_n$ to $\T/C_r$, then $(l, g) \circ (k, f) = (lk, h)$, where $h$ is the morphism in $\OrbT$ given by
	\[
		\T/C_{lkm} \xrightarrow{i_l(f)} \T/C_{ln} \xrightarrow{g} \T/C_r.
	\]

	\item Let $(k', f')$ denote an additional morphism from $\T/C_m$ to $\T/C_n$ in $\cat{E}$. If $k = k'$, then
	\[
		\Hom_{\Hom_{\cat{E}}(\T/C_m, \T/C_n)}((k, f), (k', f')) = \Hom_{\Hom_{\OrbT}(\T/C_{km}, \T/C_n)}(f, g)
	\]
	and empty otherwise. We have used that $\OrbT$ is equivalent to a $2$-category. 
\end{itemize}

The construction $(k, f) \mapsto k$ determines a functor $p : \cat{E} \to \B\N^\times$ which is a cocartesian fibration classifying the action of the multiplicative monoid $\N^\times$ on $\OrbT$. The fact that $p$ is a cocartesian fibration follows from~\cite[Lemma 3.10 and Lemma 3.11]{BG16}. Consequently, there is an equivalence
\[
	(\OrbT)_{h\N^\times} \simeq \cat{E}[\mathrm{cocart}^{-1}]
\]
by~\cite[Corollary 3.3.4.3]{Lur09}, where $\cat{E}[\mathrm{cocart}^{-1}]$ denotes the localization at the set of $p$-cocartesian edges of $\cat{E}$. We prove that there is an equivalence of $\infty$-categories $\cat{E}[\mathrm{cocart}^{-1}] \simeq \B\W$. For every pair of elements $k$ and $m$ of $\N^\times$, the pair $(k, \id)$ defines a morphism $\T/C_m \to \T/C_{km}$ which is a $p$-cocartesian edge of $\cat{E}$. Indeed, if $(l, f)$ is an additional morphism $\T/C_m \to \T/C_n$ in $\cat{E}$ such that $l = kr$ for some $r \in \N^\times$, then the pair $(r, f)$ defines a morphism $\T/C_{km} \to \T/C_n$ in $\cat{E}$ such that $(l, f) = (r, f) \circ (k, \id)$. In fact, every $p$-cocartesian edge of $\cat{E}$ is of this form. Consider the assignment which regards $\T/C_n$ as an object of the $\infty$-category $\mathrm{Mfld}_1^c$, and which carries a morphism $(k, f) : \T/C_m \to \T/C_n$ in $\cat{E}$ to the following map
\[
	\T/C_m \xrightarrow{\sqrt[k]{-}} \T/C_{km} \xrightarrow{f} \T/C_n,
\]
where the first map is given by $[x] \mapsto [\sqrt[k]{x}]$. We claim that this assignment determines a functor $\mathrm{W} : \cat{E} \to \mathrm{Mfld}_1^c$. Indeed, this assignment is compatible with composition: if $(l, g) : \T/C_n \to \T/C_r$ is an additional morphism in $\cat{E}$, then the following composite
\[
	\T/C_m \xrightarrow{\sqrt[lk]{-}} \T/C_{lkm} \xrightarrow{i_l(f)} \T/C_{ln} \xrightarrow{g} \T/C_r
\]
is canonically equivalent to the following composite
\[
	\T/C_m \xrightarrow{\sqrt[k]{-}} \T/C_{km} \xrightarrow{f} \T/C_n \xrightarrow{\sqrt[l]{-}} \T/C_{ln} \xrightarrow{g} \T/C_r.  
\]
since the morphism $i_l(f)$ is given by $\T/C_{lkm} \xrightarrow{(-)^l} \T/C_{km} \xrightarrow{f} \T/C_n \xrightarrow{\sqrt[l]{-}} \T/C_{ln}$. Consequently, we obtain a functor of $\infty$-categories $\mathrm{W} : \cat{E} \to \mathrm{Mfld}_1^c$. Note that the map $\T/C_m \to \T/C_{km}$ is an equivalence in $\mathrm{Mfld}_1^c$, so the functor $\mathrm{W}$ carries every $p$-cocartesian edge of $\cat{E}$ to an equivalence in $\mathrm{Mfld}_1^c$. Hence, the functor $\mathrm{W} : \cat{E} \to \mathrm{Mfld}_1^c$ canonically extends to a functor of $\infty$-categories
\[
	(\OrbT)_{h\N^\times} \simeq \cat{E}[\mathrm{cocart}^{-1}] \to \mathrm{Mfld}_1^c.
\]
This functor is fully faithful, so $(\OrbT)_{h\N^\times}$ is equivalent to the $\infty$-category with a single object $\T/C_1$, and $\mathrm{End}_{(\OrbT)_{h\N^\times}}(\T/C_1) \simeq \mathrm{End}_{\mathrm{Mfld}_1^c}(S^1) \simeq \W$. This proves the desired statement. 
\end{proof}

We define the $\infty$-category of topological Cartier modules as follows:

\begin{definition} \label{definition:TCart}
The $\infty$-category of $\cat{D}$-valued topological Cartier modules is defined by
\[
	\TCart_{\cat{D}} = \Mack_{\cat{D}}(\B\W) = \Fun^{\times}(\Span(\Fin_{\B\W}), \cat{D}),
\]
where $\cat{D}$ denotes an additive $\infty$-category. If $\cat{D} = \Sp$, then we will write $\TCart$ instead of $\TCart_{\Sp}$, and refer to the former as the $\infty$-category of topological Cartier modules. 
\end{definition}

\begin{remark} \label{remark:unwind_tcart}
The functor $\B\W^{\op} \to \Span(\Fin_{\B\W})$ induces a functor of $\infty$-categories
\[
	\TCart \to \Sp^{\Fr}
\]
which regards the underlying spectrum of a topological Cartier module as a spectrum with Frobenius lifts, where the underlying spectrum of a topological Cartier module is defined as the spectrum with $\T$-action obtained by precomposing the functor above with the forgetful functor $\Sp^{\Fr} \to \Sp^{\B\T}$. The functor $\B\W \to \Span(\Fin_{\B\W})$ induces a functor of $\infty$-categories
\[
	\TCart \to \Fun(\B\W, \Sp),
\]
so as a consequence, if $M$ is a topological Cartier module, then for each integer $k \geq 1$, there is a $\T$-equivariant map of spectra $V_k : M_{hC_k} \to M$, which we will refer to as the $k$th Verschiebung map of $M$.
\end{remark}

\begin{remark}
Antieau--Nikolaus~\cite[Definition 3.1]{AN20} define a $p$-typical topological Cartier module as a spectrum $M$ with $\T$-action equipped with the datum of a $\T$-equivariant factorization
\[
	M_{hC_p} \xrightarrow{V} M \xrightarrow{F} M^{hC_p}
\]
of the norm map for the cyclic group $C_p$. In Remark~\ref{remark:isotropy}, we make the discussion in Remark~\ref{remark:unwind_tcart} precise by showing that the underlying spectrum of a topological Cartier module in the sense of Definition~\ref{definition:TCart} canonically admits the structure of a $p$-typical topological Cartier module for every prime number $p$.  
\end{remark}

\begin{example}
If $M$ denotes a topological Cartier module, then the abelian group $\pi_0 M$ comes equipped with a pair of endomorphism $V_k = i'(k)$ and $F_k = i(k)$ for every positive integer $k$, where $i'$ and $i$ denote the functors defined in Construction~\ref{construction:maps_into_span}. Let $g = \mathrm{gcd}(m, n)$ denote the greatest common divisor of a pair of positive integers $m$ and $n$. Then the commutative diagram
\[
\begin{tikzcd}[sep = large]
	\ast \amalg \cdots \amalg \ast \arrow{r}{\frac{n}{g} \circ (g \cdot \id)} \arrow[swap]{d}{\frac{m}{g} \circ (g \cdot \id)} & \ast \arrow{d}{m} \\
	\ast \arrow{r}{n} & \ast 
\end{tikzcd}
\]
is a pullback in the finite coproduct completion $\Fin_{\B\W}$, so we conclude that $F_m V_n = \mathrm{gcd}(m, n) V_{\frac{n}{g}} F_{\frac{m}{g}}$. In particular, we have that $F_k V_k = k \cdot \id$ for every positive integer $k$, and if $\mathrm{gcd}(m, n) = 1$, then $F_m V_n = V_n F_m$. We will refer to this as a Cartier module structure on $\pi_0 M$ (see Zink~\cite{Zink84}). This structure arises frequently in algebra; for instance if $R$ is a commutative ring, then the ring of big Witt vectors $W(R)$ and the ring of rational Witt vectors $W_{\mathrm{rat}}(R)$ both admit the structure of a Cartier module in this sense.
\end{example}

We show that the $\infty$-category of topological Cartier modules is equivalent to the $\infty$-category of genuine topological Cartier modules as defined by Antieau--Nikolaus~\cite[Definition 5.1]{AN20}.

\begin{definition} \label{definition:Tcart_gen}
The $\infty$-category of genuine topological Cartier modules is defined by
\[
	\TCart^{\gen} = \lim(F_{\Psi} : \B\N^{\times} \to \Cat_{\infty}) = (\Sp_{\T})^{h\N^\times},
\]
where the multiplicative monoid $\N^{\times}$ acts on the $\infty$-category $\Sp_{\T}$ as in Lemma~\ref{lemma:N_action_SpT}. 
\end{definition}

\begin{remark}
Unwinding the definition, we see that an object of $\TCart^{\gen}$ is given by a genuine $\T$-spectrum $M$ with compatible equivalences of genuine $\T$-spectra $M \simeq M^{C_k}$ for every $k \geq 1$. 
\end{remark}

\begin{proposition} \label{proposition:TCart_equiv_TCart_gen}
There is an equivalence of $\infty$-categories $\TCart \simeq \TCart^{\gen}$. 
\end{proposition}

\begin{proof}
There is a sequence of equivalences of $\infty$-categories
\[
	\TCart = \Mack(\B\W) \xrightarrow{\simeq} \Mack((\OrbT)_{h\N^{\times}}) \simeq \Mack(\OrbT)^{h\N^{\times}} = \TCart^{\gen}
\]
by virtue of Lemma~\ref{lemma:BW_orbital} and Lemma~\ref{lemma:N_action_SpT} which shows the wanted.  
\end{proof}

Antieau--Nikolaus prove that there is an equivalence $\TCart_p \simeq \TCart^{\gen}_p$\cite[Proposition 5.5]{AN20}, and Proposition~\ref{proposition:TCart_equiv_TCart_gen} provides an integral version of this result. Note that these comparison results hold unconditionally in contrast to the comparison between the $\infty$-category of cyclotomic spectra and the $\infty$-category of genuine cyclotomic spectra established by Nikolaus--Scholze~\cite{NS18}, where it is crucial to restrict to the bounded below case.

\begin{remark} \label{remark:isotropy}
There is a canonical sequence of functors of $\infty$-categories
\[
	\TCart \to \Sp_{\T} \to \cat{P}_{\Sp}(\OrbT) \to \Sp^{\B\T},
\]
where the first functor exists by virtue of Proposition~\ref{proposition:TCart_equiv_TCart_gen}. In particular, if $M$ is a topological Cartier module, then there is a commutative diagram of spectra with $\T$-action
\[
\begin{tikzcd}
	M_{hC_k} \arrow{r} \arrow{d}{\simeq} & M \arrow{r} \arrow{d} & M^{\Phi C_k} \arrow{d} \\
	M_{hC_k} \arrow{r}{\Nm_{C_k}} & M^{hC_k} \arrow{r}{\can} & M^{tC_k}
\end{tikzcd}
\]
for every positive integer $k$, where we have used that $M^{C_k} \simeq M$ in $\Sp_{\T}$. This diagram was introduced by Hesselholt--Madsen~\cite{HM97} building on work of Greenlees--May~\cite{GM95}, and is commonly referred to as the isotropy-separation diagram. Note that the commutative square on the left in the diagram above exhibits the underlying spectrum with $\T$-action of $M$ as a $p$-typical topological Cartier module for every prime number $p$ in the sense of Antieau--Nikolaus~\cite{AN20}. 
\end{remark}

\begin{remark}
A detailed treatment of the $\infty$-category of topological Cartier modules extending the results of Antieau--Nikolaus~\cite{AN20} to the integral situation will be the content of forthcoming work. 
\end{remark}

We obtain an explicit identification of the topological Cartier module obtained from a spectrum with Frobenius lifts by freely adjoining Verschiebung maps. It follows formally from~\cite[Proposition 6.5]{Bar17}, that there is an adjunction of $\infty$-categories
\begin{equation} \label{equation:free_adjunction}
\begin{tikzcd}
	\Sp^{\Fr} \arrow[yshift=0.7ex]{r}{\Free} & \TCart \arrow[yshift=-0.7ex]{l}
\end{tikzcd}
\end{equation}
where $\Free$ denotes a left adjoint of the forgetful functor $\TCart \to \Sp^{\Fr}$. Antieau--Nikolaus~\cite{AN20} give a direct construction of the $p$-typical analogue of $\Free$, but this is not a viable approach in the integral situation due to the infinite hierarchy of coherences that need to be specified to define such a functor. We will resolve this issue by establishing a general version of the Segal--tom Dieck splitting for spectral Mackey functors on orbital $\infty$-categories. We will begin by recalling the usual Segal--tom Dieck splitting~\cite{Die75,LMS86} in equivariant stable homotopy theory (see~\cite{Sch20} for a classical treatment). For a finite group $G$, there is an adjunction
\begin{equation} \label{equation:tomDieck_adjunction}
\begin{tikzcd}
	\cat{P}_{\Sp}(\Orb_G) \arrow[yshift=0.7ex]{r}{F} & \Sp_G \arrow[yshift=-0.7ex]{l}
\end{tikzcd}
\end{equation}
where the right adjoint is obtained by restricting along the functor $\Orb_G^{\op} \to \Span(\Fin_G)$, and the left adjoint $F$ is given by sending a spectral presheaf $X$ on $\Orb_G$ to the genuine $G$-spectrum with
\[
	F(X)^G \simeq \bigoplus_{(H)} (X^H)_{h\mathrm{W}_G(H)},
\]
where the sum is indexed by conjugacy classes of subgroups of $G$, and $\mathrm{W}_G(H)$ is the Weyl group. In general, if $Y$ is a spectral presheaf on $\Orb_G$, then we define $Y^H = Y(G/H)$ for every subgroup $H$ of $G$. The idea is to regard the adjunction in~(\ref{equation:free_adjunction}) as an instance of this adjunction, where we replace the role of the orbit category of $G$ above by the Witt monoid. Let $T$ denote an orbital $\infty$-category, and note that there is an adjunction of $\infty$-categories
\[
\begin{tikzcd}
	\Fun(T^{\op}, \Sp) \arrow[yshift=0.7ex]{r}{L} & \Mack(T) = \Fun^{\times}(\Span(\Fin_T), \Sp), \arrow[yshift=-0.7ex]{l}
\end{tikzcd}
\]
where the right adjoint is the forgetful functor. In Proposition~\ref{proposition:tomDieck_orbital} below, we obtain an explicit formula for the left adjoint $L$ above which we regard as a version of the Segal--tom Dieck splitting for spectral Mackey functors defined on the orbital $\infty$-category $T$. In~\cite[Theorem A.9]{Bar17}, Barwick proves a similar result establishing a version of the Segal--tom Dieck splitting for coherent topoi.

\begin{proposition} \label{proposition:tomDieck_orbital}
If $X$ is a presheaf on an orbital $\infty$-category $T$ with values in the $\infty$-category of spectra, then there is a natural equivalence of spectra
\[
	(LX)(t) \simeq \colim_{(t' \to t) \in (T_{/t})^{\simeq}} X(t')
\]
for every object $t$ of $T$. 
\end{proposition}

\begin{proof}
Consider the following pair of adjunctions of $\infty$-categories
\[
\begin{tikzcd}
	\Fun(T^{\op}, \Sp) \arrow[yshift=0.7ex]{r}{j_!} & \Fun(\Span(\Fin_T), \Sp) \arrow[yshift=0.7ex]{r}{p} \arrow[yshift=-0.7ex]{l}{j^\ast} & \Fun^{\times}(\Span(\Fin_T), \Sp), \arrow[yshift=-0.7ex, hook]{l}{\iota}
\end{tikzcd}
\]
where $j_!$ denotes the functor given by forming the left Kan extension along $j : T^{\op} \to \Span(\Fin_T)$, and $p$ denotes a left adjoint of the inclusion $\iota$ which exists by virtue of~\cite[Proposition 6.5]{Bar17}). Let $u = j^\ast \iota$ denote the right adjoint of $L \simeq p j_!$. Let $X : T^{\op} \to \Sp$ denote a presheaf on $T$ with values in the $\infty$-category of spectra. If $t$ is an object of $T$, then there is a natural equivalence of spectra
\[
	(j^\ast j_! X)(t) \simeq \colim(\Span(\Fin_T)_{/t} \times_{\Span(\Fin_T)} T^{\op} \to T^{\op} \xrightarrow{X} \Sp),
\]
by virtue of the pointwise formula for the left Kan extension. Consequently, the canonical functor $T \to \Span(\Fin_T)$ induces a natural map $\colim_{(t' \to t) \in (T_{/t})^{\simeq}} X(t') \to (j^\ast j_! X)(t)$. The unit $\id \to \iota p$ induces a natural transformation $j^\ast j_! X \to j^\ast \iota p j_! X \simeq uLX$, and hence a natural map of spectra
\[
	\colim_{(t' \to t) \in (T_{/t})^{\simeq}} X(t') \to (j^\ast j_! X)(t) \to (uLX)(t)
\]
for every object $t$ of $T$. Additionally, this map is natural in $X$. Consequently, it suffices to prove the assertion of Proposition~\ref{proposition:tomDieck_orbital} for a presheaf on an orbital $\infty$-category $T$ with values in the $\infty$-category of connective spectra since the general case follows by passing to stabilizations. The forgetful functor $\Fun^\times(\Span(\Fin_T), \Spaces) \to \Fun(T^{\op}, \Spaces)$ admits a left adjoint which is canonically equivalent to the following composite of left adjoint functors of $\infty$-categories
\[
	\Fun(T^{\op}, \Spaces) \xrightarrow{\Free_{\E_\infty}} \Fun(T^{\op}, \Mon_{\E_\infty}(\Spaces)) \xrightarrow{L'} \Fun^{\times}(\Span(\Fin_T), \Mon_{\E_\infty}(\Spaces)),
\]
where we have used that the forgetful functor $\Mon_{\E_\infty}(\Spaces) \to \Spaces$ induces an equivalence
\[
	\Fun^{\times}(\Span(\Fin_T), \Mon_{\E_\infty}(\Spaces)) \simeq \Fun^{\times}(\Span(\Fin_T), \Spaces)
\]
since the $\infty$-category $\Span(\Fin_T)$ is semiadditive (cf.~\cite[Corollary 2.5]{GGN16}). Let $\Grp_{\E_\infty}(\Spaces)$ denote the full subcategory of $\Mon_{\E_\infty}(\Spaces)$ spanned by the grouplike $\E_\infty$-monoids, and recall that there is an equivalence $\Grp_{\E_\infty}(\Spaces) \simeq \Sp_{\geq 0}$ by~\cite{May06}. The following diagram of left adjoints commutes
\[
\begin{tikzcd}
	\Fun(T^{\op}, \Mon_{\E_\infty}(\Spaces)) \arrow{r}{L'} \arrow{d}{(-)^{\mathrm{grp}}} & \Fun^{\times}(\Span(\Fin_T), \Mon_{\E_\infty}(\Spaces)) \arrow{d}{(-)^{\mathrm{grp}}} \\
	\Fun(T^{\op}, \Grp_{\E_\infty}(\Spaces)) \arrow{r}{L} & \Fun^{\times}(\Span(\Fin_T), \Grp_{\E_\infty}(\Spaces))
\end{tikzcd}
\]
since the corresponding diagram of right adjoints commutes. The functors denoted by $(-)^{\mathrm{grp}}$ in the diagram above are induced by the group completion functor $\Mon_{\E_\infty}(\Spaces) \to \Grp_{\E_\infty}(\Spaces)$ which preserves finite products. If $X : T^{\op} \to \Sp_{\geq 0}$ is a presheaf on $T$ with values in the $\infty$-category of connective spectra, then we want to show that the map of connective spectra
\[
	\colim_{(t' \to t) \in (T_{/t})^{\simeq}} X(t') \to (LX)(t)
\]
is an equivalence for every object $t$ of $T$. Note that $X \simeq \colim_n \Sigma^{\infty-n}\Omega^{\infty-n} X$, and each presheaf $\Omega^{\infty-n} X$ is canonically a colimit of representable presheaves of the form $\Map_T(-, s)$, so we may assume that $X = \Sigma^\infty_+ \Map_T(-, s)$ for an object $s$ of $T$ since both the domain and codomain of the map above preserve colimits in $X$. Equivalently, we may assume that $X$ is given by the presheaf of grouplike $\E_\infty$-monoids on $T$ given by the formula $X = (\Free_{\E_\infty} \Map_T(-, s))^{\mathrm{grp}}$. There is a natural equivalence of functors of $\infty$-categories
\[
	LX \simeq (L' \Map_T(-, s))^{\mathrm{grp}} \simeq \Map_{\Span(\Fin_T)}(s, -)^{\mathrm{grp}} : \Span(\Fin_T) \to \Grp_{\E_\infty}(\Spaces)
\]
by the commutative square above, where the last equivalence follows by the Yoneda lemma. In this case, the map under consideration
\[
	F : \colim_{(t' \to t) \in (T_{/t})^{\simeq}} \Map_T(t', s) \to \Map_{\Span(\Fin)_T}(s, t).
\]
is induced by the map of spaces
\[
	\Map_T(t', s) \to \Map_{\Span(\Fin_T)}(s, t),
\]
determined by the construction $(t' \to s) \mapsto (s \leftarrow t' \rightarrow t)$ for every object $t' \to t$ of $(T_{/t})^{\simeq}$. Note that the map $F$ above is an equivalence, since the map of spaces
\[
	\Map_{\Span(\Fin_T)}(s, t) \to \colim_{(t' \to t) \in (T_{/t})^{\simeq}} \Map_T(t', s)
\]
determined by the construction $(s \leftarrow t' \to t) \mapsto (t' \to s)$ is an inverse. Using that $\Map_{\Span(\Fin_T)}(s, t)$ admits the structure of an $\E_\infty$-monoid since $\Span(\Fin_T)$ is semiadditive, and that both $\Free_{\E_\infty}$ and $(-)^{\mathrm{grp}}$ preserve colimits, we conclude that the map of grouplike $\E_\infty$-monoids
\[
	\colim_{(t' \to t) \in (T_{/t})^{\simeq}} (\Free_{\E_\infty}\Map_T(t', s))^{\mathrm{grp}} \to \Map_{\Span(\Fin_T)}(s, t)^{\mathrm{grp}}
\]
is an equivalence which proves the desired statement.
\end{proof}

\begin{example}
Let $G$ be a finite group, and consider the adjunction in~(\ref{equation:tomDieck_adjunction}). If $X : \Orb_G^{\op} \to \Sp$ is a spectral presheaf on the orbit category of $G$, then there is an equivalence
\[
	F(X)^G \simeq \colim_{G/H \to G/G \in ((\Orb_G)_{/(G/G)})^{\simeq}} X(G/H) \simeq \bigoplus_{(H)} (X^H)_{hW_G(H)},
\]
by Proposition~\ref{proposition:tomDieck_orbital}. This is precisely the classical Segal--tom Dieck splitting as discussed above. 
\end{example}

Recall that there is an adjunction of $\infty$-categories
\[
\begin{tikzcd}
	\Sp^{\Fr} \arrow[yshift=0.7ex]{r}{\Free} & \TCart \arrow[yshift=-0.7ex]{l}
\end{tikzcd}
\]
where $\mathrm{Free}$ denotes a left adjoint of the forgetful functor. As a consequence of Proposition~\ref{proposition:tomDieck_orbital}, we obtain an explicit description of $\Free$ which we will use in the proof of Theorem~\ref{theorem:comparison_TR}, where we show that the classical definition of $\TR$ coincides with the definition considered in \textsection\ref{subsection:TR} in the bounded below case. 

\begin{corollary} \label{corollary:explicit_Free}
If $X$ is a spectrum with Frobenius lifts, then there is an equivalence of spectra with $\T$-action
\[
	\Free(X) \simeq \bigoplus_{n \geq 1} X_{hC_n},
\]
The $k$th Verschiebung map of $\Free(X)$ is given by the canonical inclusion
\[
	\Free(X)_{hC_k} \simeq \bigoplus_{n \geq 1} X_{hC_{kn}} \hookrightarrow \bigoplus_{n \geq 1} X_{hC_n} \simeq \Free(X)
\]
for every integer $k \geq 1$, where the Verschiebung maps of $X$ are defined in Remark~\ref{remark:unwind_tcart}. 
\end{corollary}

\begin{proof}
The first assertion follows directly from Proposition~\ref{proposition:tomDieck_orbital}, so it remains to identify the Verschiebung maps of $\Free(X)$. We begin by identifying the $p$th Verschiebung map of $\Free(X)$ for a prime $p$. Let $\Free_p$ denote a left adjoint of the forgetful functor $\TCart_p \to \Sp_p^{\Fr}$. If $X$ is a $p$-typical spectrum with Frobenius lift, then there is an equivalence of spectra with $\T$-action
\[
	\Free_p(X) \simeq \bigoplus_{i \geq 0} X_{hC_{p^i}}
\]
by virtue of~\cite[Lemma 4.1]{AN20}, where the $p$th Verschiebung of $\Free_p(X)$ is given by the canonical inclusion under this identification. Since the following diagram commutes
\[
\begin{tikzcd}
	\TCart \arrow{r} \arrow{d} & \TCart_p \arrow{d} \\
	\Sp^{\Fr} \arrow{r} & \Sp_p^{\Fr}
\end{tikzcd}
\]
we obtain a canonical map of $p$-typical topological Cartier modules $\Free_p(X) \to \Free(X)$ for every spectrum $X$ with Frobenius lifts, such that the underlying map of spectra with $\T$-action is given by the canonical inclusion. This identifies the $p$th Verschiebung of $\Free(X)$ as wanted. The general case follows by prime decomposition. Indeed, if $k = p q$, then 
\[
	\Free(X)_{hC_{pq}} \simeq (\Free(X)_{hC_p})_{hC_q} \xrightarrow{(V_p)_{hC_q}} \Free(X)_{hC_q} \xrightarrow{V_q} \Free(X).
\]
\end{proof}

\begin{remark}
The functor $\Free$ restricts to a functor of $\infty$-categories
\[
	\Free : \Sp^{\Fr}_{\flat} \to \TCart_{\flat},
\]
where the decoration $(-)_{\flat}$ denotes the full subcategory spanned by those objects whose underlying spectrum is bounded below, since the functor $(-)_{hC_k}$ preserves connectivity for every integer $k \geq 1$. 
\end{remark}

\subsection{Genuine cyclotomic spectra and genuine TR} \label{subsection:genuine_cycsp}
We recall the construction of $\TR$ based on the notion of a genuine cyclotomic spectrum following Hesselholt--Madsen~\cite{HM97} and Blumberg--Mandell~\cite{BM16}. We establish an adjunction between the $\infty$-category of topological Cartier modules as introduced in \textsection\ref{subsection:tcart} and the $\infty$-category of genuine cyclotomic spectra whose right adjoint is given by $\TR$. Finally, we prove Theorem~\ref{theorem:comparison_TR} which asserts that the classical construction of $\TR$ agrees with the construction discussed in \textsection\ref{subsection:TR} in the bounded below case. 

\begin{definition} \label{definition:cycsp_gen}
The $\infty$-category of genuine cyclotomic spectra is defined by
\[
	\CycSp^{\gen} = \mathrm{lim}(F_{\Phi} : \B\N^{\times} \to \Cat_{\infty}) = (\Sp_\T)^{h\N^\times},
\]
where the multiplicative monoid $\N^{\times}$ acts on the $\infty$-category $\Sp_{\T}$ as in Example~\ref{example:geometric_fixedpoints}. 
\end{definition} 

\begin{remark}
Unwinding the definition, we see that an object of the $\infty$-category $\CycSp^{\gen}$ is given by a genuine $\T$-spectrum $X$ with compatible equivalences of genuine $\T$-spectra $X \simeq X^{\Phi C_k}$ for every integer $k \geq 1$. 
\end{remark}

The notion of a genuine cyclotomic spectrum was introduced by Hesselholt--Madsen~\cite{HM97} building on ideas introduced by Bökstedt--Hsiang--Madsen~\cite{BHM93}. The homotopy theory of genuine cyclotomic spectra was first studied by Blumberg--Mandell~\cite{BM16}, and later studied by Kaledin~\cite{Kal10,Kal13}, Barwick--Glasman~\cite{BG16}, Nikolaus--Scholze~\cite{NS18}, and Ayala--Mazel-Gee--Rozenblyum~\cite{AMGR17a,AMGR17b,AMGR17c}.

\begin{remark} \label{remark:main_result_NS}
There is a canonical sequence of functors of $\infty$-categories
\[
	\CycSp^{\gen} \to \Sp_{\T} \to \Sp^{\B\T}
\]
which extracts the underlying spectrum with $\T$-action of a genuine cyclotomic spectrum. Let $X$ denote a genuine cyclotomic spectrum. For every prime number $p$, the $\T$-equivariant map
\[
	X \to X^{\Phi C_p} \to X^{tC_p}
\]
exhibits the underlying spectrum with $\T$-action of $X$ as a cyclotomic spectrum in the sense of Nikolaus--Scholze~\cite{NS18}. The main result of Nikolaus--Scholze~\cite[Theorem II.6.9]{NS18} asserts that the construction above determines a functor of $\infty$-categories
\[
	\CycSp^{\gen} \to \CycSp
\]
which restricts to an equivalence on the full subcategories of those objects whose underlying spectrum is bounded below.
\end{remark}

\begin{remark}
If $R$ is a connective $\E_1$-ring, then Bökstedt--Hsiang--Madsen~\cite{BHM93} and Hesselholt--Madsen~\cite{HM97} prove that $\THH(R)$ admits the structure of a genuine cyclotomic spectrum by using the Bökstedt construction. Angeltveit--Blumberg--Gerhardt--Hill--Lawson--Mandell~\cite{ABGHLM14} construct the genuine cyclotomic structure on $\THH(R)$ using the Hill--Hopkins--Ravenel norm~\cite{HHR16}, and these constructions are equivalent by the work of Dotto--Malkiewich--Patchkoria--Sagave--Woo~\cite{DMPSW19}. Nikolaus--Scholze~\cite{NS18} construct a cyclotomic structure on $\THH(R)$ using the Tate-valued diagonal. 
\end{remark}

We recall the classical construction of $\TR$ following Blumberg--Mandell~\cite{BM16}.

\begin{construction} \label{construction:restriction_map}
Let $X$ be a genuine cyclotomic spectrum. For every pair of positive integer $(m, n)$ with $m = ln$, the restriction map $R_l : X^{C_m} \to X^{C_n}$ is the map of genuine $\T$-spectra given by
\[
	R_l : X^{C_m} \simeq (X^{C_l})^{C_n} \to (X^{\Phi C_l})^{C_n} \xrightarrow{\simeq} X^{C_n},
\]
where the final equivalence is induced by the genuine cyclotomic structure of $X$. Furthermore, the construction $n \mapsto X^{C_n}$ determines a functor of $\infty$-categories
\[
	X^{C_{(-)}} : (\N, \mathrm{div}) \to \Sp_{\T},
\]
where $(\N, \mathrm{div})$ denotes set of natural numbers regarded as a poset with respect to the divisibility relation, which means that there is a morphism $m \to n$ precisely if $n$ divides $m$. 
\end{construction}

\begin{definition}
Let $\TR^{\gen}$ denote the functor of $\infty$-categories
\[
	\TR^{\gen} : \CycSp^{\gen} \to \Sp_{\T}
\]
determined by the construction $X \mapsto \lim_{(\N, \mathrm{div})} X^{C_{(-)}}$. 
\end{definition}

\begin{remark} \label{remark:restrict_bounded_below}
The construction $X \mapsto \TR^{\gen}$ canonically refines to a functor of $\infty$-categories
\[
	\TR^{\gen} : \CycSp^{\gen} \to \TCart,
\]
where we have used Proposition~\ref{proposition:TCart_equiv_TCart_gen} and that the functor $(-)^{C_k}$ preserves limits for every $k \geq 1$. Additionally, the functor $\TR^{\gen} : \CycSp^{\gen} \to \TCart$ restricts to a functor
\[
	\TR^{\gen} : \CycSp_{\flat} \xrightarrow{\simeq} \CycSp^{\gen}_{\flat} \to \TCart_{\flat},
\]
where the decoration $(-)_{\flat}$ denotes the full subcategory spanned by those objects whose underlying spectrum is bounded below, and the first equivalence follows from Nikolaus--Scholze~\cite{NS18}.
\end{remark}

We recall an important construction due to Antieau--Nikolaus~\cite[Example 3.5]{AN20}. If $M$ is a $p$-typical topological Cartier module, then there is a cofiber sequence of spectra with $\T$-action
\[
	M_{hC_p} \xrightarrow{V} M \to M/V,
\]
and the cofiber $M/V$ canonically admits the structure of a $p$-typical cyclotomic spectrum. The construction $M \mapsto M/V$ determines a functor of $\infty$-categories
\[
	(-)/V : \TCart_p \to \CycSp_p,
\]
which admits a right adjoint given by $p$-typical $\TR$ by Antieau--Nikolaus~\cite[Proposition 3.17]{AN20} using Krause--Nikolaus~\cite[Proposition 10.3]{KN19}. We obtain an integral version of the functor $(-)/V$. Note that if we regard $M$ as a genuine $p$-typical topological Cartier module, then the underlying spectrum with $\T$-action of $M^{\Phi C_p}$ is equivalent to $M/V$ by the isotropy-separation sequence.

\begin{construction} \label{construction:the_other_restriction_map}
Let $M$ be a genuine topological Cartier module. For every pair of positive integers $(m, n)$ with $m = ln$, consider the map of genuine $\T$-spectra defined by{}
\[
	M^{\Phi C_n} \xrightarrow{\simeq} (M^{C_l})^{\Phi C_n} \to M^{\Phi C_m},
\]
where the first equivalence is induced by the genuine topological Cartier module structure of $M$. Furthermore, the construction $n \mapsto M^{\Phi C_n}$ determines a functor of $\infty$-categories
\[
	M^{\Phi C_{(-)}} : (\N, \mathrm{div})^{\op} \to \Sp_{\T}
\]
as in Construction~\ref{construction:restriction_map}. 
\end{construction}

\begin{remark}
The construction $M \mapsto \mathrm{colim}_{(\N, \mathrm{div})^{\op}} M^{\Phi C_{(-)}}$ refines to a functor of $\infty$-categories
\[
	\mathrm{L} : \TCart \to \CycSp^{\gen},
\]
where we have used Proposition~\ref{proposition:TCart_equiv_TCart_gen} and that the functor $(-)^{\Phi C_k}$ preserves colimits for every $k \geq 1$. Additionally, the functor $\mathrm{L} : \TCart \to \CycSp^{\gen}$ restricts to a functor of $\infty$-categories
\[
	\mathrm{L} : \TCart_{\flat} \to \CycSp^{\gen}_{\flat} \xrightarrow{\simeq} \CycSp_{\flat}
\]
as in Remark~\ref{remark:restrict_bounded_below}. 
\end{remark} 

\begin{remark} \label{remark:cofinality}
Let $(\N, \leq)$ denote the set of natural numbers regarded as a poset with respect to the usual ordering, which means that there is a morphism $m \to n$ in $(\N, \leq)$ precisely if $n \leq m$. The functor $(\N, \leq) \to (\N, \mathrm{div})$ determined by the construction $n \mapsto n!$ is initial\footnote{The author learned this observation from Markus Land.}. Consequently, if $X$ is a genuine cyclotomic spectum, then there is an equivalence of topological Cartier modules
\[
	\TR^{\gen}(X) \simeq \varprojlim(\cdots \to X^{C_{3!}} \to X^{C_{2!}} \to X),
\]
where the maps are given by $X^{C_{(n+1)!}} \simeq (X^{C_{n+1}})^{C_{n!}} \to (X^{\Phi C_{n+1}})^{C_{n!}} \simeq X^{C_{n!}}$. On the other hand, if $M$ is a topological Cartier module, then there is an equivalence of genuine cyclotomic spectra
\[
	\mathrm{L}(M) \simeq \varinjlim(M \to M^{\Phi C_{2!}} \to M^{\Phi C_{3!}} \to \cdots),
\]
where the maps are given by $M^{\Phi C_{n!}} \simeq (M^{C_{n+1}})^{\Phi C_{n!}} \to (M^{\Phi C_{n+1}})^{\Phi C_{n!}} \simeq M^{\Phi C_{(n+1)!}}$. 
\end{remark}

Using Remark~\ref{remark:cofinality}, we show that the functors $\mathrm{L}$ and $\TR^{\gen}$ form an adjunction. We adapt the proof given by Antieau--Nikolaus~\cite[Proposition 5.2]{AN20} to the integral situation. We begin by reviewing the $\infty$-category of coalgebras for a family of endofunctors following~\cite[Section II.5]{NS18}. If $F : \cat{C} \to \cat{C}$ is a functor, then the $\infty$-category $\coAlg_F(\cat{C})$ of $F$-coalgebras in $\cat{C}$ is defined by
\[
\begin{tikzcd}
	\coAlg_F(\cat{C}) = \mathrm{LEq} \big(
	\cat{C} \arrow[yshift=0.7ex]{r}{\id} \arrow[yshift=-0.7ex, swap]{r}{F} & \cat{C}
	\big),
\end{tikzcd}
\]
where $\mathrm{LEq}(-)$ denotes the lax equalizer (see~\cite[Definition II.1.4]{NS18}). Unwinding the definition, we see that an object of $\coAlg_F(\cat{C})$  consists of an object $X$ of $\cat{C}$ together with a morphism $X \to FX$. Let $\{F_k\}_{k \geq 1}$ be a family of commuting endofunctors of $\cat{C}$. There is a functor of $\infty$-categories
\[
	\coAlg_{F_1}(\cat{C}) \to \coAlg_{F_1}(\cat{C})
\]
determined by the assignment $X \mapsto F_2 X$, where $F_2 X$ is considered as an $F_1$-coalgebra using the map $F_2 X \to F_2 F_1 X \simeq F_1 F_2 X$ (cf.~\cite[Construction II.5.2]{NS18}). The $\infty$-category of coalgebras for the family of commuting endofunctor $\{F_k\}_{k \geq 1}$ of $\cat{C}$ is defined by
\[
	\coAlg_{\{F_k\}}(\cat{C}) = \varprojlim_n \coAlg_{\{F_1, \ldots, F_n\}}(\cat{C}),
\]
where $\coAlg_{\{F_1, \ldots, F_n\}}(\cat{C}) = \coAlg_{F_n}(\coAlg_{\{F_1, \ldots, F_{n-1}\}}(\cat{C}))$. Unwinding the definition, we see that an object of the $\infty$-category $\coAlg_{\{F_k\}}(\cat{C})$ consists of an object $X$ of $\cat{C}$ together with compatible maps $X \to F_k X$ for every $k \geq 1$. Let $\Fix_{\{F_k\}}(\cat{C})$ denote the full subcategory of $\coAlg_{\{F_k\}}(\cat{C})$ spanned by those objects where the map $X \to F_k X$ is an equivalence for each $k \geq 1$. Dually, there is an $\infty$-category of algebras for a family of commuting endofunctors $\{G_k\}_{k \geq 1}$ of $\cat{C}$ which we will denote by $\Alg_{\{G_k\}}(\cat{C})$. An object of $\Alg_{\{G_k\}}(\cat{C})$ consists of an object of $\cat{C}$ together with compatible maps $G_k X \to X$ for every $k \geq 1$. We have the following result: 

\begin{proposition} \label{proposition:adjunction_TR_L}
There is an adjunction of $\infty$-categories
\[
\begin{tikzcd}[column sep = large]
	\TCart \arrow[yshift=0.7ex]{r}{\mathrm{L}} & \CycSp^{\gen} \arrow[yshift=-0.7ex]{l}{\TR^{\gen}} 
\end{tikzcd}
\]
\end{proposition}

\begin{proof}
In the following, we will implicitly employ the equivalence $\TCart \simeq \TCart^{\gen}$ established in Proposition~\ref{proposition:TCart_equiv_TCart_gen}. Let $X$ be a genuine cyclotomic spectrum and let $M$ be a topological Cartier module. There is a natural transformation $\eta : \id \to \TR^{\gen} \circ \mathrm{L}$ incuced by the canonical map
\[
	M \simeq \mathrm{lim}_{m \in \N} M^{C_{m!}} \to \mathrm{lim}_{m \in \N}(\mathrm{colim}_{n \in \N} M^{\Phi C_{n!}})^{C_{m!}},
\]
by virtue of Remark~\ref{remark:cofinality}, and we show that $\eta$ induces an equivalence of spaces
\[
	\Map_{\CycSp^{\gen}}(\mathrm{L}(M), X) \to \Map_{\TCart}(M, \TR^{\gen}(X)).
\]
There is an equivalence of $\infty$-categories $\CycSp^{\gen} \simeq \Fix_{\{(-)^{\Phi C_{k!}}\}}(\Sp_{\T})$ by~\cite[Lemma II.5.9]{NS18}, so we obtain a fully faithful functor of $\infty$-categories
\[
	\CycSp^{\gen} \hookrightarrow \coAlg_{\{(-)^{\Phi C_{k!}}\}}(\Sp_{\T})
\]
which admits a left adjoint  given by $X \mapsto \mathrm{colim}_{n \in \N} X^{\Phi C_{n!}}$ by~\cite[Section II.5]{NS18}. It follows that
\[
	\Map_{\CycSp^{\gen}}(\mathrm{L}(M), X) \simeq \Map_{\coAlg_{\{(-)^{\Phi C_{k!}}\}}(\Sp_{\T})}(M, X),	
\]
where $M$ is regarded as a coalgebra using the map $M \simeq M^{C_{k!}} \to M^{\Phi C_{k!}}$ for each $k \geq 1$. Similarly, we find that there is a fully faithful functor of $\infty$-categories
\[
	\TCart \hookrightarrow \Alg_{\{(-)^{C_{k!}}\}}(\Sp_{\T})
\]
which admits a right adjoint given by $M \mapsto \mathrm{lim}_{m \in \N} M^{C_{m!}}$. We conclude that
\[
	\Map_{\TCart}(M, \TR^{\gen}(X)) \simeq \Map_{\Alg_{\{(-)^{C_{k!}}\}}(\Sp_{\T})}(M, X),
\]
where $X$ is regarded as an algebra using the map $X^{C_{k!}} \to X^{\Phi C_{k!}} \simeq X$ for each $k \geq 1$. In other words, we have reduced to showing that $\eta : \id \to \TR^{\gen} \circ \mathrm{L}$ induces a natural equivalence of spaces
\[
	\Map_{\coAlg_{\{(-)^{\Phi C_{k!}}\}}(\Sp_{\T})}(M, X) \simeq \Map_{\Alg_{\{(-)^{C_{k!}}\}}(\Sp_{\T})}(M, X).
\]
Using~\cite[Proposition II.1.5]{NS18}, we have that $\Map_{\coAlg_{\{(-)^{\Phi C_{1!}}, \ldots, (-)^{\Phi C_{(k+1)!}}\}}(\Sp_{\T})}(M, X)$ is given by the following equalizer
\[
\begin{tikzcd}
	\Map_{\coAlg_{\{(-)^{\Phi C_{1!}}, \ldots, (-)^{\Phi C_{k!}}\}}(\Sp_{\T})}(M, X) \arrow[yshift=0.7ex]{r} \arrow[yshift=-0.7ex]{r} & \Map_{\coAlg_{\{(-)^{\Phi C_{1!}}, \ldots, (-)^{\Phi C_{k!}}\}}(\Sp_{\T})}(M, X^{\Phi C_{(k+1)!}}),
\end{tikzcd}
\]
where $M \to X$ is carried to $M \to X \simeq X^{\Phi C_{(k+1)!}}$ by the top map, and to $M \simeq M^{C_{(k+1)!}} \to M^{\Phi C_{(k+1)!}}$ by the bottom map. Similarly, we have that $\Map_{\Alg_{\{(-)^{C_{1!}}, \ldots, (-)^{C_{(k+1)!}}\}}(\Sp_{\T})}(M, X)$ is given by the following equalizer
\[
\begin{tikzcd}
	\Map_{\Alg_{\{(-)^{C_{1!}}, \ldots, (-)^{C_{k!}}\}}(\Sp_{\T})}(M, X) \arrow[yshift=0.7ex]{r} \arrow[yshift=-0.7ex]{r} & \Map_{\coAlg_{\{(-)^{C_{1!}}, \ldots, (-)^{C_{k!}}\}}(\Sp_{\T})}(M^{C_{(k+1)!}}, X),
\end{tikzcd}
\]
where $M \to X$ is carried to $M^{C_{(k+1)!}} \to M \to X$ by the top map, and to $M^{C_{(k+1)!}} \to M^{\Phi C_{(k+1)!}} \simeq M$ by the bottom map. By induction, we conclude that these two equalizers are equivalent using the compatible equivalences $M^{C_k} \simeq M$ and $X^{\Phi C_k} \simeq X$ for each $k \geq 1$. 
\end{proof}

We prove the main result of this section. As mentioned previously, this result is inspired by the result of Blumberg--Mandell~\cite{BM16} which asserts that $\TR^{\gen}$ is corepresentable by the reduced topological Hochschild homology $\widetilde{\THH}(\S[t])$ as a functor defined on the homotopy category of the $\infty$-category of genuine cyclotomic spectra with values in the homotopy category of spectra (see~\cite[Theorem 6.12]{BM16}).

\begin{theorem} \label{theorem:comparison_TR}
There is a natural equivalence of spectra with Frobenius lifts
\[
	\TR^{\gen}(X) \simeq \map_{\CycSp}(\widetilde{\THH}(\S[t]), X)
\]
for every cyclotomic spectrum $X$ whose underlying spectrum is bounded below. 
\end{theorem}

The proof of Theorem~\ref{theorem:comparison_TR} relies on a genuine version of the Tate orbit lemma~\cite[Lemma I.2.1]{NS18} obtained by Antieau--Nikolaus~\cite{AN20}.

\begin{lemma} \label{lemma:genuine_tate_orbit}
If $Y$ is a genuine $\T$-spectrum, then the canonical map
\[
	(Y^{C_p})^{\Phi C_p} \to Y^{\Phi C_{p^2}}
\]
is an equivalence in the $\infty$-category $\Sp_{\T}$ for every prime $p$. 
\end{lemma}

\begin{proof}
See Antieau--Nikolaus~\cite[Lemma 5.3]{AN20}. 
\end{proof}

\begin{proof}[Proof of Theorem~\ref{theorem:comparison_TR}]
First note that the composite 
\[
	\Sp^{\Fr}_{\flat} \xrightarrow{\Free} \TCart_{\flat} \xrightarrow{\mathrm{L}} \CycSp^{\gen}_{\flat} \xrightarrow{\simeq} \CycSp_{\flat}
\]
is equivalent to the canonical functor $\Sp^{\Fr}_{\flat} \to \CycSp_{\flat}$ as a consequence of Corollary~\ref{corollary:explicit_Free} combined with the formula for the functor $\mathrm{L}$. Indeed, we have that
\[
	\mathrm{L}\Free(X) \simeq \varinjlim\Big( \bigoplus_{k \geq 1} X_{hC_k} \xrightarrow{\mathrm{pr}} \bigoplus_{\substack{k \geq 1 \\ 2 \nmid k}} X_{hC_k} \xrightarrow{\mathrm{pr}} \bigoplus_{\substack{k \geq 1 \\ 2, 3 \nmid k}} X_{hC_k} \xrightarrow{\mathrm{pr}} \cdots\Big) \simeq X,
\]
by Lemma~\ref{lemma:genuine_tate_orbit}, thus it follows that $\mathrm{L}\Free(X) \simeq X$ as objects of the $\infty$-category $\CycSp^{\gen}$. Now, since every functor in the composite above is a left adjoint, we conclude that the composite
\[
	\CycSp_{\flat} \xrightarrow{\simeq} \CycSp^{\gen}_{\flat} \xrightarrow{\TR^{\gen}} \TCart_{\flat} \to \Sp^{\Fr}_{\flat}
\]
determines a right adjoint of the canonical functor $\Sp^{\Fr}_{\flat} \to \CycSp_{\flat}$, which shows the wanted by virtue of Proposition~\ref{proposition:TR_right_adjoint}. 
\end{proof}

\section{Applications to curves on K-theory} \label{section:application}
We discuss an application of Theorem~\ref{theorem:comparison_TR}. Specifically, we prove that $\TR$ evaluated on a connective $\E_1$-ring $R$ admits a description in terms of the spectrum of curves on the algebraic $\K$-theory of $R$, extending work of Hesselholt~\cite{Hes96} and Betley--Schlichtkrull~\cite{BS05}.  

\subsection{Topological Hochschild homology of truncated polynomial algebras} \label{subsection:truncated}
In this section, we obtain a convenient description of the cyclotomic structure on $\varprojlim \widetilde{\THH}(\S[t]/t^n)$ which will be instrumental in \textsection\ref{subsection:curves}. Concretely, we compute the Hochschild homology of $\Z[t]/t^n$ as a commutative differential graded algebra, and deduce a connectivity estimate for $\THH(\S[t]/t^n)$. Similar computations have previously been obtained by Hesselholt--Madsen~\cite{HM97,HM97b} building on the work of~\cite{GGRSV91,LL92}. We refer the reader to Speirs~\cite{Spe20} for a summary of these computations. We will begin by reviewing the notion of a graded object in a symmetric monoidal $\infty$-category. 

\begin{notation} \label{notation:graded}
Let $\Z_{\geq 0}^{\ds}$ denote the set of non-negative integers regarded as a discrete category, and let $\cat{C}$ denote a presentable symmetric monoidal $\infty$-category.

\begin{enumerate}[leftmargin=2em, topsep=7pt, itemsep=7pt]
	\item The $\infty$-category of graded objects of $\cat{C}$ is defined by
	\[
		\Gr(\cat{C}) = \Fun(\Z_{\geq 0}^{\ds}, \cat{C}),
	\]
	which we regard as a symmetric monoidal $\infty$-category using the Day convolution symmetric monoidal structure coming from the symmetric monoidal structure on $\Z_{\geq 0}^{\ds}$ given by multiplication. We will denote an object $X$ of $\Gr(\cat{C})$ by $\{X_i\}_{i \geq 0}$. For every integer $n \geq 0$, the inclusion $\{n\} \hookrightarrow \Z_{\geq 0}^{\ds}$ induces an evaluation functor
	\[
		\mathrm{ev}_n : \Gr(\cat{C}) \to \cat{C}
	\]
	which is determined by the construction $\{X_i\}_{i \geq 0} \mapsto X_n$.  

	\item The projection $\Z_{\geq 0}^{\ds} \to \{0\}$ induces a functor $\cat{C} \to \Gr(\cat{C})$ which admits a left adjoint 
	\[
		\mathrm{und} : \Gr(\cat{C}) \to \cat{C}
	\]
	determined by forming the left Kan extension along the projection $\Z_{\geq 0}^{\ds} \to \{0\}$. Concretely, the functor $\mathrm{und} : \Gr(\cat{C}) \to \cat{C}$ is determined by the construction $\{X_i\}_{i \geq 0} \mapsto \bigoplus_{i \geq 0} X_i$. The functor $\cat{C} \to \Gr(\cat{C})$ canonically admits a symmetric monoidal structure since the projection $\Z_{\geq 0}^{\ds} \to \{0\}$ is a map of commutative monoids, so the left adjoint $\mathrm{und} : \Gr(\cat{C}) \to \cat{C}$ canonically refines to a symmetric monoidal functor of $\infty$-categories (see~\cite[Corollary 3.8]{Nik16}). 
\end{enumerate}
\end{notation}

\begin{remark}
If $\{X_i\}_{i \geq 0}$ is a graded spectrum, then the underlying spectrum $X = \mathrm{und}(\{X_i\}_{i \geq 0})$ of $\{X_i\}_{i \geq 0}$ carries an additional grading on its homotopy groups using the formula
\[
	\pi_\ast(X) \simeq \bigoplus_{i \geq 0} \pi_\ast(X_i). 
\]
If $x \in \pi_\ast(X_i)$, then we will write $w(x) = i$, and think of this as the horizontal grading direction. See Example~\ref{example:picture_gradings} and the discussion following the proof of Proposition~\ref{proposition:HH_truncated_dga}. 
\end{remark}

A pointed monoid is a monoid object in the category of pointed sets equipped with the smash product symmetric monoidal structure. For every $n \in \N \cup \{\infty\}$, we will let $\Pi_n = \{0, 1, t, \ldots, t^{n-1}\}$ denote the pointed monoid with $0$ as basepoint and whose multiplication is determined by $t^n = 0$. In the following, we will regard the pointed monoid $\Pi_n$ as an object of the $\infty$-category $\CAlg(\Spaces_\ast)$ of pointed $\E_\infty$-monoids, where the $\infty$-category $\Spaces_\ast$ of pointed spaces is equipped with the smash product symmetric monoidal structure. The $\infty$-category $\CAlg(\Spaces_\ast)$ inherits the structure of a symmetric monoidal $\infty$-category whose unit is given by the pointed monoid $\{0, 1\}$.

\begin{example} \label{example:graded_Pi_infty}
In this example, we endow the underlying pointed space of $\Pi_\infty$ with two distinct graded pointed $\E_\infty$-monoid structures:
\begin{enumerate}[leftmargin=2em, topsep=7pt, itemsep=7pt]
	\item The functor $\Z_{\geq 0}^{\ds} \to \Spaces_\ast$ defined by $i \mapsto \{0, t^i\}$ endows the underlying pointed space of $\Pi_\infty$ with the structure of an object of the $\infty$-category $\Gr(\Spaces_\ast)$ of graded pointed spaces with $t$ in grading degree $1$, and we will denote this object by $\Pi_\infty^{w(t) = 1}$. We have that $\Pi_\infty^{w(t)=1}$ canonically refines to an object of the $\infty$-category $\CAlg(\Gr\Spaces_\ast)$ since the grading is compatible with the monoid structure on $\Pi_\infty$.

	\item The functor $\Z_{\geq 0}^{\ds} \to \Spaces_\ast$ defined by $i \mapsto \{0, t^j\}$ if $i = jn$ and by $i \mapsto \{0\}$ if $n$ does not divide $i$, endows the underlying pointed space of $\Pi_\infty$ with the structure of an object of the $\infty$-category $\Gr(\Spaces_\ast)$ with $t$ in grading degree $n$, and we will similarly denote this object by $\Pi_\infty^{w(t) = n}$. As before, we have that $\Pi_\infty^{w(t) = n}$ canonically  refines to an object of $\CAlg(\Gr\Spaces_\ast)$.
\end{enumerate}
\end{example}

For every integer $n \geq 1$, the assignment $t \mapsto t^n$ determines a map  of pointed monoids $\Pi_\infty \to \Pi_\infty$ which canonically refines to a map of $\E_\infty$-algebras $\Pi_\infty^{w(t) = n} \to \Pi_\infty^{w(t) = 1}$ in graded pointed spaces. We have the following result which will play an important role in the following. 

\begin{lemma} \label{lemma:pushout_pointed_monoids}
The following square is a pushout of $\E_\infty$-algebras in graded pointed spaces
\[
\begin{tikzcd}
	\Pi_\infty^{w(t)=n} \arrow{r}{t \mapsto t^n} \arrow{d}{t \mapsto 0} & \Pi_\infty^{w(t)=1} \arrow{d} \\
	\{0,1\} \arrow{r} & \Pi_n^{w(t)=1}
\end{tikzcd}
\]
\end{lemma}

\begin{proof}
We first prove the following general assertion: Assume that $M$ is a $\Pi_\infty$-module in $\Spaces_\ast$, and let $M'$ denote the cofiber of the endomorphism of $M$ given by multiplication by $t$. Then the canonical map $M \to M'$ induces an equivalence of $\Pi_\infty$-modules
\[
	M \otimes_{\Pi_\infty} \{0,1\} \to M'.
\]
We may assume that $M = \Pi_\infty$ since the unit $\Pi_\infty$ generates the $\infty$-category $\Mod_{\Pi_\infty}(\Spaces_\ast)$ under colimits, and in this case the assertion is true by inspection. To prove the assertion of the lemma, we note that it suffices to prove that the diagram above is a pushout of $\E_\infty$-algebras in $\Spaces_\ast$ after applying the forgetful functor $\CAlg(\Gr\Spaces_\ast) \to \CAlg(\Spaces_\ast)$. In this case, the pushout is given by $\Pi_\infty \otimes_{\Pi_\infty} \{0,1\}$, where the $\Pi_\infty$-module structure on $\Pi_\infty$ is obtained by restriction of scalars along the map $\Pi_\infty \to \Pi_\infty$ in $\CAlg(\Spaces_\ast)$ determined by $t \mapsto t^n$. Consequently, we have that $\Pi_\infty \otimes_{\Pi_\infty} \{0,1\}$ is equivalent to the cofiber of the map $\Pi_\infty \to \Pi_\infty$ given by multiplication by $t^n$, which shows the wanted.
\end{proof}

There is a functor of $\infty$-categories $\Sigma^\infty : \Gr(\Spaces_\ast) \to \GrSp$ obtained by composition with the reduced suspension spectrum functor. This functor canonically refines to a functor of $\infty$-categories
\[
	\Sigma^\infty : \CAlg(\Gr\Spaces_\ast) \to \CAlg(\GrSp)
\]
since $\Sigma^\infty : \Spaces_\ast \to \Sp$ and thus $\Sigma^\infty : \Gr\Spaces_\ast \to \GrSp$ admits the structure of a symmetric monoidal functor. It is a consequence of Lemma~\ref{lemma:pushout_pointed_monoids} that the following square is a pushout of $\E_\infty$-algebras in the $\infty$-category of graded spectra
\[
\begin{tikzcd}
	\S[s] \arrow{r}{s \mapsto t^n} \arrow{d}{s \mapsto 0} & \S[t] \arrow{d} \\
	\S \arrow{r} & \S[t]/t^n
\end{tikzcd}
\]
where $s$ is in grading degree $n$ and $t$ is in grading degree $1$. The $\E_\infty$-algebra $\S[t]/t^n$ in $\GrSp$ is defined by $\S[t]/t^n = \Sigma^\infty \Pi_n^{w(t)=1}$. We conclude that the diagram obtained by applying the functor $\mathrm{und} : \GrSp \to \Sp$ to the diagram above is a pushout of $\E_\infty$-rings.

\begin{example} \label{example:truncated_poly_sphere}
By the discussion above, there is an equivalence of $\E_\infty$-rings
\[
	\S[t]/t^n = \S \otimes_{\S[s]} \S[t],
\]
for every integer $n \geq 1$, where $\S[t]$ is a module over $\S[s]$ by restriction of scalars along the map of $\E_\infty$-rings $\S[s] \to \S[t]$ given by $s \mapsto t^n$. There is an isomorphism of commutative rings
\[
	\pi_0(\S[t]/t^n) \simeq \pi_0(\S) \otimes_{\pi_0(\S[s])} \pi_0(\S[t]) \simeq \Z \otimes_{\Z[s]} \Z[t] \simeq \Z[t]/t^n
\]
by virtue of~\cite[Corollary 7.2.1.23]{Lur17} since both $\S$ and $\S[t]$ are connective.
\end{example}

Presently, we discuss a graded refinement of cyclotomic spectra following~\cite{AMMN20}. The construction $\{X_i\}_i \mapsto \{X_{pi}^{tC_p}\}_i$ determines an endofunctor $F_p$ of the $\infty$-category $\GrSp^{\B\T}$ of graded spectra with $\T$-action for every prime $p$, where $X_{pi}^{tC_p}$ is equipped with the residual $\T/C_p \simeq \T$-action.

\begin{definition} \label{definition:grcycsp}
The $\infty$-category of graded cyclotomic spectra is defined as the pullback
\[
\begin{tikzcd}[column sep = large]
	\GrCycSp \arrow{r} \arrow{d} & \displaystyle\prod_{p} (\GrSp^{\B\T})^{\Delta^1} \arrow{d}{(\mathrm{ev}_0, \mathrm{ev}_1)}\\
	\GrSp^{\B\T} \arrow{r}{(\id, F_p)} & \displaystyle\prod_{p} (\GrSp^{\B\T} \times \GrSp^{\B\T})
\end{tikzcd}
\]
\end{definition} 

It follows from Definition~\ref{definition:grcycsp} that a graded cyclotomic spectrum is given by a graded spectrum with $\T$-action $\{X_i\}_{\geq 0}$ together with a $\T$-equivariant map
\[
	\varphi_{p, i} : X_i \to X_{pi}^{tC_p}
\]
for every prime number $p$ and integer $i \geq 0$, where the target carries the residual $\T/C_p \simeq \T$-action. Informally, the cyclotomic Frobenius multiplies the grading degree by $p$. If $R$ is an $\E_1$-algebra in the $\infty$-category of graded spectra, then there is a functor of $\infty$-categories
\[
	\THH^{\mathrm{gr}} : \Alg(\GrSp) \to \GrCycSp
\]
obtained by applying the cyclic bar construction in the $\infty$-category $\GrSp$ of graded spectra. We will refer to the functor $\THH^{\mathrm{gr}}$ as graded topological Hochschild homology, and refer the reader to~\cite[Appendix A]{AMMN20} for the details of this construction. As in Notation~\ref{notation:graded}, there is a functor
\[
	\GrCycSp \to \CycSp
\]
determined by the construction $\{X_i\}_{i \geq 0} \mapsto \bigoplus_{i \geq 0} X_i$, and this functor preserves colimits. We have that the following diagram of $\infty$-categories commutes
\[
\begin{tikzcd}
	\Alg(\GrSp) \arrow{d}{\mathrm{und}} \arrow{r}{\THH^{\gr}} & \GrCycSp \arrow{d}{\mathrm{und}} \\
	\Alg \arrow{r}{\THH} & \CycSp
\end{tikzcd}
\]
where we have used that the underlying functor $\GrSp \to \Sp$ of Notation~\ref{notation:graded} admits a canonical symmetric monoidal structure.

\begin{definition} \label{definition:graded_pieces}
Let $\{R_i\}_{i \geq 0}$ denote an $\E_1$-algebra in $\GrSp$ with $R \simeq \mathrm{und}(\{R_i\}_{i \geq 0})$, and define
\[
	\THH(R)_i = \mathrm{ev}_i \THH^{\gr}(\{R_i\}_{i \geq 0})
\]
for every integer $i \geq 0$. By definition, there is an equivalence of cyclotomic spectra
\[
	\THH(R) \simeq \mathrm{und} \THH^{\gr}(\{R_i\}_{i \geq 0}) \simeq \bigoplus_{i \geq 0} \THH(R)_i.
\]
\end{definition}

\begin{example} \label{example:picture_gradings}
There is an isomorphism of graded rings
\[
	\THH_\ast(\S[t]) \simeq \S_\ast[t, dt]/(dt^2 = \eta t dt),
\]
where $\eta \in \pi_1 \S$ denotes the Hopf element and $|dt| = 1$. We may regard $\THH_\ast(\S[t])$ as a graded abelian group with $w(t) = w(dt) = 1$ using that the underlying $\E_\infty$-ring of $\THH^{\mathrm{gr}}(\Sigma^\infty \Pi_\infty^{w(t)=1})$ is equivalent to $\THH(\S[t])$. We have the following picture:
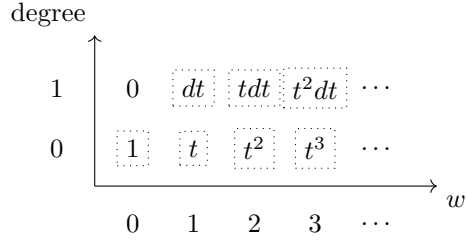
\[
\begin{tikzpicture}
	\node[draw, dotted] at (0, 0) {$1$};
	\node[draw, dotted] at (0.8, 0) {$t$};
	\node[draw, dotted] at (1.6, 0) {$t^2$};
	\node[draw, dotted] at (2.4, 0) {$t^3$};
	\node at (3.2, 0) {$\ldots$};

	\node at (0, 0.8) {$0$};
	\node[draw, dotted] at (0.8, 0.8) {$dt$};
	\node[draw, dotted] at (1.6, 0.8) {$t dt$};
	\node[draw, dotted] at (2.4, 0.8) {$t^2 dt$};
	\node at (3.2, 0.8) {$\ldots$};

	\node at (0, -1) {$0$};
	\node at (0.8, -1) {$1$};
	\node at (1.6, -1) {$2$};
	\node at (2.4, -1) {$3$};
	\node at (3.2, -1) {$\ldots$};

	\node at (-1, 0) {$0$};
	\node at (-1, 0.8) {$1$};

	\draw[->] (-0.5, -0.5) -- (4, -0.5) node[anchor=north west] {$w$};
	\draw[->] (-0.5, -0.5) -- (-0.5, 1.5) node[anchor=south east] {degree};
\end{tikzpicture}
\]
Alternatively, we may regard $\THH_\ast(\S[t])$ as a graded abelian group with $w(t) = w(dt) = n$, and in this case we have the following picture:
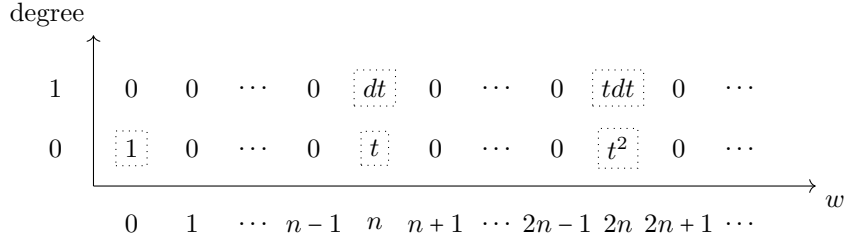
\[
\begin{tikzpicture}
	\node[draw, dotted] at (0, 0) {$1$};
	\node at (0.8, 0) {$0$};
	\node at (1.6, 0) {$\ldots$};
	\node at (2.4, 0) {$0$};
	\node[draw, dotted] at (3.2, 0) {$t$};
	\node at (4, 0) {$0$};
	\node at (4.8, 0) {$\ldots$};
	\node at (5.6, 0) {$0$};
	\node[draw, dotted] at (6.4, 0) {$t^2$};
	\node at (7.2, 0) {$0$};
	\node at (8, 0) {$\ldots$};

	\node at (0, 0.8) {$0$};
	\node at (0.8, 0.8) {$0$};
	\node at (1.6, 0.8) {$\ldots$};
	\node at (2.4, 0.8) {$0$};
	\node[draw, dotted] at (3.2, 0.8) {$dt$};
	\node at (4, 0.8) {$0$};
	\node at (4.8, 0.8) {$\ldots$};
	\node at (5.6, 0.8) {$0$};
	\node[draw, dotted] at (6.4, 0.8) {$t dt$};
	\node at (7.2, 0.8) {$0$};
	\node at (8, 0.8) {$\ldots$};

	\node at (0, -1) {$0$};
	\node at (0.8, -1) {$1$};
	\node at (1.6, -1) {$\ldots$};
	\node at (2.4, -1) {$n-1$};
	\node at (3.2, -1) {$n$};
	\node at (4, -1) {$n+1$};
	\node at (4.8, -1) {$\ldots$};
	\node at (5.6, -1) {$2n-1$};
	\node at (6.4, -1) {$2n$};
	\node at (7.2, -1) {$2n+1$};
	\node at (8, -1) {$\ldots$};

	\node at (-1, 0) {$0$};
	\node at (-1, 0.8) {$1$};

	\draw[->] (-0.5, -0.5) -- (9, -0.5) node[anchor=north west] {$w$};
	\draw[->] (-0.5, -0.5) -- (-0.5, 1.5) node[anchor=south east] {degree};
\end{tikzpicture}
\]
\end{example}

As in Definition~\ref{definition:graded_pieces} above, there is an equivalence of cyclotomic spectra
\[
	\THH(\S[t]/t^n) \simeq \mathrm{und} \THH^{\gr}(\Sigma^\infty \Pi_n^{w(t) = 1}) \simeq \bigoplus_{i \geq 0} \THH(\S[t]/t^n)_i.
\]
Hesselholt--Madsen~\cite{HM97b} determine the $\T$-equivariant homotopy type of $\THH(\S[t]/t^n)_i$ for every $i \geq 0$ by analyzing the facet structure of regular cyclic polytopes. Using Example~\ref{example:picture_gradings} above, we obtain the following special case of the result of Hesselholt--Madsen. 

\begin{proposition} \label{proposition:identification_THH_i}
For every integer $n \geq 1$, there is an equivalence of spectra with $\T$-action
\[
	\THH(\S[t]/t^n) \simeq \bigoplus_{i \geq 0} \THH(\S[t]/t^n)_i,
\]
where $\THH(\S[t]/t^n)_0 \simeq \S$ and $\THH(\S[t]/t^n)_i \simeq \Sigma^\infty_+(S^1/C_i)$ for $1 \leq i \leq n-1$.
\end{proposition}

\begin{proof}
It remains to show that $\THH(\S[t])_0 \simeq \S$ and $\THH(\S[t]/t^n)_i \simeq \Sigma^\infty_+(S^1/C_i)$ for $1 \leq i \leq n-1$ by the discussion above. There is an equivalence of cyclotomic spectra
\[
	\THH(\S[t]) \simeq \mathrm{und} \THH^{\gr}(\Sigma^\infty \Pi_\infty^{w(t) = 1}) \simeq \bigoplus_{i \geq 0} \THH(\S[t])_i,
\]
where $\THH(\S[t])_0 \simeq \S$ and $\THH(\S[t])_i \simeq \Sigma^\infty_+(S^1/C_i)$ for every $i \geq 1$, thus we want to prove that there is an equivalence of spectra with $\T$-action
\[
	\THH(\S[t]/t^n)_i \simeq \THH(\S[t])_i
\]
for every $0 \leq i \leq n-1$. Note that the square
\[
\begin{tikzcd}
	\THH^{\mathrm{gr}}(\S[s]) \arrow{r}{s \mapsto t^n} \arrow{d}{s \mapsto 0} & \THH^{\mathrm{gr}}(\S[t]) \arrow{d} \\
	\S \arrow{r} & \THH^{\mathrm{gr}}(\S[t]/t^n)
\end{tikzcd}
\]
is a pushout in the $\infty$-category of graded cyclotomic spectra, where $w(s) = n$ and $w(t) = 1$. Indeed, it suffices to prove that the square is a pushout in the $\infty$-category of cyclotomic spectra after applying the underlying functor $\GrCycSp \to \CycSp$. This is a consequence of Lemma~\ref{lemma:pushout_pointed_monoids} since both $\mathrm{und} : \Alg(\GrSp) \to \Alg$ and $\THH : \Alg \to \CycSp$ preserve pushouts. Using Example~\ref{example:picture_gradings}, we conclude that there is an equivalence of spectra with $\T$-action
\[
	\THH(\S[t]/t^n)_i \simeq \THH(\S[t])_i
\]
for every $0 \leq i \leq n-1$ as wanted. 
\end{proof}

We will not need to determine the $\T$-equivariant homotopy type of $\THH(\S[t]/t^n)_i$ for $i \geq n$ as in Hesselholt--Madsen~\cite{HM97b} (see Remark~\ref{remark:decomposition_over_S}). Instead, we will only need a connectivity estimate for $\THH(\S[t]/t^n)_i$ for $i \geq n$, which we deduce from a calculation of the Hochschild homology groups of truncated polynomial rings over the integers. This computation was previously obtained by Guccione--Guccione--Redondo--Solotar--Villamayor~\cite{GGRSV91}. We will let $\Z\langle y\rangle$ denote the free divided power algebra which has generators $y^{[1]}, y^{[2]}, \ldots$ with $y = y^{[1]}$ and the relations that $y^{[i]}y^{[j]} = \binom{i+j}{i} y^{[i+j]}$ for every pair of positive integers $i$ and $j$. We have the following result:

\begin{proposition} \label{proposition:HH_truncated_dga}
For every integer $n \geq 1$, the Hochschild homology $\HH(\Z[t]/t^n)$ is equivalent to the $\E_1$-algebra given by the following differential graded algebra
\[
	(\Z[t]/t^n \otimes \Lambda(dt) \otimes \Z\langle y \rangle, \partial),
\]
with $|dt| = 1$ and $|y| = 2$, whose differential is determined by $\partial(y^{[i]}) = nt^{n-1}y^{[i-1]} dt$ and $\partial(dt) = 0$. 
\end{proposition}

\begin{proof}
The following square is a pushout of $\E_\infty$-algebras over $\Z$
\[
\begin{tikzcd}
	\HH(\Z[s]) \arrow{r}{s \mapsto t^n} \arrow{d}{s \mapsto 0} & \HH(\Z[t]) \arrow{d} \\
	\Z \arrow{r} & \HH(\Z[t]/t^n)
\end{tikzcd}
\]
since the functor $- \otimes \Z : \CAlg \to \CAlg_{\Z}$ preserves colimits and $\THH(\S[t]/t^n) \otimes \Z \simeq \HH(\Z[t]/t^n)$. Recall that the pushout of $\E_\infty$-algebras over $\Z$ above is calculated by the relative tensor product. The Hochschild homology $\HH(\Z[s])$ is $1$-truncated since $\HH_\ast(\Z[s]) \simeq \mathrm{H}_\ast(\Z[s] \otimes \Lambda(ds))$ by virtue of the Hochschild--Kostant--Rosenberg theorem, and the map $\HH_\ast(\Z[s]) \to \HH_\ast(\Z[t])$ is given by $s \mapsto t^n$ and $ds \mapsto nt^{n-1} dt$. The Hochschild homology $\HH(\Z[s])$ is equivalent to the $1$-truncation of the free $\E_1$-$\Z$-algebra on the generators $s$ and $ds$, so we conclude that $\HH(\Z[s]) \simeq \Z[s] \otimes \Lambda(ds)$. Additionally, it follows from the homology statement above that the map $\HH(\Z[s]) \to \HH(\Z[t])$ is given by $s \mapsto t^n$ and $ds \mapsto nt^{n-1}dt$ under these identifications since there is an equivalence
\[
	\Map_{\tau_{\leq 1} \Alg(\Sp)}(\HH(\Z[s]), \HH(\Z[t])) \simeq \Map_{\Alg(\Sp)}(\HH(\Z[s]), \HH(\Z[t])),
\]
where $\tau_{\leq 1} \Alg(\Sp)$ denotes the full subcategory of $\Alg(\Sp)$ spanned by those $\E_1$-rings which are $1$-truncated. In other words, we conclude that $\HH(\Z[t]/t^n)$ is given by the following pushout
\[
\begin{tikzcd}[column sep = large]
	\Z[s] \otimes \Lambda(ds) \arrow{d}{s \mapsto 0} \arrow{rr}{s \mapsto t^n,\,\, ds \mapsto nt^{n-1}dt} & & \Z[t] \otimes \Lambda(dt) \arrow{d} \\ 
	\Z \arrow{rr} & & \HH(\Z[t]/t^n)
\end{tikzcd}
\]
Note that the $\Z$ appearing in the lower left corner of the pushout square above is equivalent to the commutative differential graded algebra given by $\Z[s] \otimes \Lambda(ds) \otimes \Lambda(\varepsilon) \otimes \Z\langle y\rangle$
with $|ds|=|\varepsilon|=1$ and $|y|=2$, whose differential is determined by $\partial(\varepsilon) = s$ and $\partial(y^{[i]}) = y^{[i-1]} ds$. Since the ring homomorphism $\Z[s] \to \Z[t]$ given by $s \mapsto t^n$ is flat, we conclude that $\HH(\Z[t]/t^n)$ is given by the pushout of the following diagram of commutative differential graded algebras
\[
\begin{tikzcd}[column sep = large]
	\Z[s] \otimes \Lambda(ds) \arrow{r}{s \mapsto t^n,\,\, ds \mapsto nt^{n-1}dt} \arrow{d} & \Z[t] \otimes \Lambda(dt) \\
	\Z[s] \otimes \Lambda(ds) \otimes \Lambda(\varepsilon) \otimes \Z\langle y\rangle &
\end{tikzcd}
\]
which is equivalent to the commutative differential graded algebra given by $\Z[t]/t^n \otimes \Lambda(dt) \otimes \Z\langle y\rangle$ with $|dt|=1$ and $|y|=2$, whose differential is determined by $\partial(dt) = 0$ and $\partial(y^{[i]}) = nt^{n-1}y^{[i-1]}dt$, where we have used that $\varepsilon$ does not contribute to the pushout since $\Z[s] \to \Z[t]$ given by $s \mapsto t^n$ is injective and $\partial(\varepsilon) = s$. 
\end{proof}

There is an equivalence of spectra with $\T$-action
\[
	\HH(\Z[t]/t^n) \simeq \bigoplus_{i \geq 0} \HH(\Z[t]/t^n)_i,
\]
where $\HH(\Z[t]/t^n)_i \simeq \THH(\S[t]/t^n)_i \otimes \Z$. Using Proposition~\ref{proposition:HH_truncated_dga}, we may regard $\HH_\ast(\Z[t]/t^n)$ as a graded abelian group with $w(t) = w(dt) = 1$ and $w(y) = n$. As in Example~\ref{example:picture_gradings} above, we have the following picture:
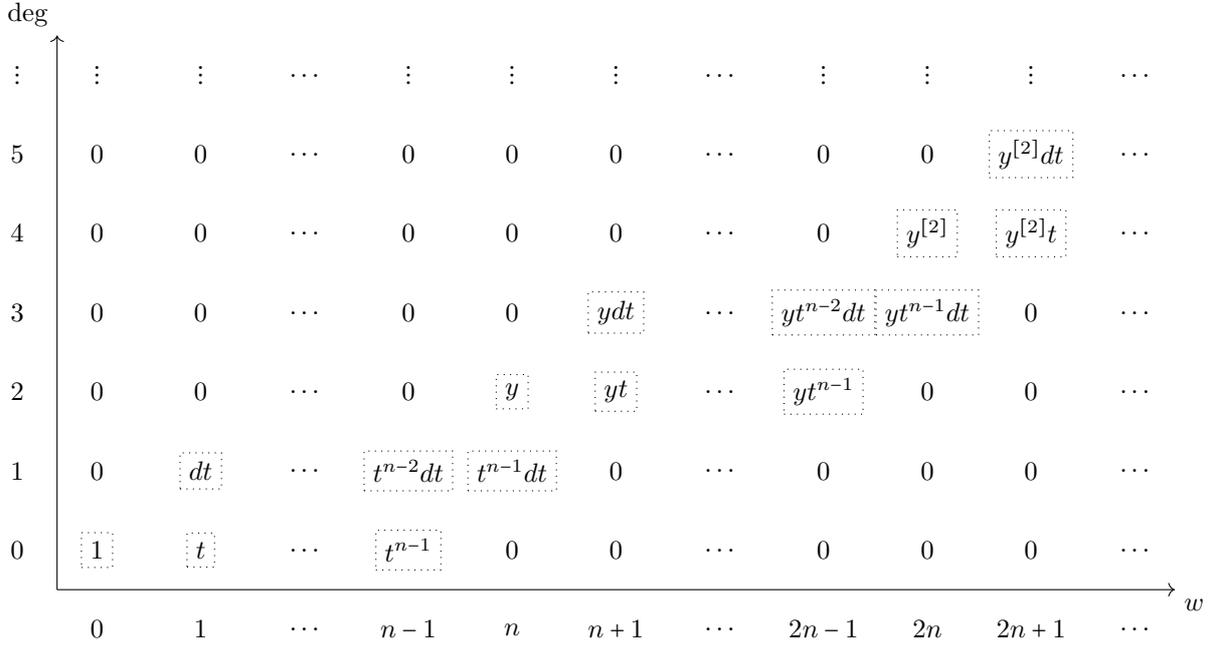
\[
\begin{tikzpicture}[scale=1.05]
	\node[draw, dotted] at (0, 0) {$1$};
	\node[draw, dotted] at (1.3, 0) {$t$};
	\node at (2.6, 0) {$\ldots$};
	\node[draw, dotted] at (3.9, 0) {$t^{n-1}$};
	\node at (5.2, 0) {$0$};
	\node at (6.5, 0) {$0$};
	\node at (7.8, 0) {$\ldots$};
	\node at (9.1, 0) {$0$};
	\node at (10.4, 0) {$0$};
	\node at (11.7, 0) {$0$};
	\node at (13, 0) {$\ldots$};

	\node at (0, 1) {$0$};
	\node[draw, dotted] at (1.3, 1) {$dt$};
	\node at (2.6, 1) {$\ldots$};
	\node[draw, dotted] at (3.9, 1) {$t^{n-2} dt$};
	\node[draw, dotted] at (5.2, 1) {$t^{n-1} dt$};
	\node at (6.5, 1) {$0$};
	\node at (7.8, 1) {$\ldots$};
	\node at (9.1, 1) {$0$};
	\node at (10.4, 1) {$0$};
	\node at (11.7, 1) {$0$};
	\node at (13, 1) {$\ldots$};

	\node at (0, 2) {$0$};
	\node at (1.3, 2) {$0$};
	\node at (2.6, 2) {$\ldots$};
	\node at (3.9, 2) {$0$};
	\node[draw, dotted] at (5.2, 2) {$y$};
	\node[draw, dotted] at (6.5, 2) {$y t$};
	\node at (7.8, 2) {$\ldots$};
	\node[draw, dotted] at (9.1, 2) {$yt^{n-1}$};
	\node at (10.4, 2) {$0$};
	\node at (11.7, 2) {$0$};
	\node at (13, 2) {$\ldots$};

	\node at (0, 3) {$0$};
	\node at (1.3, 3) {$0$};
	\node at (2.6, 3) {$\ldots$};
	\node at (3.9, 3) {$0$};
	\node at (5.2, 3) {$0$};
	\node[draw, dotted] at (6.5, 3) {$y dt$};
	\node at (7.8, 3) {$\ldots$};
	\node[draw, dotted] at (9.1, 3) {$yt^{n-2} dt$};
	\node[draw, dotted] at (10.4, 3) {$yt^{n-1} dt$};
	\node at (11.7, 3) {$0$};
	\node at (13, 3) {$\ldots$};

	\node at (0, 4) {$0$};
	\node at (1.3, 4) {$0$};
	\node at (2.6, 4) {$\ldots$};
	\node at (3.9, 4) {$0$};
	\node at (5.2, 4) {$0$};
	\node at (6.5, 4) {$0$};
	\node at (7.8, 4) {$\ldots$};
	\node at (9.1, 4) {$0$};
	\node[draw, dotted] at (10.4, 4) {$y^{[2]}$};
	\node[draw, dotted] at (11.7, 4) {$y^{[2]} t$};
	\node at (13, 4) {$\ldots$};

	\node at (0, 5) {$0$};
	\node at (1.3, 5) {$0$};
	\node at (2.6, 5) {$\ldots$};
	\node at (3.9, 5) {$0$};
	\node at (5.2, 5) {$0$};
	\node at (6.5, 5) {$0$};
	\node at (7.8, 5) {$\ldots$};
	\node at (9.1, 5) {$0$};
	\node at (10.4, 5) {$0$};
	\node[draw, dotted] at (11.7, 5) {$y^{[2]} dt$};
	\node at (13, 5) {$\ldots$};

	\node at (0, 6) {$\vdots$};
	\node at (1.3, 6) {$\vdots$};
	\node at (2.6, 6) {$\ldots$};
	\node at (3.9, 6) {$\vdots$};
	\node at (5.2, 6) {$\vdots$};
	\node at (6.5, 6) {$\vdots$};
	\node at (7.8, 6) {$\ldots$};
	\node at (9.1, 6) {$\vdots$};
	\node at (10.4, 6) {$\vdots$};
	\node at (11.7, 6) {$\vdots$};
	\node at (13, 6) {$\ldots$};

	\node at (0, -1) {$0$};
	\node at (1.3, -1) {$1$};
	\node at (2.6, -1) {$\ldots$};
	\node at (3.9, -1) {$n-1$};
	\node at (5.2, -1) {$n$};
	\node at (6.5, -1) {$n+1$};
	\node at (7.8, -1) {$\ldots$};
	\node at (9.1, -1) {$2n-1$};
	\node at (10.4, -1) {$2n$};
	\node at (11.7, -1) {$2n+1$};
	\node at (13, -1) {$\ldots$};

	\node at (-1, 0) {$0$};
	\node at (-1, 1) {$1$};
	\node at (-1, 2) {$2$};
	\node at (-1, 3) {$3$};
	\node at (-1, 4) {$4$};
	\node at (-1, 5) {$5$};
	\node at (-1, 6) {$\vdots$};

	\draw[->] (-0.5, -0.5) -- (13.5, -0.5) node[anchor=north west] {$w$};
	\draw[->] (-0.5, -0.5) -- (-0.5, 6.5) node[anchor=south east] {deg};
\end{tikzpicture}
\]

As a consequence, we obtain the following result from the diagram above, which was previously obtained by Hesselholt--Madsen~\cite[\textsection 7.3]{HM97} and Hesselholt~\cite[Lemma 3.1.6]{Hes96}. 

\begin{corollary} \label{corollary:coconnectivity}
For every $n \geq 1$ and every $i \geq 1$, the spectrum $\THH(\S[t]/t^n)_i$ is $2\ell$-connective, where $\ell = \lfloor \frac{i-1}{n} \rfloor$ denotes the largest integer less than $\frac{i-1}{n}$. 
\end{corollary}

\begin{proof}
It suffices to prove the assertion for $\HH(\Z[t]/t^n)_i$. Note that if $n$ divides $i$, then $\HH(\Z[t]/t^n)_i$ is concentrated in degree $2\ell + 1$ and $2\ell + 2$. If $n$ does not divide $i$, then $\HH(\Z[t]/t^n)_i$ is concentrated in degree $2\ell$ and $2\ell + 1$. We conclude that $\HH(\Z[t]/t^n)_i$ is $2\ell$-connective for every $i \geq 1$, which shows the desired statement by Hurewicz since $\THH(\S[t]/t^n)_i \otimes \Z \simeq \HH(\Z[t]/t^n)_i$. 
\end{proof}

\begin{remark} \label{remark:decomposition_over_S}
In this remark, we will identify the underlying spectrum of $\THH(\S[t]/t^n)$ using Proposition~\ref{proposition:HH_truncated_dga}. The homology of $\HH(\Z[t]/t^n)_{kn}$ is given by the homology of the complex
\[
	\cdots \to 0 \to \Z[2k] \xrightarrow{n} \Z[2k-1] \to 0 \to \cdots
\]
for every integer $k \geq 1$. It follows that there is an equivalence of spectra 
\[
	\THH(\S[t]/t^n)_{kn} \simeq \S/n[2k-1]
\]
since the Moore spectrum $\S/n$ is uniquely characterized by $\pi_\ast(\S/n \otimes \Z) \simeq \Z/n[0]$. Additionally, we recall that $\Sigma^\infty_+ S^1 \simeq \S \oplus \S[1]$. Using the picture following the proof of Proposition~\ref{proposition:HH_truncated_dga}, we conclude that there is an equivalence of spectra
\[
	\THH(\S[t]/t^n) \simeq \S \oplus \Big(\bigoplus_{i \geq 0} \Big(\bigoplus_{j = 1}^{n-1} (\S \oplus \S[1])[2i] \Big) \oplus \S/n[2i+1]\Big).
\]
In fact, it is possible to determine the underlying spectrum with $\T$-action of $\THH(\S[t]/t^n)$. Such an identification will be useful for computing the topological Hochschild homology of truncated polynomial rings over the integers or a perfectoid base ring.  
\end{remark}

\begin{example}   \label{example:cyclotomic_product}
In this example, we describe a cyclotomic structure on the product
\[
	\prod_{n \geq 1} X \otimes \Sigma^\infty_+(S^1/C_n),
\]
where $X$ denotes a cyclotomic spectrum whose underlying spectrum is bounded below as follows: Using~\cite[Lemma 2.11]{AN20}, we conclude that the canonical map of spectra with $\T$-action
\[
	\big( \prod_{n \geq 1} X \otimes \Sigma^\infty_+(S^1/C_n) \big)^{tC_p} \to \prod_{n \geq 1} (X \otimes \Sigma^\infty_+(S^1/C_n))^{tC_p}
\]
is an equivalence for every prime number $p$. As a consequence, we may regard the product above as a cyclotomic spectrum whose cyclotomic Frobenius is induced by the $\T$-equivariant map
\[
	X \otimes \Sigma^\infty_+(S^1/C_n) \xrightarrow{\id \otimes \psi_p} X \otimes \Sigma^\infty_+(S^1/C_{pn})^{tC_p} \xrightarrow{\varphi_p \otimes \id} X^{tC_p} \otimes \Sigma^\infty_+(S^1/C_{pn})^{tC_p} \xrightarrow{\ell} (X \otimes \Sigma^\infty_+(S^1/C_{pn}))^{tC_p}
\]
where $\ell$ denotes the map induced by the canonical lax symmetric monoidal structure of the Tate construction (see~\cite[Theorem I.3.1]{NS18}).
\end{example}

Finally, we obtain a convenient description of the cyclotomic structure of $\varprojlim (X \otimes \widetilde{\THH}(\S[t]/t^n))$ for every cyclotomic spectrum $X$ whose underlying spectrum is bounded below using the connectivity result obtained in Corollary~\ref{corollary:coconnectivity}. 

\begin{proposition} \label{proposition:cyclotomic_structure_limit}
There is an equivalence of cyclotomic spectra
\[
	\varprojlim (X \otimes \widetilde{\THH}(\S[t]/t^n)) \simeq \displaystyle\prod_{n \geq 1} X \otimes \Sigma^\infty_+(S^1/C_n)
\]
for every cyclotomic spectrum $X$ whose underlying spectrum is bounded below, where the cyclotomic structure on the product is described in Example~\ref{example:cyclotomic_product}. 
\end{proposition} 

\begin{proof}
For every integer $n \geq 2$, we let $\widetilde{\THH}(\S[t])_{<n}$ denote the cyclotomic spectrum defined by
\[
	\widetilde{\THH}(\S[t])_{<n} = \bigoplus_{i = 1}^{n-1} \Sigma^\infty_+(S^1/C_i),
\]
whose cyclotomic Frobenius is induced by $\Sigma^\infty_+(S^1/C_i) \to \Sigma^\infty_+(S^1/C_{pi})^{hC_p}$ for $pi \leq n-1$, and by $\Sigma^\infty_+(S^1/C_i) \to 0$ for $pi \geq n$. By construction, the canonical projection
\[
	\widetilde{\THH}(\S[t]/t^n) \to \widetilde{\THH}(\S[t])_{< n}
\]
determines a map of cyclotomic spectra such that the following diagram commutes
\[
\begin{tikzcd}
	\widetilde{\THH}(\S[t]/t^{n+1}) \arrow{r} \arrow{d} & \widetilde{\THH}(\S[t])_{<n+1} \arrow{d} \\
	\widetilde{\THH}(\S[t]/t^{n}) \arrow{r} & \widetilde{\THH}(\S[t])_{<n}
\end{tikzcd}
\]
where the left vertical map is induced by the canonical map of pointed monoids $\Pi_{n+1} \to \Pi_n$, and the right vertical map is given by the projection. We obtain a map of cyclotomic spectra
\begin{equation} \label{equation:map_between_limits}
	\varprojlim (X \otimes \widetilde{\THH}(\S[t]/t^n)) \to \varprojlim (X \otimes \widetilde{\THH}(\S[t])_{<n})
\end{equation}
for every cyclotomic spectrum $X$ whose underlying spectrum is bounded below, and we show that this map is an equivalence of cyclotomic spectra. First note that the forgetful functor $\CycSp \to \Sp^{\B\T}$ preserves both of the limits appearing in~(\ref{equation:map_between_limits}) by virtue of Corollary~\ref{corollary:coconnectivity} and the assumption that the underlying spectrum of $X$ is bounded below. Consequently, it suffices to show that the map of spectra with $\T$-action
\[
	\varprojlim (X \otimes \widetilde{\THH}(\S[t])_{<n}) \to \varprojlim (X \otimes \widetilde{\THH}(\S[t]/t^n))
\]
induced by the inclusion $\widetilde{\THH}(\S[t])_{<n} \hookrightarrow \widetilde{\THH}(\S[t]/t^n)$ is an equivalence. There is a commutative diagram of cofiber sequences of spectra with $\T$-action
\[
\begin{tikzcd}
	X \otimes \widetilde{\THH}(\S[t])_{<n+1} \arrow{r} \arrow{d} & X \otimes \widetilde{\THH}(\S[t]/t^{n+1}) \arrow{r} \arrow{d} & X \otimes \displaystyle\bigoplus_{i \geq n+1} \widetilde{\THH}(\S[t]/t^{n+1})_i \arrow{d} \\
	X \otimes \widetilde{\THH}(\S[t])_{<n} \arrow{r} & X \otimes \widetilde{\THH}(\S[t]/t^{n}) \arrow{r} & X \otimes \displaystyle\bigoplus_{i \geq n} \widetilde{\THH}(\S[t]/t^{n})_i
\end{tikzcd}
\]
where the right vertical map is induced by the composite
\[
	\widetilde{\THH}(\S[t]/t^{n+1})_i \to \widetilde{\THH}(\S[t]/t^{n})_i \hookrightarrow \widetilde{\THH}(\S[t]/t^n).
\]
Thus, it suffices to show that
\[
	\varprojlim_n \bigoplus_{i \geq n} (X \otimes \widetilde{\THH}(\S[t]/t^n)_i) \simeq 0,
\]
which now follows from Corollary~\ref{corollary:coconnectivity}. Indeed, we have that
\[
	\varprojlim_n \bigoplus_{i \geq n} (X \otimes \widetilde{\THH}(\S[t]/t^n)_i) \simeq \varprojlim_n \big( \varprojlim_m \bigoplus_{i \geq m} (X \otimes \widetilde{\THH}(\S[t]/t^n)_i) \big)
\]
since the diagonal $\N \to \N \times \N$ is an initial functor, and using Corollary~\ref{corollary:coconnectivity}, we conclude that the limit appearing in the parenthesis on the right hand side of the equivalence above vanishes. In conclusion, we have proved that the map appearing in~(\ref{equation:map_between_limits}) is an equivalence. To finish the proof, we show that the map of cyclotomic spectra
\[
	\displaystyle\prod_{n \geq 1} X \otimes \Sigma^\infty_+(S^1/C_n) \to \varprojlim (X \otimes \widetilde{\THH}(\S[t])_{<n}). 
\]
induced by the projection maps is an equivalence, where the cyclotomic structure on the product is defined in Example~\ref{example:cyclotomic_product}. This follows from the fact that if
\[
	\cdots \xrightarrow{\mathrm{proj}} A_1 \oplus A_2 \oplus A_3 \xrightarrow{\mathrm{proj}} A_1 \oplus A_2 \xrightarrow{\mathrm{proj}} A_1
\]
is a tower of abelian groups, then the limit is given by the product $\prod_{n \geq 1} A_n$.
\end{proof}

\subsection{Curves on K-theory} \label{subsection:curves}
As an application of Theorem~\ref{theorem:comparison_TR} and Proposition~\ref{proposition:cyclotomic_structure_limit}, we obtain the desired description of $\TR$ evaluated on a connective $\E_1$-ring $R$ in terms of the spectrum of curves on $\K(R)$ extending work of Hesselholt~\cite{Hes96} and Betley--Schlichtkrull~\cite{BS05}. We will begin by recalling the following notation (see the discussion following Lemma~\ref{lemma:pushout_pointed_monoids}).

\begin{notation} \label{notation:relative_theories}
Recall that for every integer $n \geq 1$, the following square is a pushout of $\E_\infty$-rings
\[
\begin{tikzcd}
	\S[t] \arrow{r}{t \mapsto t^n} \arrow{d}{t \mapsto 0} & \S[t] \arrow{d} \\
	\S \arrow{r} & \S[t]/t^n
\end{tikzcd}
\]
In particular, there is a map of $\E_\infty$-rings $\S[t]/t^n \to \S$ determined by the assignment $t \mapsto 0$. If $R$ is a connective $\E_1$-ring, then we  define the $\E_1$-ring $R[t]/t^n$ by $R[t]/t^n = R \otimes \S[t]/t^n$, and we have that $\pi_\ast(R[t]/t^n) \simeq (\pi_\ast R)[t]/t^n$. We obtain a map of connective $\E_1$-rings $R[t]/t^n \to R$ such that the kernel of the induced ring homomorphism
\[
	\pi_0(R[t]/t^n) \simeq (\pi_0 R)[t]/t^n \to \pi_0 R
\]
is given by the nilpotent ideal $(t)$. If $E : \Alg^{\mathrm{cn}} \to \Sp$ is a functor, then we will let $E(R[t]/t^n, (t))$ denote the fiber of the induced map of spectra $E(R[t]/t^n) \to E(R)$. 
\end{notation}

We recall the definition of the spectrum of curves on algebraic $\K$-theory following Hesselholt~\cite{Hes96}, which is based on a previous variant studied by Bloch~\cite{Blo77} in his work on the relationship between algebraic $\K$-theory and crystalline cohomology. 

\begin{definition}
The spectrum of curves on $\K$-theory is defined by
\[
	\mathrm{C}(R) = \varprojlim \Omega \K(R[t]/t^n, (t))
\]
for every connective $\E_1$-ring $R$. 
\end{definition}

A fundamental result of Hesselholt~\cite[Theorem 3.1.10]{Hes96} asserts that if $R$ is a discrete commutative $\Z/p^j$-algebra for some $j \geq 1$, then there is a natural equivalence of spectra $\TR(R) \simeq \mathrm{C}(R)$. As a consequence, Hesselholt~\cite[Theorem C]{Hes96} proves that the homotopy groups of the $p$-typical summand of $\mathrm{C}(A)$ is isomorphic to the de Rham--Witt complex $\mathrm{W}\Omega_A^{\ast}$ in the case where $A$ is a smooth algebra over a perfect field of characteristic $p$. In~\cite[Theorem 1.3]{BS05}, Betley--Schlichtkrull establish a variant of the result of Hesselholt for topological cyclic homology on discrete associative rings after profinite completion, where the inverse limit in the definition of the spectrum of curves on $\K$-theory is replaced by a limit over a diagram which additionally encodes the transfer maps $R[t]/t^m \to R[t]/t^{mn}$ determined by $t \mapsto t^n$. Our main result is the following: 

\begin{theorem} \label{theorem:TR_curves_TC}
There is a natural equivalence of spectra
\[
	\TR(X) \simeq \varprojlim \Omega \TC(X \otimes \widetilde{\THH}(\S[t]/t^n))
\]
for every cyclotomic spectrum $X$ whose underlying spectrum is bounded below. 
\end{theorem}

In the course of the proof of Theorem~\ref{theorem:TR_curves_TC}, we will need the following result:

\begin{lemma} \label{lemma:tate_map_equivalence}
If $X$ is a spectrum with $\T$-action, then the following $\T$-equivariant map of spectra
\[
\begin{tikzcd}
	X^{tC_p} \otimes \Sigma^\infty_+(S^1/C_n) \arrow{rr}{\id \otimes (\can \circ \psi_p)} & & X^{tC_p} \otimes \Sigma^\infty_+(S^1/C_{pn})^{tC_p} \arrow{r}{\ell} & (X \otimes \Sigma^\infty_+(S^1/C_{pn}))^{tC_p}
\end{tikzcd}
\]
is an equivalence for every prime $p$ and every integer $n \geq 1$, where $\ell$ denotes the map induced by the lax symmetric monoidal structure on the Tate construction $(-)^{tC_p}$.  
\end{lemma}

\begin{proof}
See Speirs~\cite[Lemma 8]{Spe20}. 
\end{proof}

\begin{proof}[Proof of Theorem~\ref{theorem:TR_curves_TC}]
There is a natural equivalence of spectra 
\[
	\TR(X) \simeq \map_{\CycSp}(\widetilde{\THH}(\S[t]), X)
\]
by virtue of Theorem~\ref{theorem:comparison_TR}, so we conclude that
\begin{align*}
	\TR(X) &\simeq \mathrm{Eq}\Big(
	\begin{tikzcd}[ampersand replacement = \&]
		\displaystyle\map_{\Sp^{\B\T}}\Big(\bigoplus_{n \geq 1} \Sigma^\infty_+(S^1/C_n), X\Big) \arrow[yshift=0.7ex]{r} \arrow[yshift=-0.7ex]{r} \& \displaystyle\prod_{p} \map_{\Sp^{\B\T}}\Big(\bigoplus_{n \geq 1} \Sigma^\infty_+(S^1/C_n), X^{tC_p}\Big)
	\end{tikzcd}
	\Big) \\
	&\simeq \mathrm{Eq}\Big(
	\begin{tikzcd}[ampersand replacement = \&]
		\displaystyle\prod_{n \geq 1} \map_{\Sp^{\B\T}}(\Sigma^\infty_+(S^1/C_n), X) \arrow[yshift=0.7ex]{r} \arrow[yshift=-0.7ex]{r} \& \displaystyle\prod_{p} \prod_{n \geq 1} \map_{\Sp^{\B\T}}(\Sigma^\infty_+(S^1/C_n), X^{tC_p})
	\end{tikzcd}
	\Big)
\end{align*}
by~\cite[Proposition II.1.5]{NS18}, where the top map is induced by the cyclotomic structure of $\widetilde{\THH}(\S[t])$, and the bottom map is induced by the cyclotomic structure of $X$. Note that $\Sigma^\infty_+(S^1/C_n)$ is dualizable regarded as a spectrum with $\T$-action for every $n \geq 1$. Indeed, we have that $\T$-equivariant Atiyah duality identifies its dual with $\Sigma^{\infty-1}_+(S^1/C_n)$, and the following $\T$-equivariant map
\[
\begin{tikzcd}
	X^{tC_p} \otimes \Sigma^{\infty-1}_+(S^1/C_n) \arrow{rr}{\id \otimes (\can \circ \psi_p)} & & X^{tC_p} \otimes \Sigma^{\infty-1}_+(S^1/C_{pn})^{tC_p} \arrow{r}{\ell} & (X \otimes \Sigma^{\infty-1}_+(S^1/C_{pn}))^{tC_p}
\end{tikzcd}
\]
is an equivalence for every prime $p$ and every integer $n \geq 1$ by virtue of Lemma~\ref{lemma:tate_map_equivalence} above. It follows that the second equalizer above is equivalent to
\[
	\mathrm{Eq}\Big(
	\begin{tikzcd}
		\displaystyle\prod_{n \geq 1} \map_{\Sp^{\B\T}}(\S, X \otimes \Sigma^{\infty-1}_+(S^1/C_n)) \arrow[yshift=0.7ex]{r} \arrow[yshift=-0.7ex]{r} & \displaystyle\prod_{p} \prod_{n \geq 1} \map_{\Sp^{\B\T}}(\S, (X \otimes \Sigma^{\infty-1}_+(S^1/C_{pn}))^{tC_p})
	\end{tikzcd}
	\Big),
\]
where the top map is induced by the trivial cyclotomic structure of the sphere spectrum $\S$, and the bottom map carries $f : \S \to X \otimes \Sigma^{\infty-1}_+(S^1/C_n)$ to the following $\T$-equivariant map
\[
	\S \xrightarrow{f} X \otimes \Sigma^{\infty-1}_+(S^1/C_n) \xrightarrow{\varphi_p} X^{tC_p} \otimes \Sigma^{\infty-1}_+(S^1/C_n) \xrightarrow{\simeq} (X \otimes \Sigma^{\infty-1}_+(S^1/C_{pn}))^{tC_p}
\]
As in Example~\ref{example:cyclotomic_product}, the canonical map of spectra with $\T$-action
\[
	\big( \prod_{n \geq 1} X \otimes \Sigma^{\infty-1}_+(S^1/C_{pn}) \big)^{tC_p} \to \prod_{n \geq 1} (X \otimes \Sigma^{\infty-1}_+(S^1/C_{pn}))^{tC_p}
\]
is an equivalence, so we conclude that the equalizer above is equivalent to
\begin{align*}
	&\mathrm{Eq}\Big(
	\begin{tikzcd}[ampersand replacement = \&]
		\displaystyle\map_{\Sp^{\B\T}}\Big(\S, \prod_{n \geq 1} X \otimes \Sigma^{\infty-1}_+(S^1/C_n) \Big) \arrow[yshift=0.7ex]{r} \arrow[yshift=-0.7ex]{r} \& \displaystyle\prod_{p} \map_{\Sp^{\B\T}}\Big(\S, \Big(\prod_{n \geq 1} X \otimes \Sigma^{\infty-1}_+(S^1/C_{pn})\Big)^{tC_p}\Big)
	\end{tikzcd}
	\Big) \\
	\simeq \, & \mathrm{Eq}\Big(
	\begin{tikzcd}[ampersand replacement = \&]
		\displaystyle\map_{\Sp^{\B\T}}\Big(\S, \prod_{n \geq 1} X \otimes \Sigma^{\infty-1}_+(S^1/C_n) \Big) \arrow[yshift=0.7ex]{r} \arrow[yshift=-0.7ex]{r} \& \displaystyle\prod_{p} \map_{\Sp^{\B\T}}\Big(\S, \Big(\prod_{n \geq 1} X \otimes \Sigma^{\infty-1}_+(S^1/C_{n})\Big)^{tC_p}\Big)
	\end{tikzcd}
	\Big),
\end{align*}
where the equivalence uses that $(X \otimes \Sigma^\infty_+(S^1/C_n))^{tC_p} \simeq 0$ provided that $p$ does not divide $n$. Finally, using Proposition~\ref{proposition:cyclotomic_structure_limit}, we conclude that $\TR(X)$ is equivalent to the following spectrum
\[
	\varprojlim \Omega \mathrm{Eq}\Big(
	\begin{tikzcd}
		\displaystyle\map_{\Sp^{\B\T}}\big(\S, X \otimes \widetilde{\THH}(\S[t]/t^n) \big) \arrow[yshift=0.7ex]{r} \arrow[yshift=-0.7ex]{r} & \displaystyle\prod_{p} \map_{\Sp^{\B\T}}\big(\S, (X \otimes \widetilde{\THH}(\S[t]/t^n))^{tC_p}\big)
	\end{tikzcd}
	\Big),
\]
where the top map is induced by the trivial cyclotomic structure of the sphere spectrum $\S$, and the bottom map is induced by the cyclotomic structure of $X \otimes \widetilde{\THH}(\S[t]/t^n)$. We conclude that
\[
	\TR(X) \simeq \varprojlim \Omega \map_{\CycSp}(\S, X \otimes \widetilde{\THH}(\S[t]/t^n)) \simeq \varprojlim \Omega \TC(X \otimes \widetilde{\THH}(\S[t]/t^n))
\]
which shows the desired statement. 
\end{proof}

Using Theorem~\ref{theorem:TR_curves_TC}, we obtain the following result: 

\begin{corollary} \label{corollary:TR_curves}
If $R$ is a connective $\E_1$-ring, then there is a natural equivalence of spectra
\[
	\TR(R) \simeq \varprojlim \Omega \TC(R[t]/t^n, (t)).
\]
In particular, there is a natural equivalence of spectra $\TR(R) \simeq \mathrm{C}(R)$. 
\end{corollary}

\begin{proof}
If $R$ is a connective $\E_1$-ring, then there is an equivalence of cyclotomic spectra
\[
	\THH(R) \otimes \widetilde{\THH}(\S[t]/t^n) \simeq \widetilde{\THH}(R[t]/t^n)
\]
since $\THH$ is a symmetric monoidal functor, so the first assertion follows readily from Theorem~\ref{theorem:TR_curves_TC}. Furthermore, there is a natural equivalence
\[
	\varprojlim \Omega \TC(R[t]/t^n, (t)) \simeq \varprojlim \Omega \K(R[t]/t^n, (t))
\]
by virtue of the Dundas--Goodwillie--McCarthy theorem~\cite{BGM12}, which shows the desired statement. 
\end{proof}

\begin{remark} \label{remark:frobenius_TR_in_terms_of_curves}
In Corollary~\ref{corollary:TR_curves}, the Frobenius endomorphism of $\TR(R)$ can be expressed in terms of suitable transfer maps on the spectrum of curves on $\K$-theory. More precisely, the $k$th Frobenius endomorphism of $\TR(R)$ is induced by the transfer map $K(R[t]/t^n) \to K(R[t]/t^{kn})$ determined by the assignment $t \mapsto t^k$. In~\cite{BS05}, Betley--Schlichtkrull work with a refined version of the spectrum of curves on $\K$-theory which additionally takes these transfer maps into account. Consequently, they obtain a formula for $\TC$ evaluated on a discrete ring in terms of this variant of the spectrum of curves on $\K$-theory after profinite completion (cf.~\cite[Theorem 1.3]{BS05}).
\end{remark}

\bibliographystyle{abbrv}
\bibliography{TR} 

\end{document}